\documentclass{amsart}

\usepackage{hyperref}
\usepackage{graphicx} 
\usepackage[utf8]{inputenc}
\usepackage[T1]{fontenc}
\usepackage{amsmath}
\usepackage{comment}
\usepackage{amsfonts}
\usepackage{amssymb}
\usepackage{amsthm}
\usepackage{mathtools}      
\usepackage{mathabx}        
\usepackage{mathabx}
\usepackage{tikz}
\usepackage{tikz-3dplot}
\usepackage{tikz-cd}
\usetikzlibrary{patterns}
\usetikzlibrary{patterns.meta}
\usepackage{etoolbox}
\usepackage{pgfplots}
\pgfplotsset{compat=1.15}
\usepackage{mathrsfs}
\usetikzlibrary{arrows}
\usetikzlibrary{arrows.meta}

\tikzset{
  background line/.style={gray!65,thin},
  figure line/.style={blue,semithick},
  figure arrow/.style={
    figure line,
    -{Straight Barb[length=2.5mm,width=2mm]}
  }
}

\usepackage{multicol}

\usepackage[colorinlistoftodos]{todonotes}
\usepackage{marginnote}

\hypersetup{bookmarksdepth  =  3} 

\theoremstyle{plain}
\newtheorem{thm}{Theorem}[section]
\newtheorem{theorem}[thm]{Theorem}
\newtheorem{lem}[thm]{Lemma}
\newtheorem{lemma}[thm]{Lemma}
\newtheorem{prop}[thm]{Proposition}
\newtheorem{proposition}[thm]{Proposition}

\newtheorem{corollary}[thm]{Corollary}

\theoremstyle{definition}
\newtheorem{defn}[thm]{Definition}
\newtheorem{defnprop}[thm]{Definition/Proposition}
\newtheorem{definition}[thm]{Definition}

\newtheorem{example}[thm]{Example}

\theoremstyle{remark}
\newtheorem{rem}[thm]{Remark}
\newtheorem{remark}[thm]{Remark}

\theoremstyle{plain}

\newtheorem{mainthm}{Theorem}

\newtheorem{maincor}[mainthm]{Corollary}

\AtBeginEnvironment{rem}{
  \pushQED{\qed}
  
}
\AtEndEnvironment{rem}{\popQED}

\AtBeginEnvironment{remark}{
  \pushQED{\qed}
  
}
\AtEndEnvironment{remark}{\popQED}

\AtBeginEnvironment{rems}{
  \pushQED{\qed}
  
}
\AtEndEnvironment{rems}{\popQED}

\newcommand{\R}{\mathbb{R}}
\newcommand{\N}{\mathbb{N}}
\newcommand{\C}{\mathbb{C}}
\DeclareMathOperator{\Leb}{Leb}
\newcommand{\X}{\mathbb{X}}

\newcommand{\HH}{\mathcal{H}}
\newcommand{\Id}{\mathrm{Id}}
\newcommand{\1}{\mathmybb{1}}

\newcommand{\partialvis}{\partial_{\mathrm{vis}}}
\newcommand{\Diag}{\operatorname{Diag}}

\DeclareMathOperator{\Tr}{Tr}
\DeclareMathOperator{\Ker}{Ker}
\DeclareMathOperator{\PSL}{PSL}
\DeclareMathOperator{\SL}{SL}
\DeclareMathOperator{\Lie}{Lie}
\DeclareMathOperator{\GL}{GL}

\DeclareMathOperator{\Gr}{Gr}

\DeclareMathOperator{\SO}{SO}
\DeclareMathOperator{\Sp}{Sp}

\DeclareMathOperator{\Lag}{Lag}
\DeclareMathOperator{\rad}{rad}

\newcommand{\Aff}{\Lin_0}

\DeclareMathOperator{\Lin}{Span}
\DeclareMathOperator{\ver}{vert}
\DeclareMathOperator{\stable}{stable}
\DeclareMathOperator{\fixed}{fixed}
\DeclareMathOperator{\flipped}{flipped}
\DeclareMathOperator{\Span}{Span}
\DeclareMathOperator{\Cone}{Cone}
\DeclareMathOperator{\Bis}{Bis}

\newcommand{\abs}[1]{\left\lvert {#1} \right\rvert}
\newcommand{\Abs}[1]{\left\lVert {#1} \right\rVert}

\DeclareMathAlphabet{\mathmybb}{U}{bbold}{m}{n}

\newcommand{\ie}{i.e.\ }

\title[IPVT for Finsler metrics on symmetric spaces]{Ideal Poisson--Voronoi tessellations for Finsler metrics on symmetric spaces}

\author{Colin Davalo}
\address{Dipartimento di Matematica ``Giuseppe Peano'',
  Università degli Studi di Torino,
  Via Carlo Alberto 10,
  10123 Torino, Italy}
\email{colinmarcthierry.davalo@unito.it}

\author{J. Maxwell Riestenberg}
\address{Max Planck Institute for Mathematics in the Sciences,
  Inselstraße 22,
  04103 Leipzig, Germany}
\email{max.riestenberg@mis.mpg.de}

\author{Anna Wienhard}
\address{Max Planck Institute for Mathematics in the Sciences,
  Inselstraße 22,
  04103 Leipzig, Germany}
\email{anna.wienhard@mis.mpg.de}

\begin{document}

\begin{abstract}
Ideal Poisson--Voronoi tessellations are random tessellations which arise as limits of Voronoi diagrams for Poisson point processes of low intensity. 
For the Riemannian metric on a symmetric space of higher rank, Fr{\k{a}}czyk-Mellick-Wilkens \cite{FMW23poissonvoronoi} proved that the corresponding ideal Poisson Voronoi tessellation is almost surely completely connected, with an unbounded side between any two cells, marking a very different geometric behavior than the ideal Poisson--Voronoi tessellation of hyperbolic space investigated in \cite{DAchilleCurienEnriquezLyonsUnel2026ipvthyperbolic}. 
In this paper we introduce and study ideal Poisson--Voronoi tessellations for Finsler metrics on higher rank symmetric spaces. 
We show that ideal Poisson--Voronoi tessellations with respect to polyhedral Finsler metrics can exhibit either behavior seen in these previous works: generically they resemble the hyperbolic case, where almost surely all cells have bounded sides.

We further introduce the ideal Selberg--Poisson--Voronoi tessellation of the space of positive definite symmetric matrices, which has particularly nice properties. 

\end{abstract}

\maketitle
\tableofcontents

\section{Introduction}

Given a metric space $(X,d)$, a natural way to define a tiling  of $X$, given a discrete collection of points $\Pi\subset X$, is the \emph{Voronoi} tessellation. 
For each point $p\in \Pi$ its Voronoi cell $\mathcal{C}_p$ consists of all points in $X$ whose distance to $p$ is less than or equal than to any other point in the collection of points $\Pi$: 
\begin{equation}\label{eqn: voronoi tile}
    \mathcal{C}_p=\lbrace x\in X \mid \forall q\in \Pi, \,d(x,p)\leq d(x,q)\rbrace.
\end{equation}
Of course, the Voronoi tessellation depends not only on the collection of points, but also heavily on the distance function on $X$.

A \emph{Poisson--Voronoi} tessellation is the natural random tiling one obtains when the collection of points $\Pi$ is a random set of points following a \emph{Poisson} law with a given intensity measure. 
On a Riemannian manifold $M$, a natural measure is given by the Riemannian volume measure $\Leb$, and a Poisson--Voronoi tessellation with respect to this measure is often called the Poisson--Voronoi tessellation of $M$. The scaling of the volume measure affects the intensity of the corresponding Poisson process. 

In \cite{DAchilleCurienEnriquezLyonsUnel2026ipvthyperbolic}, D'Achille--Curien-Enriquez--Lyons--Ünel considered Poisson--Voronoi tessellations of hyperbolic space $\mathbb{H}^n$ of dimension $n$ and introduced the \emph{ideal Poisson--Voronoi tessellation}.
This is the random tiling of $\mathbb{H}^n$ that arises as a limit of the Poisson--Voronoi tessellations of $\mathbb{H}^n$ when the intensity goes to zero, see Figure \ref{fig:TessellationsH2}. 
It is quite remarkable that in hyperbolic space - in contrast to the Euclidean situation - there is a well defined limit, which gives a random tiling of $\mathbb{H}^n$ that is a universal object. 
The authors of \cite{DAchilleCurienEnriquezLyonsUnel2026ipvthyperbolic} further proved that the ideal Poisson--Voronoi tessellation of $\mathbb{H}^n$ has nice geometric  properties: 
\begin{enumerate}
    \item Almost surely, each tile is unbounded with an infinite number of bounded faces. In particular, almost surely, the intersection of two cells of the tessellation is bounded. 
    \item Almost surely, each cell has a unique limit point in the boundary of $\mathbb{H}^n$. 
\end{enumerate}

\begin{figure}[htbp!]
    \centering
    \includegraphics[scale=1]{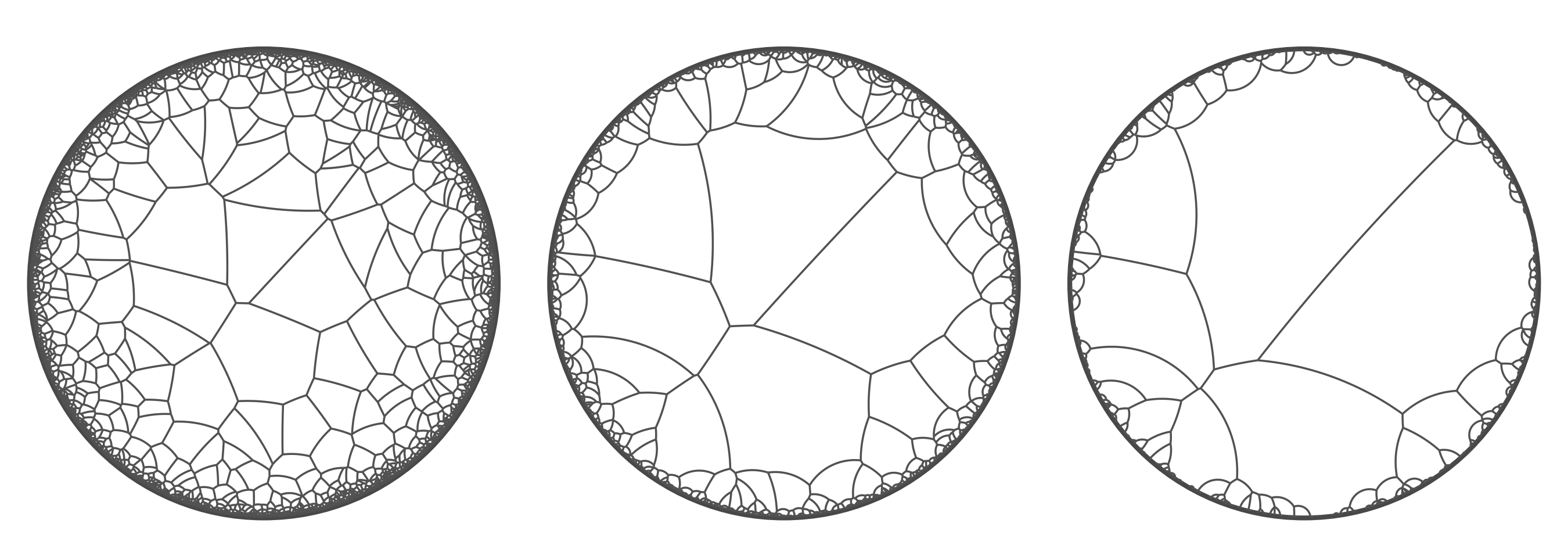}
    \caption{Two realisations of Poisson--Voronoi tessellations with intensity $3\Leb$, $0.5\Leb$ and a realisation of an ideal Poisson--Voronoi tessellation. These three realisations are coupled to emphasize the convergence, in the way described in \cite{DAchilleCurienEnriquezLyonsUnel2026ipvthyperbolic}.}
    \label{fig:TessellationsH2}
\end{figure}

This random tessellation provides information on the geometry of $\mathbb{H}^n$. 
Budzinski--Curien--Petri used this idea to show the existence of a gap between the Cheeger constant of the hyperbolic plane, which is equal to $1$, and that of closed hyperbolic surfaces of large genus \cite{BudzinskiCurienPetri2025cheeger}, due to the fact that the average length of the sides of the tessellation in a given set is equal to $\frac{2}{\pi}$ per unit of area.
Earlier, percolation with respect to the Poisson--Voronoi tessellation on the hyperbolic plane was introduced and studied by Benjamini--Schramm \cite{BenjaminiSchramm2001percolationsinH2}.
Low intensity limits were studied recently in \cite{Bhupatiraju2019LowIntensity}. 

\medskip

In a recent breakthrough Fr{\k{a}}czyk--Mellick--Wilkens \cite{FMW23poissonvoronoi} introduced the Poisson--Voronoi tessellation of higher rank symmetric spaces and products of trees. 
They proved that the properties of the ideal Poisson--Voronoi tessellation with respect to the Riemannian metric on a higher rank symmetric space are quite different from those of the real hyperbolic case. 
In particular, in the higher rank situation, almost surely the intersection of two cells is unbounded. 
This is a remarkable property which is difficult to picture. 
Fr{\k{a}}czyk--Mellick--Wilkens used this to prove that higher rank lattices have fixed price \cite{FMW23poissonvoronoi}, settling a long-standing conjecture in measurable group theory \cite{Gaboriau_cost}. 

\medskip

In \cite{FMW23poissonvoronoi} the authors also introduced a different perspective on the ideal Poisson--Voronoi tessellation of a symmetric space. 
They do not construct the tessellation as a limit of Poisson--Voronoi tessellations of low intensity, but as a tessellation associated to a Poisson point process on the visual  boundary of the symmetric space, or more precisely on an $\mathbb{R}$-bundle over it.

The intuition of this limit comes from examining the definition of the Voronoi cell in (\ref{eqn: voronoi tile}). 
In the low-intensity limit, the points appearing in $\Pi$ tend to leave every compact set. 
However, differences of distance functions $d(\cdot,p) - d(\cdot,q)$ can still limit to well-defined functions on $X$ as $p,q$ diverge. 
Such limit functions are called \emph{horofunctions}.

To make this picture more precise, observe that the tiles (\ref{eqn: voronoi tile}) are unchanged when each distance function $d(\cdot,p)$ is shifted by the same constant $t \in \R$. 
A function of the form $d(\cdot,p)-t$ is called \emph{distance-like}. 
The closure of the space of distance-like functions is the space of \emph{horofunctions}, denoted $\mathcal{H}(X)$. 
It is an $\R$-bundle over the \emph{horoboundary} $\partial X$, defined to be horofunctions modulo constants. 
When $X$ is proper, $\overline{X} = X \cup \partial X$ is the \emph{horofunction compactification} of $X$, the closure of the image of the embedding of $X$, via distance functions, into its space of continuous functions  modulo constants.

If $\Pi$ is a random set of points following a Poisson law with a given intensity measure $\mu$ on $\mathcal{H}(X)$, one gets an ideal Poisson--Voronoi tessellations of $X$ with respect to the measure $\mu$. 
For each $t \in \R$, there is an embedding 
    \[ \iota_t(x) = d(\cdot,x) -t , \quad \iota_t \colon X \to C(X) .\]
Suitably renormalized limits of the form $(\iota_{t})_\ast (\lambda \mu)$ (for $t$ large and $\lambda$ small) converge to a measure on $\mathcal{H}(X)$. 
In \cite{FMW23poissonvoronoi} the authors consider the Riemannian distance function $d$ on a symmetric space $\X = G/K$ of higher rank and the Riemannian volume measure $\mu$.
Here the space of horofunctions is more traditionally called the space of \emph{Busemann functions} and forms an $\R$-bundle over the visual boundary $\partial_{vis}\X$. 
In higher rank, the visual boundary and the space of Busemann functions have infinitely many $G$-orbits. 
A key point in \cite{FMW23poissonvoronoi} is computing the precise $G$-orbit where this low intensity accumulates.  

\medskip

An irreducible Riemannian symmetric space has a unique (up to scaling) invariant Riemannian distance function, but in higher rank it carries a wealth of invariant Finsler metrics. 
Of particular interest are polyhedral Finsler metrics, which are determined completely by a finite sided polyhedron in a maximal flat of the symmetric space $\X = G/K$.
In this article, we generalize the construction of ideal Poisson--Voronoi tessellations from \cite{FMW23poissonvoronoi} for polyhedral Finsler metrics on Riemannian symmetric spaces.
The geometric features of these random tessellations depend heavily on the precise Finsler metric. 
In particular, on every symmetric space $\X$ of higher rank, there are polyhedral Finsler metrics whose ideal Poisson--Voronoi tessellation exhibits the unbounded intersection phenomenon of the ideal Riemannian Poisson--Voronoi tessellation, and there are polyhedral Finsler metrics for which the cells of the ideal Poisson--Voronoi tessellations have geometric properties similar to the cells of the ideal Poisson--Voronoi tessellations of hyperbolic space $\mathbb{H}^n$. 
For the symmetric space of positive definite symmetric matrices we further define the \emph{ideal Selberg--Poisson--Voronoi tessellation} and describe its remarkable geometric properties. 

\medskip

We now describe the results of the article in more detail, describing first the space of horofunctions, and then the construction of measures on the space of horofunction with respect to which the ideal Poisson--Voronoi tessellations are constructed. We discuss the geometric properties of the ideal Poisson--Voronoi tessellations, as well as the special properties of the ideal Selberg--Poisson--Voronoi tessellation of the space of positive definite symmetric matrices.

\subsection{Polyhedral Finsler norms and their spaces of horofunctions}
An invariant Finsler metric on a symmetric space $\X = G/K$ is uniquely determined by a Weyl group invariant convex set in a Cartan subspace $\mathfrak{a}$ of $\X$ which contains the origin in its interior. 
We call an invariant Finsler metric \emph{polyhedral} if this convex set is a finite sided polyhedron $B^N$ and denote the corresponding distance function on $\X$ by $d_N$. 
As described above we can use $d_N$ to embed $\X$ into the space of continuous functions  $C(\X)$, defining for $t\in \R$ the map $\iota_t^N \colon \X \rightarrow C(\X)$ by $y \mapsto d_N( \cdot, y) - t$.
We emphasize that the Finsler metrics we consider are not necessarily symmetric, so the order of the arguments in $d_N$ matters.
The complement of $\bigcup_{t\in \R}\iota_t^N(X)$ in the closure 
of the space of Lipschitz functions on $X$ of the form $d_N(\cdot, y)-t$ is the space of horofunctions $\mathcal{H}^N(\X)$. The space of horofunctions  is not compact. To obtain a compact space, one usually considers the quotient space of $C(\X)/\mathbb{R}$ modulo the constant functions. The closure of $\iota_0^N(\X)$ in this quotient is the horofunction compactification $\overline{\X}^N$ of $X$, and the complement of $\iota_0^N(\X)$ is its horofunction boundary $\partial_N \X$. 
The space of horofunctions $\mathcal{H}^N(\X)$ is an $\R$-bundle over $\partial_N \X$. 

\medskip

The horofunction compactification of metric spaces with respect to Finsler norms have been studied in \cite{Walsh2007normed,Schilling2013Horofunction, JS2024polyhedral}. For symmetric spaces, it has been studied in \cite{FriedlandFreitas2004pMetrics,FriedlandFreitas2004Siegel} and then in more generality in \cite{ Schilling2013Horofunction, KapovichLeeb2018finsler,haettel2018horofunctioncompactificationssymmetricspaces}.
Building on this previous work, and inspired by the description of generalized Satake compactifications by Guichard-Kassel-Wienhard \cite{GuichardKasselWienhard}, our first result is a description of the space of horofunctions and its $G$-orbits, which we just summarize here. 
For a precise Lie theoretic description we refer to Proposition \ref{prop:Description space Horofunctions} and Proposition \ref{prop:Space horofunctions}.

\begin{mainthm}
    The space of horofunctions $\mathcal{H}^N(\X)$ is a disjoint union of finitely many $G$-orbits $\mathcal{O}_F$, which are indexed by the faces of the polyhedron $B^N \subset \mathfrak{a}$, considered up to the action of the Weyl group:
    $$\mathcal{H}^N(\X)=\bigcup_{[F]\in \text{Faces}(B^N)}\mathcal{O}_F. $$
    Each orbit fibers over a flag variety associated to $G$ with fiber being a (smaller) symmetric space of rank equal to the codimension of $F$. 
\end{mainthm}

We work this out for 3 examples of polyhedral Finsler metrics on $\mathbb{H}^2\times \mathbb{H}^2$ in Section \ref{subsec:H2H2}, as well as for a particular metric $d_{\omega_1}$ (related to the spectral radius) on the symmetric space $\X=\SL(n,\R)/\SO(n,\R)$ associated to $G=\SL(n,\R)$ in Section \ref{subsec:Horofunctions for SLn}.

\medskip

\subsection{Measures on the space of horofunctions}

In order to define ideal Poisson--Voronoi tessellations we need a measure on the space of horofunctions $\mathcal{H}^n(\X)$. 
The measures which are of importance to us are measures which arise as a limit of a family of push-forward measures of the Riemannian volume measure by the embedding $\iota^N: \X \rightarrow C(\X)$. 
More precisely, as in \cite{FMW23poissonvoronoi}, we consider the shifted embeddings :
$$\iota_t^N: \X \rightarrow C(\X), \quad y \mapsto \iota_t^N(y) := d_N(\cdot,y) - t,$$
and the family of measures $\mu_t:= \frac{1}{m_t}{\iota_t}_{*}(\Leb)$, where $\Leb$ is the Riemannian volume measure on $\X$. The re-normalization factor takes into account the Riemannian volume of growing balls of the Finsler metric $d_N$ on $\X$. 

For the Riemannian metric, \cite{FMW23poissonvoronoi} prove that this family of measures has a unique limit measure $\mu$, which agrees with the unique $G$-invariant measure on the $G$-orbit of $\mathcal{H}^N(\X)$ corresponding to the weighted barycenter of a Weyl chamber. 

For polyhedral Finsler norms 
we prove the following (Theorem \ref{thm: Convergence measure}):

\begin{mainthm}
\label{thm_intro:measure}
The family of measures $\frac{1}{m_t}(\iota_t)_*\Leb$ converges to a measure $\mu_N$ which is the unique non-zero measure (up to scalar) that is $G$-invariant and supported on the orbit $\mathcal{O}_{F_\rho}\subset \mathcal{H}^N(\X)$.
 \end{mainthm} 
Here $F_\rho$ is the face of the polyhedron $B^N \subset \mathfrak{a}$ on which the \emph{restricted Weyl vector} 
\begin{equation}\label{eqn: weyl vector}
    \rho \coloneqq \frac{1}{2}\sum_{\alpha \in \Sigma^+} \alpha \dim \mathfrak{g}_\alpha,
\end{equation}
the half sum of all positive roots $\Sigma^+$ in $\mathfrak{a}^*$, is maximized. 
The restricted Weyl vector is the same weighted barycenter mentioned above.
This is a consequence of the fact that the directions in which the points can be seen in a Poisson point process of low intensity are close to the direction of maximal volume growth.

Note that for invariant Finsler metrics, using the Lie theoretic description of the $G$-orbits in the space of horofunctions, one can classify which $G$-orbits $\mathcal{O}_F \subset \mathcal{H}^N(X)$ admit a $G$-invariant measure (see Proposition \ref{prop: Existence measure}), and there are measures which do not arise as limits of the construction in Theorem~\ref{thm_intro:measure}.

\subsection{The proper-on-bisector property}\label{subsec:properonbisector}

A key issue in understanding the geometric properties of the ideal Poisson--Voronoi tessellations of polyhedral Finsler metrics is to understand the bisectors of two generic horofunctions $h_1, h_2$ in an orbit $\mathcal{O}_F\subset \mathcal{H}^N(X)$:
$$ \Bis(h_1,h_2) \coloneqq \left\{ x \in X\, |\,h_1(x) = h_2(x) \right\}. $$ 
The faces of the cells of the ideal Poisson--Voronoi tessellations are subsets of these bisectors, and so the shape of the bisectors and the behavior of the horofunctions on $\Bis(h_1,h_2)$ governs their geometric properties. 

One difficulty in higher rank is that bisectors might not have nice properties. For example, the bisectors for horofunctions restricted to the Cartan subspace $\mathfrak{a}$ can have non-empty interior. However, for the symmetric space we prove the following (Theorem \ref{thm:Empty Interior}):
\begin{mainthm}
For any polyhedral Finsler norm and any two horofunctions $h_1, h_2\in \mathcal{O}_F\subset \mathcal{H}^N(X)$ in generic position, 
the bisector $\Bis(h_1,h_2) \subset X$ has empty interior. 
\end{mainthm}

The different behaviors of the ideal Poisson--Voronoi tessellations of different polyhedral Finsler metrics $d_N$ stem from different behavior of the horofunctions $h_1, h_2 \in \mathcal{O}_F$ on these bisectors, namely whether they are proper and bounded from below or not. 
If the horofunctions are proper and bounded from below on the bisector for any generic pair, we say that $\mathcal{O}_F$ has the \emph{proper-on-bisector property}. 

We give a characterization of which orbits 
$\mathcal{O}_F$ in the space of horofunctions have the proper-on-bisector property, depending purely on the shape of the dual face $F^\vee\subset \mathfrak{a}^*$ associated to $F$.
\begin{mainthm}\label{thm_intro:bounded_bisector}
    An orbit $\mathcal{O}_F$ has the proper-on-bisector property if and only if :
    \begin{equation}
    \Cone(F^\vee \cup w_0\cdot F^\vee) =\mathfrak{a}^* ,\end{equation} 
    Here $\Cone(S)$ denotes the smallest convex cone in $\mathfrak{a}^\ast$ containing the set $S$, and $w_0$ is the longest word of the Weyl group.

     The special orbit $\mathcal{O}_{F_\rho}$ has the proper-on-bisector property if and only if :
     $$w_0\cdot \Lin_0(F_\rho)\cap \Lin_0(F_\rho)=\lbrace 0\rbrace.$$ 
Here $\Lin_0(F_\rho)$ is the linear part of the smallest affine subspace containing $F_\rho$. 
In particular if $F_\rho$ is a vertex, then $\mathcal{O}_{F_\rho}$ has the proper-on-bisector property.
\end{mainthm}

See Proposition \ref{prop: proper bisectors} and Proposition \ref{prop:Caracterisation proper bisectors Frho}.
When the opposition involution is trivial, the norm $N$ is necessarily symmetric and the statement simplifies considerably, see Corollary \ref{cor: proper bisectors when opposition involution is trivial}.

\begin{maincor}
    When the opposition involution is trivial, the orbit $\mathcal{O}_{F_\rho}$ has the proper-on-bisector property if and only if $F_\rho$ is a vertex. 
\end{maincor}

For simple Lie algebras the opposition involution is non-trivial precisely when the restricted root system has type $A_n (n \ge 2)$, $D_{2n+1} (n \ge 1)$, or $E_6$.  
When no simple factor of these types appears, the opposition involution is trivial.

\subsection{Ideal Poisson--Voronoi tessellations}
Now we have all the ingredients to construct the ideal Poisson--Voronoi tessellations for polyhedral Finsler metrics. 
Given a polyhedral Finsler metric $d_N$ we have the space of horofunctions $\mathcal{H}^N(\X)$, as well as the $G$-invariant measure $\mu_N$ on the orbit $\mathcal{O}_{F_\rho} \subset \mathcal{H}^N(\X)$. Thus we can consider a random set of points $\Pi$ following a Poisson law with respect to $\mu_N$ and consider the corresponding ideal Poisson--Voronoi tessellation, whose cells are given by 
$$
C_h = \{ x\in \X \, |\, \forall{h'\in \Pi},\, h(x) \leq h'(x)\}.$$

We prove the following (see Theorem \ref{thm:ProperBissectro dichotomy}) :
\begin{mainthm}
\label{thm_intro:PorperBissectrors consequences}
    Let $(C_p^{\Pi})_{p\in \Pi}$ be the ideal Poisson--Voronoi tessellation associated with a polyhedral Finsler metric on a symmetric space $\X$. We have the following dichotomy: 
    \begin{enumerate}
        \item 
    If $\mathcal{O}_{F_\rho}$ has the proper-on-bisector property, then almost surely the intersection of two cells is bounded. 

    \item If $\mathcal{O}_{F_\rho}$ does not have the proper-on-bisector property, then almost surely the intersection of two cells is unbounded. 
    \end{enumerate}
\end{mainthm}

The characterization of the proper-on-bisector property in Theorem~\ref{thm_intro:bounded_bisector} allows one to easily determine if a polyhedral Finsler metric $d^N$ has the proper-on-bisector property, by analyzing the polyhedron $B^N$. 

\medskip

In $\mathbb{H}^2\times \mathbb{H}^2$ we describe the implications of Theorem \ref{thm_intro:PorperBissectrors consequences} for examples of polyhedral norms, see Section \ref{subsec:ExamplesV2}:
\begin{itemize}
    \item[-] For the distance $d_\infty((x_1,x_2),(y_1,y_2))=\max(d(x_1,y_1),d(x_2,y_2))$, the orbit $\mathcal{O}_{F_\rho}$ has the proper-on-bisector property.
    \item[-] For the distance $d_{a,b},(y_1,y_2))=ad(x_1,y_1)+bd(x_2,y_2))$ for all $a\neq b>0$, the orbit $\mathcal{O}_{F_\rho}$ has the proper-on-bisector property. Only in the case $a=b>0$ does $\mathcal{O}_{F_\rho}$ not have the proper-on-bisector property, which is the case of the $\ell^1$ norm studied by D'Achille \cite{DAchille2024ipvtbeyondhyperbolic}. 
\end{itemize}

In general, for every symmetric space $\X$ of higher rank both behaviors can occur.
\begin{itemize}
    \item[-] Let $\rho \in \mathfrak{a}^\ast$ be the half sum of positive roots as in (\ref{eqn: weyl vector}) and consider the metric $d_\rho$ associated to the norm $N(x)=\max_{w\in W}\rho(w\cdot x)$ on $\mathfrak{a}$, where $W$ is the Weyl group. 
    For this metric the orbit $\mathcal{O}_{F_\rho}$ never has the proper-on-bisector property. 
    The corresponding ideal
    Poisson--Voronoi tessellation for this metric coincides with the ideal Poisson--Voronoi tessellation for the Riemannian metric. 
    See Remark \ref{rem:drho}.
    \item[-] For ``generic" polyhedral metrics, $F_\rho$ is a vertex and $\mathcal{O}_{F_\rho}$ has the proper-on-bisector property. 
    This is the case for the Finsler metric associated to $N(x)=\max_{w\in W}\omega(w\cdot x)$ for almost all linear forms $\omega\in \mathfrak{a}^*$, see Proposition \ref{prop:ProperbissectorLinear forms}.

\end{itemize}

Further corollaries of Theorem \ref{thm_intro:PorperBissectrors consequences} are discussed in Section \ref{subsec:Corollaries}.

\medskip

For a symmetric space of rank one, the Riemannian metric is (up to scaling) the unique invariant Finsler metric. 
By Theorem~\ref{thm_intro:bounded_bisector} it has the proper-on-bisector property. 

Thus we obtain the following corollary.

\begin{maincor}
Let $\X$ be a symmetric space of rank one, and     $(C_p^{\Pi})_{p\in \Pi}$ the ideal Poisson--Voronoi tessellation. 
Then almost surely the side between any two cells is bounded.
\end{maincor}

These results show that, in some sense, generic polyhedral Finsler metrics on higher rank symmetric spaces behave more like rank one symmetric spaces. A similar phenomenon occurred in \cite{davaloriestenberg2026dirichletsanosov}, where the first two authors associated to certain polyhedral Finsler metrics $d_\omega$ a class of discrete subgroups which behave like convex cocompact isometry groups of rank one symmetric spaces, in that their Dirichlet domains are finite-sided in the same strong sense. 

\subsection{The ideal Selberg--Poisson--Voronoi tessellation}
A key property of real hyperbolic space $\mathbb{H}^n$ is that it admits a \emph{projective model} where Riemannian geodesics are projective segments and, more generally, complete totally geodesic $k$-submanifolds are intersections of $\mathbb{H}^n$ with projective $k$-planes. 
This special property fails for other symmetric spaces of non-compact type, but nevertheless some do admit projective models. 
In these models the projective geometry may be considered even though it might be incompatible with the Riemannian metric. 

\medskip

Such a method goes back to Selberg, who considered finite-sided ``Dirichlet'' domains for uniform lattices in $\SL(d,\R)$ associated to the invariant
    \[ \mathfrak{s}(x,y) \coloneqq \frac12 \log \frac1{d} \Tr(x^{-1}y) \]
on the space of positive definite symmetric matrices of determinant one. 
Bisectors of $\mathfrak{s}$ are projective hyperplanes. 
With this method, Selberg established that uniform lattices of $\SL(d,\R)$ are locally rigid when $d \ge 3$ \cite{Selberg1960discontinuous}, an early precursor to Margulis' superrigidity theorems. 
Kapovich recently proposed a more systematic investigation of Dirichlet--Selberg domains for discrete subgroups of $\SL(d,\R)$ \cite{Kapovich2023geometricalgorithms}, see \cite{Du2024selbergbisectors,Du2025selbergcompleteness,davaloriestenberg2026dirichletsanosov} for recent developments. 
The Selberg invariant is at bounded distance from the Finsler metric $d_{\iota\omega_1}$. 

\medskip

The space of horofunctions for the Selberg invariant decomposes as $n-1$ orbits $\mathcal{O}_{F_i}$, which can be identified with the space of positive $n\times n$ symmetric matrices of corank $i= 1,2,\cdots,n-1$. 
Each of these spaces admit an $\SL(n,\R)$-invariant measure. The orbit $\mathcal{O}_{F_\rho}$ where the ideal Poisson--Voronoi tessellation lives is $\mathcal{O}_{F_1}$. 

We prove the following theorem for the ideal Selberg--Poisson--Voronoi tessellation, highlighting the similarity with the ideal Poisson--Voronoi tessellation of $\mathbb{H}^n$ (Theorem \ref{thm:Selberg Main}).

\begin{mainthm}
   The ideal Selberg--Poisson--Voronoi tessellation satisfies the following properties: 
   \begin{enumerate}
     \item Almost surely, the side between any two cells is bounded. 
     
     \item Almost surely, each cell $c$ is convex, has locally finitely many sides and its closure meets $\partial \mathcal{X}$ in a single point $P_c$. 
     
     \item Moreover, for every projective line $L$  passing through $P_c$ and some  $X\in \mathcal{X}$, the intersection of $L$ with the cell $c$ contains a non-trivial segment $[X',P_c)$.
   \end{enumerate}
\end{mainthm}

As a comparison with the ideal Selberg--Poisson--Voronoi tessellation illustrated in Figure \ref{fig:CellSymSpace}, which has bounded sides between pairs of cells, we illustrate in Figure \ref{fig:CellSymSpace2} the opposite phenomenon. We draw a cell for a Poisson--Voronoi tessellation associated to the $\SL(3,\R)$-invariant measure $\mu_2$ on the other orbit of horofunctions for the Selberg invariant $\mathcal{O}_{F_2}$. On this second picture all sides of the cell are unbounded. 

\medskip

In Figure \ref{fig:CellSymSpace} and \ref{fig:CellSymSpace2} we just draw the intersection of a cell with a 3-dimensional subspace of the 5-dimensional space $\SL(3,\R)/\SO(3,\R)$. 
This subspace is a \emph{parallel set}: it is a totally geodesic subspace isometric to $\mathbb{H}^2\times \R$ that appears as a cone in the projective model for $\SL(3,\R)/\SO(3,\R)$.

\begin{figure}[htbp!]
    \centering
    \includegraphics[scale=.4]{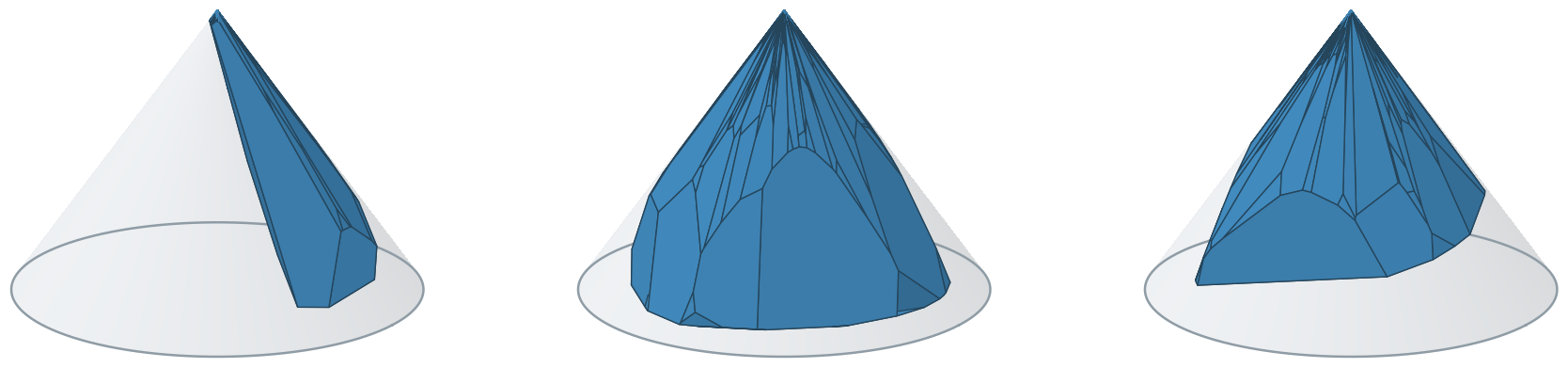}
    \caption{The intersection of a $3$-dimensional slice of $\SL(3,\R)/\SO(3,\R)$ with a cell $c$ for 3 realizations of the ideal Selberg--Poisson--Voronoi tessellation. This slice is a cone with circular base, and is chosen so that $P_c$ is the tip of the cone.}
    \label{fig:CellSymSpace}
\end{figure}

\begin{figure}[htbp!]
    \centering
    \includegraphics[scale=.4]{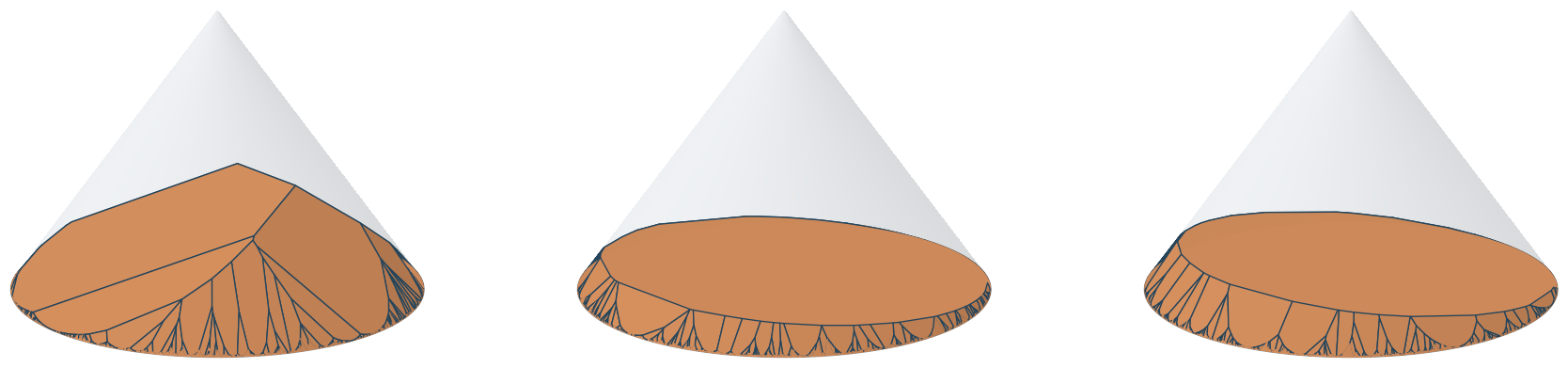}
    \caption{The intersection of a $3$-dimensional slice in $\SL(3,\R)/\SO(3,\R)$ with a cell for 3 realizations of a Poisson--Voronoi tessellation associated to the unique invariant measure $\mu_2$ supported on the orbit of horofunctions $\mathcal{O}_{F_2}$. The parallel set is chosen so that the base of the cone is contained in the closure of the cell.}
    \label{fig:CellSymSpace2}
\end{figure}

\subsection{Organization of the paper}

We begin in Section \ref{sec:prelim} by introducing the necessary background on Lie algebras and symmetric spaces,and on invariant Finsler metrics on symmetric spaces. We illustrate these notions in Section \ref{subsec:Symmetric spaces examples} for the symmetric spaces $(\mathbb{H}^2)^n$ as well as $\SL(n,\R)/\SO(n,\R)$.

In Section \ref{sec:HorofunctionsCompact} we introduce spaces of horofunctions. We then describe the decomposition into orbits of the space of horofunctions for polyhedral Finsler metrics. We illustrate these notions for the symmetric spaces $\mathbb{H}^2\times \mathbb{H}^2$ and $\SL(n,\R)/\SO(n,\R)$ in Section \ref{subsec:H2H2} and \ref{subsec:Horofunctions for SLn}.

In Section \ref{sec:bisectors} we discuss bisectors for horofunctions, in the setting of polyhedral Finsler metrics, and in particular study the proper-on-bisector property. 
We also show that bisectors in $\X$ have empty interior, which may fail for normed vector spaces.

In Section \ref{sec:Measures} we compute the limit of the renormalized pushforwards of the Lebesgue measure into the space of horofunctions.

In Section \ref{sec:IdealPoisson} we discuss Poisson point processes and prove our main results on the ideal Poisson--Voronoi tessellation. 
We illustrate the consequences in the example of  $\mathbb{H}^2\times \mathbb{H}^2$ in Section \ref{subsec:ExamplesV2}. 

Finally in Section \ref{sec:Selberg}, we recall the Selberg invariant on the symmetric space $\SL(n,\R)/\SO(n,\R)$ and prove our results on the ideal Selberg--Poisson--Voronoi tessellation.

\subsection*{Acknowledgments}
The second author thanks Amanda Wilkens for interesting discussion, and all three authors thank Miko{\l}aj Fr{\k{a}}czyk for interesting discussion. 

A.W. was supported by the European Research Council under ERC-Advanced Grant 101018839. She thanks the Hector Fellow Academy for support.  

C.D. was funded by the ERC Grant 101124349. 
Views and opinions expressed are however those of the authors only and do not
necessarily reflect those of the European Union or the European Research Council Executive Agency.
Neither the European Union nor the granting authority can be held responsible for them.

\subsection*{AI usage statement}
The illustrations in Figure \ref{fig:TessellationsH2}, Figure \ref{fig:CellSymSpace} and Figure \ref{fig:CellSymSpace2} in the introduction were created using python code that was written by AI, and reviewed by the authors. 
In the final stages, we used AI to assist in proof-reading the manuscript. 
Otherwise the mathematical arguments and text in this paper are made by the authors with no AI input.

\section{Preliminaries}
\label{sec:prelim}

In this paper we will consider symmetric spaces of non-compact type associated to semi-simple Lie groups. 
In this section we recall all the necessary background about the geometry of these spaces, and their flag manifolds. References for most result cited here can be found in \cite{Eberlein1996book}. 
In Section \ref{subsec:Symmetric spaces examples} we illustrate each of the notions on the concrete examples $G=\SL(n,\R)$ and $G=\SL(2,\R)^n$.

\subsection{Lie algebra background}

Let $G$ be a connected semi-simple Lie group with  finite center and let $\mathfrak{g}$ be its Lie algebra. 
It admits a Cartan decomposition $\mathfrak{g} = \mathfrak{k} \oplus \mathfrak{p}$ and associated Cartan involution $\theta$ acting by $-1$ on $\mathfrak{p}$ and $+1$ on $\mathfrak{k}$. 
Let $\mathfrak{a}$ be a maximal abelian subspace of $\mathfrak{p}$, which is a maximal real split toral subalgebra. 

The \emph{real rank}, or simply the rank, of $G$ is defined as the dimension of $\mathfrak{a}$. 
Such a subalgebra is unique up to the adjoint action of $G$. 
The Lie algebra $\mathfrak{g}$ admits a decomposition into \emph{root spaces} that is :

\begin{equation}\label{eq:Root Decomposition}
\mathfrak{g}=\mathfrak{a}\oplus \mathfrak{m}\oplus \bigoplus_{\alpha \in \Sigma}\mathfrak{g}_\alpha.
\end{equation}

Here $\Sigma$ is a finite set of non-zero linear forms on $\mathfrak{a}$ called \emph{roots}, and $\mathfrak{g}_\alpha$ is the simultaneous eigenspace $\lbrace x\in \mathfrak{g}\mid \forall \tau \in  \mathfrak{a},\, [\tau , x]=\alpha(\tau)x\rbrace$. 
The subalgebra $\mathfrak{m}$ is the orthogonal complement of $\mathfrak{a}$ for the Killing form of the zero eigenspace $\mathfrak{g}_0$. 

\medskip

The \emph{Weyl group} $W$ is the quotient $N_\mathfrak{a}/Z_\mathfrak{a}$ of the normalizer $N_\mathfrak{a}=\lbrace g\in G\mid\text{Ad}_g(\mathfrak{a})=\mathfrak{a} \rbrace$ by the centralizer $Z_\mathfrak{a}=\lbrace g\in G\mid {\text{Ad}_g}_{|\mathfrak{a}}=\text{Id} \rbrace$. 
This group is a finite group acting naturally on $\mathfrak{a}$ and $\mathfrak{a}^*$ and preserving the root system $\Sigma\subset \mathfrak{a}^*$.

\medskip

An element $\tau\in \mathfrak{a}$ is \emph{regular} if for all $\alpha\in \Sigma$, $\alpha(\tau)\neq 0$. 
The choice of such an element $\tau$ defines a decomposition $\Sigma=\Sigma^+\cup \Sigma^-$ where $\Sigma^+=\lbrace \alpha\mid \alpha(\tau)>0\rbrace$. 
The subset $\Sigma^+$ is called the \emph{positive roots}.

\medskip

We fix once and for all for $G$ the subalgebra $\mathfrak{a}$ and a choice of positive roots $\Sigma^+\subset \Sigma$. 
Let $\mathfrak{a}^+\subset \mathfrak{a}$ be \emph{standard Weyl chamber} defined as the set $\lbrace x\in \mathfrak{a}\mid \forall \alpha\in \Sigma^+, \alpha(x)>0\rbrace$. This standard Weyl chamber is a cone bounded by finitely many hyperplanes $\text{Ker}(\alpha)$ for $\alpha\in \Delta\subset \Sigma^+$ (we exclude from $\Delta$ the roots $\alpha$ such that $\alpha/2\in \Sigma)$. These roots are called \emph{simple roots}, and they form a basis of $\mathfrak{a}^*$. Weyl chambers are the translates of $\mathfrak{a}^+$ by the Weyl group, and a crucial property of these Weyl chambers is that their closures form a tiling of $\mathfrak{a}$, i.e. the translates of $\overline{\mathfrak{a}^+}$ have disjoint interior and cover all of $\mathfrak{a}$.

Given a non-empty set of simple roots $\Theta\subset \Delta$ we define the associated parabolic Lie subalgebra as:
$$\mathfrak{p}_\Theta=\mathfrak{a}\oplus \mathfrak{m}\bigoplus_{\alpha \in \Sigma^+ \cup -\text{span}(\Delta\setminus \Theta)}\mathfrak{g}_\alpha.$$

In this expression $\text{span}(\Delta\setminus \Theta)\subset \Sigma^+$ is the set of roots that are finite positive integral combinations of roots in $\Delta\setminus \Theta$. 
For $\tau \in \overline{\mathfrak{a}^+}$, if $\Theta = \{ \alpha \in \Delta : \alpha(\tau)>0 \}$, then $\Sigma^+ \cup -\text{span}(\Delta\setminus \Theta)$ is the set of roots nonnegative on $\tau$, and $\Sigma^+ \setminus \text{span}(\Delta \setminus \Theta)$ is the set of roots positive on $\tau$.

\medskip

\subsection{Geometry of symmetric spaces}

Each Cartan involution $\theta$ integrates to a unique Lie group automorphism of $G$, whose fixed point set is a maximal compact subgroup of $G$. 
The quotient space $\X = G/K$ is the \emph{symmetric space of non-compact type} associated to $G$. 
It carries a Riemannian metric induced by the Killing form $B$ of $\mathfrak{g}$ which is unique up to rescaling on each irreducible de Rham factor. 
With this metric $\X$ is indeed a non-positively curved Riemannian symmetric space without Euclidean factors. 

The differential of the orbit map at the identity naturally identifies the tangent space of $\X$ at $x = [K]$ with $\mathfrak{p}$. 
Under this identification, the Riemannian exponential map agrees with the restriction of the Lie theoretic exponential map to $\mathfrak{p}$. 

Let $A=\exp(\mathfrak{a})\subset G$ be a restricted Cartan subgroup, where $\mathfrak{a} \subset \mathfrak{p}$ is a maximal abelian subspace.
Then $A\cdot x$ is a \emph{maximal flat} in $\X$, i.e.\ a totally geodesic submanifold isometric to Euclidean space.

For $A^+=\exp(\overline{\mathfrak{a}^+})$ we have the Cartan decomposition:
Every $g\in G$ decomposes as $g=k_1\exp(\tau)k_2$ for some $k_1,k_2\in K$ and $\tau \in \overline{\mathfrak{a}^+}$, or in other words $G=K\overline{A^+}K$. Moreover the element $\tau$ is uniquely determined by $g$ and is called the \emph{Cartan projection} of $g$.

On the level of the Lie algebra, the Cartan decomposition allows one to define a map $\mu_o \colon \mathfrak{p}\to\overline{\mathfrak{a}^+}$ such that $\mu(v)$ is the unique element in the $\text{Ad}_K$-orbit of $v$ in $\overline{\mathfrak{a}^+}$.

\medskip

Geometrically this result means that one can find a translate of the standard maximal flat $k_1\cdot F$ through the points $o$ and $g\cdot o$, and using the action of the Weyl group one can assume $a\in \overline{\mathfrak{a}^+}$.
More generally for any $x,y\in \X$ one can find a translate $g\cdot F$ of the standard maximal flat $F$ such that $g\cdot x=o$ and $g\cdot y=\exp(a)\cdot o$ for some $a\in \overline{\mathfrak{a}^+}$.
This $a$ does not depend on any choice other than the choice of $x,y$ hence this defines the \emph{vector-valued distance} on $\X$:
$$\vec{d}:\X\times \X\to \overline{\mathfrak{a}^+}.$$
This is the unique $G$-invariant function such that $\vec{d}(o, e^a\cdot o)=a$ for all $a\in \mathfrak{a}$ in the corresponding Cartan subalgebra. This distance is not always symmetric, as in general: 
$$\vec{d}(x,y)=\iota\left(\vec{d}(y,x)\right)$$ 
Here $\iota$ is the \emph{opposition involution}. This involution is defined as  $\iota=-w_0$ where $w_0\in W$ is the \emph{longest element} of the Weyl group, i.e.\ the unique element exchanging $\mathfrak{a}^+$ and $-\mathfrak{a}^+$. 
The opposition involution is non-trivial precisely when the restricted root system has type $A_n (n\ge 2), D_{2n+1} (n\ge 1),$ or $E_6$. 

\medskip

The distance on $\X$ induced by the Riemannian metric can be written as $d(x,y)=\lVert \vec{d}(x,y)\rVert$, for some Euclidean norm on $\mathfrak{a}$.

\subsection{Flag manifolds}
\label{sec:Flag manifolds}

Associated to $G$ is also a natural finite family of compact homogeneous spaces called \emph{flag manifolds}.

When $G$ has real rank $1$ the only associated flag manifold is the visual boundary of its associated negatively curved symmetric space. 

In higher rank the visual boundary has infinitely many $G$ orbits, but only finitely many conjugacy classes of stabilizers arise. 
We emphasize that at the moment we are discussing the Riemannian metric of $G/K$; a similar discussion will apply to the horofunction compactifications relevant to the rest of the paper, where even richer behaviors of orbits and stabilizers will occur. 

\medskip

The \emph{visual boundary} $\partialvis\X$ is defined as the set of geodesic rays $\gamma:\R_{\geq 0}\to \X$ with speed one quotiented by an equivalence relation so that $[\gamma_1]=[\gamma_2]\in \partialvis \X$ if for some $C>0$ one has $\forall t>0$, $d(\gamma_1(t),\gamma_2(t))\leq C$. 
This boundary allows one to define a natural compactification $\X\cup \partialvis\X$ of $\X$ called the \emph{visual compactification}.

\medskip

Every point $p\in \partialvis\X$ is equal up to the action of $G$ to some $[\gamma_\tau]$ where $\gamma_\tau:t\mapsto \exp(t\tau)$ where $\tau\in \overline{\mathfrak{a}^+}$, $\lVert \tau\rVert=1$. 

\medskip

Let  $\tau\in \overline{\mathfrak{a}^+}$, $\lVert \tau\rVert=1$, and let $\Theta=\lbrace\alpha \in \Delta\mid \alpha(\tau)>0\rbrace$. Note here that every non-empty subset $\Theta\subset \Delta$ can be realized this way. We define $P_\Theta$ the stabilizer of $[\gamma_\tau]\in \partialvis \X$. This Lie subgroup depends only on $\Theta$, and is called a \emph{parabolic subgroup}.

We define the horospherical subgroup $U_\Theta\subset P_\Theta$ as the set of elements $u\in G$ such that $u\cdot \gamma_\tau$ and $\gamma_\tau$ are \emph{strongly asymptotic}, meaning that $d(\gamma_\tau(t),u\cdot\gamma_\tau(t))\to 0$ when $t\to \infty$.
The Lie algebra of $P_\Theta$ is the Lie subalgebra $\mathfrak{p}_\Theta$ defined previously.

\medskip

The \emph{flag manifold} $\mathcal{F}_\Theta$ associated to a non-empty subset $\Theta\subset \Delta$ is the quotient $G/P_\Theta$. 
The name comes from the fact that for $G=\SL(n,\R)$, the flag manifolds can be viewed as spaces of partial flags in $\R^n$, where the type of the flag is determined by $\Theta$. Geometrically these flag manifolds can be embedded (non-uniquely) inside the visual boundary of the symmetric space.

\begin{rem} \label{rem:Flag in boundary of a flat}
    The flag manifold $\mathcal{F}_\Theta$ can be embedded $G$-equivariantly into $\partialvis \X$ so that the base point $P_\Theta$ is mapped to $[\gamma_\tau]$. This embedding is not unique, and depends on the choice of a unit vector $\tau\in \overline{\mathfrak{a}^+}$ such that $\Theta=\lbrace\alpha \in \Delta\mid \alpha(\tau)>0\rbrace$. Given a flag $f\in \mathcal{F}_\Theta$ a a maximal flat $F\subset \X$ we say that $f$ is in the boundary of $F$ if for one of these embeddings $f$ is mapped inside $\partialvis F\subset \partialvis \X$. In fact this property does not depend on the choice of an embedding, \ie on the choice of $\tau$.
\end{rem}

Given a pair of non empty subsets $\Theta_1,\Theta_2\subset \Delta$, we call the set of \emph{relative positions} between $\mathcal{F}_{\Theta_1}$ and $\mathcal{F}_{\Theta_2}$ the quotient $P_{\Theta_1}\backslash G/P_{\Theta_2}$, or equivalently the set of pairs  $(f_1,f_2)\in \mathcal{F}_{\Theta_1}\times \mathcal{F}_{\Theta_2}$ up to the diagonal action of $G$. 

\medskip

It is a fact that any pair of flags $(f_1,f_2)\in \mathcal{F}_{\Theta_1}\times \mathcal{F}_{\Theta_2}$ lie in the boundary of a common flat $F$. This allows to show that the set of relative positions is finite and equal to $W_{\Theta_1}\backslash W/W_{\Theta_2}$ where $W$ is the Weyl group and $W_{\Theta_1}$, $W_{\Theta_2}$ its intersection with $P_{\Theta_1}$, $P_{\Theta_2}$. 

\medskip

We now recall the Bruhat order on the set of relative positions.
Given two relative positions $[w_1],[w_2]\in W_{\Theta_1}\backslash W/W_{\Theta_2}$ and their corresponding orbits $\mathcal{O}_1,\mathcal{O}_2\in \mathcal{F}_{\Theta_1} \times \mathcal{F}_{\Theta_2}$, we have $w_1\leq w_2$ if and only if $\mathcal{O}_1\subset \overline{\mathcal{O}_2}$.

\medskip

The opposition involution $\iota$ acts also as a permutation on the set of simple roots $\Delta$. 
The set of relative positions in $W_{\Theta}\backslash W/W_{\iota(\Theta)}$ admits a maximum $[w_0]$ where $w_0$ is the longest element of the Weyl group. 
A pair of flags $(f_1,f_2)\in \mathcal{F}_{\Theta} \times \mathcal{F}_{\iota(\Theta)}$ whose relative position is this special class is called \emph{transverse}. 
The space of transverse pairs of flags is open and dense in $\mathcal{F}_{\Theta} \times \mathcal{F}_{\iota(\Theta)}$.

\medskip

The Bruhat order satisfies the following more geometric property. 
Informally this proposition states that if we fix the relative position of $(f,g)$, then the relative position of $(f^\text{opp},g)$ is minimal when $(f^\text{opp},f,g)$ lie in a common flat.

\medskip

Let $w_0$ be the longest reduced word in the Weyl group, which satisfies $w_0^2=\Id$. 
A class $[w_1]\in W_{\Theta_1}\backslash W /W_{\Theta_2}$ defines a class $[w_0w_1]\in W_{\iota(\Theta_1)}\backslash W /W_{\Theta_2}$ since $W_{\iota(\Theta_1)}=w_0W_{\Theta_1}w_0$. 
This map reverses the Bruhat order (see also \cite[Lemma 3.9]{KapovichLeebPorti2018dynamicsonflagmanifolds}).

\begin{prop}
\label{prop:Tits angle}
Let $f\in \mathcal{F}_{\Theta_1}$ and $f^\text{opp}\in \mathcal{F}_{\iota(\Theta_1)}$ be a transverse pair of flags. Let $g\in \mathcal{F}_{\Theta_2}$. The relative positions $w_1$, $w_2$ corresponding to the pairs $(f,g)$ and $(f^\text{opp},g)$ satisfy:

$$ [w_0w_1]\leq [w_2]. $$

\end{prop}

\begin{proof}
Let $\tau \in \overline{\mathfrak{a}^+}$  be a unit vector such that $\lbrace \alpha \in \Delta \mid \alpha (\tau)>0\rbrace =\Theta$.
Up to acting by $G$, since $f$ and $f^\text{opp}$ are transverse we can assume that the stabilizer of $f$ in $G$ is equal to the stabilizer of the class of the ray $t \mapsto \exp(\tau t)o$ in $\partialvis \X$ and the stabilizer of $f^\text{opp}$ is equal to the stabilizer of the class of the ray $t \mapsto \exp(-\tau t)o$.

\medskip

Let $f'$ be the flag opposite to $f$ in a flat $F$ containing $f$ and $g$ in its boundary. The relative position of the pair $(f',g)$ is therefore $[w_0 w_1]$. We consider the $1$-parameter subgroup $(g_t)$ of $G$ consisting of transvections along a geodesic in the flat $F$ whose limit is $f'$. This group fixes $(f',f,g)$ but since $f$ and $f^\text{opp}$ are transverse then $g_t\cdot f^\text{opp}$ converges to $f'$. This fact is the so-called North-South dynamics of semi-simple elements in semisimple Lie groups. Hence since the pairs $g_t\cdot(f^\text{opp},g)$ converge to $(f',g)$, we have $ [w_0w_1]\leq [w_2]$.

\end{proof}

\subsection{Finsler metrics}

The Lie group $G$ acts transitively on the set of pairs of points with $x,y\in \X$ such that $\vec{d}(x,y)=\tau$ is fixed. 
However it does not act transitively on pairs of points such that $d(x,y)=1$, where we recall that $d$ is the distance associated to the Riemannian metric on $\X$. In fact there exist many more Finsler metrics that one can put on the symmetric space, inducing other $d$-invariant distances.

\medskip

Let $N:\mathfrak{a}\to \R_{\geq 0}$ be a norm, in the sense that for $\tau\in \mathfrak{a}$,  $N(\tau)=0$ if and only if $\tau=0$, $N(\lambda\tau)=\lambda N(\tau)$ for all $\lambda\in \R_{>0}$, and $N(\tau+\tau')\leq N(\tau)+N(\tau')$. 
We do not assume that this norm is \emph{symmetric} in the sense that $N(-\tau)$ is not necessarily equal to $N(\tau)$. 

\medskip

Now assume that $N$ is invariant by the action of the Weyl group. Note that if the opposition involution is trivial, then $N$ is necessarily symmetric. We define the following Finsler metric on the symmetric space. Recall that for $o\in \X$, $\mu_o:T_o\X\to \overline{\mathfrak{a}^+}$ is the Cartan projection

\begin{thm}[{\cite{Planche1995finsler}}]
\label{thm:Planche}
    Let $N$ be a Weyl group invariant norm of $\mathfrak{a}$. For every $o\in \X$ let $N_o$ be the norm on $T_o\X$ defined as:
    $$N_o(v)=N(\mu_o(v)).$$

These norms define a (possibly asymetric) Finsler metric on $\X$. The distance possibly asymetric associated distance function $d^N$ on $\X$ can be written as:
$$d^N(x,y)=N(\vec{d}(x,y)).$$
\end{thm}

Since all norms on $\mathfrak{a}$ are equivalent, these metric on $\X$ are equivalent. However these metrics can have very different behavior: for instance the associated horofunction compactifications can be very different.

\medskip

In the present paper we will focus on \emph{polyhedral Finsler metrics} on $\X$ which are metrics coming from polyhedral norm $N$ on $\mathfrak{a}$ invariant by the Weyl group action.
Let us fix a Weyl group invariant norm $N$ on $\mathfrak{a}$ that is \emph{polyhedral}, i.e. whose unit ball is a polyhedron. 

A simple class of such norms are the norms $|\cdot |_\omega$  associated with some dominant linear form $\omega\in \mathfrak{a}^*$ defined as $|\tau|_\omega=\max_{w\in W} (w\cdot\omega)(\tau)$ for $\tau\in \mathfrak{a}$. This always defines a norm on $\mathfrak{a}$ if $G$ is simple, but not necessarily if $G$ is only semi-simple. 
In general a polyhedral norm $N$ can be written, for some finite family of linear forms $(\omega_i)_{i\in I}$, as: $$N(\tau)=\max_{i\in I} |\tau|_{\omega_i}.$$ 

By convention the associated possibly asymmetric distance will be $d(x,y)=N(y-x)$. This convention matters because these distances are possibly assymetric. This distance defines a possibly asymmetric distance on $\X$:
$$ d^N(x,y)=N(\vec{d}(x,y)). $$

\begin{rem}
Given a Weyl group invariant polyhedral norm $N$ on, $\mathfrak{a}$, we denote by $\mathcal{D}^N(\X)$, $\mathcal{H}^N(\X)$, $\overline{\X}^N$, $\partial^N\X$ repectively the spaces of distance functions, of horofunctions, the horofunction compactification of $\X$ and the horoboundary of $\X$ for the Finsler metric $d^N$. 
\end{rem}

\subsection{The invariant measure of the symmetric space}
\label{sec:Invariant measures sym space}
We first recall how the volume measure $\lambda$ on the symmetric space disintegrates along $K$-orbits. 
We consider a multiple of the volume associated to the Riemannian metric in order to get rid of constants depending only on the group $G$.
The restricted Weyl vector $\rho$ emerges from this formula as the direction of largest volume growth.

Let $o \in \X$ be the point fixed by $K \subset G$, and let $M$ be the stabilizer in $K$ of an element of $\mathfrak{a}^+$. 

\begin{proposition}\label{prop: volume density along K-orbits}
    The volume measure $\Leb$ on $\X$, up to a multiplicative constant, satisfies for every integrable function $f$ on $\X$:
    $$\int_\X f(x)\mathrm{d}\Leb =\int_{\mathfrak{a}^+} \left(\int_{K/M}f(kM \cdot e^\tau o) \mathrm{d}\text{Haar}(kM)\right) \mathrm{d}\mu_\mathfrak{a}(\tau).$$
    Here $\text{Haar}$ is the Haar probability measure on $K/M$, which is naturally identified with the orbit $K\cdot e^\tau o$ for all regular $\tau\in \mathfrak{a}^+$ (hence for almost every $\tau \in \mathfrak{a}^+$), and $\mu_\mathfrak{a}$ is the measure $\prod_{\alpha\in \Sigma^+} \left(2\sinh(\alpha (\tau))\right)^{m_\alpha}\mathrm{d}\Leb_\mathfrak{a}(\tau)$ with $\Leb_\mathfrak{a}$ the volume measure on the Euclidean space $\mathfrak{a}$ and $m_\alpha$ is the dimension of the root space $\mathfrak{g}_\alpha$ associated with $\alpha\in \Sigma$.
\end{proposition}

From now on, for convenience we will rescale the Riemannian metric on $\X$ so that the Lebesgue measure $\Leb$ is rescaled exactly to satisfy this formula.

\begin{proof}
For almost every point $x \in \X$ the vector-valued distance $\vec{d}(o,x)$ is regular.
At such a regular point, $T_x\X$ splits orthogonally as $T_x((K\cdot x) \oplus T_xF$ where $F$ is the unique flat containing the base point $o\in \X$  and $x\in \X$. 
The invariant volume measure $\Leb$ therefore splits locally as a product measure of the Lebesgue measure $\Leb_{\mathfrak{a}}$ on $F \simeq \mathfrak{a}$ and a $K$-invariant measure on $K\cdot x$. 
Hence there is some $\Phi \colon \mathfrak{a}^+ \to \R_{\ge 0}$ relating this measure to the Haar probability measure on $K/M$. 

\medskip

We look at the orbit map $\phi \colon K/M \to K \cdot x$, where $x = e^v o$. 
Note that $K$ and hence $K/M$ come equipped with a natural Riemannian metric coming from the Killing form of $\mathfrak{g}$.
The determinant of the differential of $\phi$ is constant and equal to the determinant of the linear map:
$$ k\in \mathfrak{k}/\mathfrak{m} \mapsto  \pi_\mathfrak{p}(e^{\mathrm{ad}_v} k)=\sinh(\mathrm{ad}_v) k\in \mathfrak{p}/\mathfrak{a}.$$
Here $\pi_\mathfrak{p}$ is the projection onto $\mathfrak{p}$ coming from the $\mathfrak{k}\oplus \mathfrak{p}$ decomposition.
This linear map takes $\mathfrak{k}_\alpha = (\mathfrak{g}_\alpha\oplus \mathfrak{g}_{-\alpha}) \cap \mathfrak{k}$ isomorphically to $\mathfrak{p}_\alpha = (\mathfrak{g}_\alpha\oplus \mathfrak{g}_{-\alpha}) \cap \mathfrak{p}$ for all (positive) roots $\alpha$.
For $X \in \mathfrak{g}_\alpha$ the map is given by
    \[ X+ \theta X_\alpha \mapsto \sinh(\alpha(v)) (X- \theta X_\alpha), \]
where $\theta$ is the Cartan involution corresponding to $\mathfrak{g} = \mathfrak{k}\oplus \mathfrak{p}$.
Therefore the determinant is equal to :
$$ \prod_{\alpha\in \Sigma^+} \sinh(\alpha (v))^{m_\alpha} $$
where $m_\alpha$ denotes the dimension of the root space $\mathfrak{g}_\alpha$.
In particular the Riemannian measure of $K\cdot x$ is proportional to this quantity. We choose the metric on $\X$ so that we get the desired result.
\end{proof}

Recall that the \emph{restricted Weyl vector} was defined in (\ref{eqn: weyl vector}) as half the sum of positive roots, weighted by the dimension of root spaces.
Geometrically, it corresponds to a barycenter of the Weyl chamber, but weighted according to the geometry of $\X$ in directions transverse to the maximal flat.

\begin{rem}\label{rmk: def of Phi}
    It will be convenient to set 
        \[\Phi(v) \coloneqq \prod_{\alpha\in \Sigma^+} (1-e^{-2\alpha(v)})^{m_\alpha} = e^{-2 \rho(v)} \prod_{\alpha\in \Sigma^+} (2\sinh(\alpha (v)))^{m_\alpha} .\] 
    It is easy from the definitions to read off the bound:
        \[\Phi(v) \leq 1.\]
    The function $\Phi$ is close to $1$ when $v \in \mathfrak{a}^+$ is far from the walls, but tends to $0$ as $v$ gets close to the walls.
\end{rem}

\subsection{Examples}
\label{subsec:Symmetric spaces examples}

Throughout the paper we will illustrate our results with the following examples.

\subsubsection*{Let \texorpdfstring{$G=\SL(2,\R)^n$}{G=SL(2,R)}:}
The associated maximal compact is conjugated to $\SO(2,\R)^n$ and the associated symmetric space is $\X=\left(\mathbb{H}^2\right)^n=\mathbb{H}_1^2\times \cdots \times \mathbb{H}_n^2$, which has rank $n$. A restricted Cartan subalgebra $\mathfrak{a}\subset \mathfrak{sl}(2,\R)$ is simply $\mathfrak{a}=(\mathfrak{a}_0)^n\simeq \R^n$ where $\mathfrak{a}_0\subset \mathfrak{sl}(2,\R)$ is the subalgebra of traceless diagonal matrices. Flats in $\left(\mathbb{H}^2\right)^n$ are exactly of the form $\gamma_1\times \cdots \times \gamma_n$ for geodesics $\gamma_i$ in $\mathbb{H}^2$ for all $1\leq i\leq n$.    

The Weyl group is the group $\lbrace -1,1\rbrace^n$ where $(\epsilon_i)\in W$ acts on $\mathfrak{a}$ as:
$$(x_1,\cdots, x_n)\mapsto (\epsilon_1 x_1,\cdots, \epsilon_n x_n).$$

The group $\SL(2,\R)^n$ admits exactly $2n$ roots, which are $(\pm \alpha_i)$ for $1\leq i \leq n$ where $(\alpha_i)$ is the dual base to the standard basis of $\mathfrak{a}=(\mathfrak{a}_0)^n$. These roots $(\alpha_1,\cdots, \alpha_n)$ form a choice of positive roots, and also form the set of simple roots. 
The model Weyl chamber $\mathfrak{a}^+$ associated with this choice of positive roots is  the positive orthant $\R_{>0}^n\subset \R^n\simeq \mathfrak{a}$.

\medskip

The flag manifold associated to a non-empty subset $\Theta\subset \lbrace \alpha_1,\cdots, \alpha_n\rbrace=\Delta$ where $\Theta=\lbrace \alpha_i\mid i\in I\rbrace$ is the product:
$$\mathcal{F}_\Theta=\prod_{i\in I} \partial \mathbb{H}^2_i.$$
We have a natural action of $G$ on this flag manifold, where each copy of $\SL(2,\R)$
acts only on the corresponding copy of $\mathbb{H}^2$. Note that this product can be seen inside of the visual boundary of $\X$. More precisely for each $\tau \in \R_{>0}^I\subset \R_{>0}^n$ this flag manifold can be identified with the orbit of the rays with intial velocity $v=(v_1,\cdots v_n)$ when $v_j=0$ for $j\notin I$ and $\lVert v_i\rVert=\tau_i$ for $i\in I$. This orbit only depends on the projective class $[\tau]$ of $\tau$. The union of all these orbits when $\Theta$ and $[\tau]$ vary is exactly the full visual boundary.

A flag $(y_i)_{i\in I}$ lies in the boundary of the flat $\gamma_1\times \cdots \times \gamma_n$ if for all $i\in I$, $y_i\in \partial \mathbb{H}^2_i$ is one of the endpoints of $\gamma_i$.

\medskip

Now let us define some $G$-invariant metrics on $\X=\left(\mathbb{H}^2\right)^n$. For that it suffices to consider a $W$-invariant norm $N$ on $\mathfrak{a}\simeq\R^n$ and define :
$$d_N((x_i)_{1\leq i\leq n},(y_i)_{1\leq i\leq n})=N((d(x_i,y_i))_{1\leq i\leq n}).$$

More concretely let us consider the case $n=2$, we will consider the following polyhedral norms on $\mathfrak{a}^+\simeq \mathbb{R}^+\times \mathbb{R}^+$: 
$$N_\infty(x,y)=max(|x|,|y|),$$
$$N_1(x,y)=|x|+|y|,$$
$$N_{a,b}(x,y)=a|x|+b|y|.$$
The third norm depends on the data of real numbers $b>a>0$. We will also be interested in the euclidean norm, which is not a polyhedral norm, defining the Riemannian distance of $\X$:
$$N_{\text{Eu}}(x,y)=\sqrt{x^2+y^2}.$$

\subsubsection*{Let \texorpdfstring{$G=\SL(n,\R)$}{G=SL(n,R)}:}
The associated maximal compact is conjugated to $\SO(n,\R)$. A convenient model for this symmetric space is the following : $\X$ is the space of determinant one matrices in $\text{Sym}_n^+$, the space of positive definite $n$ by $n$ symmetric matrices. The action of $\SL(n,\R)$ on $\mathcal{X}$ comes from its action on $\text{Sym}_n^+$ defined as $g\cdot M=g^\perp M g$. One can also consider the \emph{projective model} $\X\simeq \mathcal{X}=\mathbb{P}(\text{Sym}_n^+)$.
The restricted Cartan subalgebra $\mathfrak{a}\subset \mathfrak{sl}(n,\R)$ is simply the subalgebra of traceless diagonal matrices, which can be viewed as a hyperplane in $\R^n$: $\mathfrak{a}=\lbrace (x_i)\in \R^{n}\mid x_1+\cdots +x_n=0\rbrace $ This has dimension $n-1$, hence the real rank of $\SL(n,\R)$ is $n-1$. Flats in $\mathcal{X}$ are exactly translates of the model flat:
$$\lbrace \Diag(e^{\sigma_1},e^{\sigma_2},\cdots, e^{\sigma_n})\mid \sigma_1+\sigma_2+\cdots +\sigma_n=0\rbrace .$$

The Weyl group is the groups of permutations on $n$ elements acting on $\R^n$ by permuting basis elements and preserving the hyperplane $\mathfrak{a}$. The set of roots $\Sigma$ consists of the restriction to $\mathfrak{a}$ of the linear forms $e_i^*-e_j^*$ for $1\leq i\neq j\leq n$, where $(e_i^*)$ is the dual base to the standard base of $\R^n$. The roots $e_i^*-e_j^*$ with $i<j$ form a set of positive roots $\Sigma^+$. The corresponding set of simple roots is $\Delta=\lbrace e_i^*-e_{i+1}^*\mid 1\leq i<n \rbrace $. Let us write $\alpha_i=e_i^*-e_{i+1}^*$ for $1\leq i<n$. The model Weyl chamber associated with this choice of positive roots is:
$$\mathfrak{a}^+=\lbrace (\sigma_1,\cdots ,\sigma_n)\in \R^n\mid \sigma_1> \cdots > \sigma_n,\, \sigma_1+\cdots +\sigma_n=0\rbrace.$$

The flag manifold $\mathcal{F}_\Theta$ associated with a non-empty subset $\Theta\subset \Delta$ where $\Theta=\lbrace \alpha_i\mid i\in I\rbrace$ for $I\subset \lbrace 1, \cdots, n-1\rbrace$ can be viewed as the space of flags $(E_i)_{i\in I}$ where $E_i$ is a dimension $i$ vector subspace of $\R^n$ and $E_i\subset E_j$ of $i<j$, $i,j\in I$. The group $\SL(n,\R)$ acts naturally on $\mathcal{F}_\Theta$.

\medskip

Note here that the visual boundary is not the boundary of the projective model for $\X$ when $n\geq 3$. Given a non-empty $\Theta\subset \Delta$ and $\tau\in \mathfrak{a}$ such that for $i\notin I$, $\alpha_i(\tau)=0$, one can define a corresponding orbit in the visual boundary which is $G$-equivariantly diffeomorphic to $\mathcal{F}_\Theta$, and the union of these orbits when $\Theta$ and the porjective class $[\tau]$ of $\tau$ is equal to the full visual boundary. 

A flag lies in the boundary of a flat of the form $\lbrace \Diag(e^{\sigma_1},e^{\sigma_2},\cdots, e^{\sigma_n})\mid \sigma_1+\sigma_2+\cdots +\sigma_n=0\rbrace $ for some basis of $\R^n$ if up to the action of the permutation group on this basis, the flag is of the form $(E_i)_{i\in I}$ where $E_i=\langle e_1,\cdots, e_i\rangle$ for all $i\in I$. 

\medskip

Let us define some $G$-invariant metrics on $\X$. For that it suffices to consider a $W$-invariant norm $N$ on $\mathfrak{a}$ and define :
$$d_N(x,y)=N(\vec{d}(x,y)).$$
Here $\vec{d}(x,y)=(\sigma_1,\cdots \sigma_n)$ are the logarithms of the eigenvalues of the matrix $yx^{-1}$ where $y,x$ are viewed as matrices in $\text{Sym}_n^+$, ordered so that $\sigma_1\geq \cdots, \geq \sigma_n$.
Then $\mathfrak{a}$ is identified with $\{ x \in \R^n : \sum x_i =0\}$, and $W \cong S_n$ acts by permuting the indices.

More concretely we will consider the following polyhedral norms on $\mathfrak{a}$: 
$$N_{\omega_1}((x_i))=\max_{1\leq i\leq n}(x_i),$$
$$N_{\iota \omega_1}((x_i))=\max_{1\leq i\leq n}(-x_i),$$
$$N_1((x_i))=\sum_{i=1}^n |x_i|.$$
One can also consider the euclidean norm, which is not a polyhedral norm, defining the Riemannian distance of $\X$:
$$N_{\text{Eu}}((x_i))=\sqrt{\sum_{i=1}^n x_i^2}.$$

\section{Horofunctions and horofunction compactifications of symmetric spaces}
\label{sec:HorofunctionsCompact}
In this section we introduce and investigate the space of horofunctions of symmetric spaces. 

\subsection{Horofunction compactifications}

We consider a metric space $(X,d)$ with possibly asymmetric distance function.
The metric induces a natural embedding of $X$ into the space $C(X)$ of continuous functions $X \to \R$:
    \begin{equation}\label{eqn: distance function}
        y \mapsto d(\cdot,y) .
    \end{equation}
A \emph{distance-like function on $X$} is any function of the form $d(\cdot,y)+c$ for some $y \in X$ and $c \in \R$.
The space of all distance-like functions is denoted $\mathcal{D}(X)$. 
We consider the closure of this set $\overline{\mathcal{D}(X)} \subset C(X)$ in the compact-open topology. 
The complement of $\mathcal{D}(X)$ in $\overline{\mathcal{D}(X)}$ is denoted $\mathcal{H}(X)$ and called the space of \emph{horofunctions}.
We let $\partial X$ denote $\mathcal{H}(X)/\R$, where $\R$ is identified with the set of constant functions, and call this quotient the \emph{horoboundary}.
When $X$ is proper, $\partial X$ is compact, and the inclusion of $X$ into $C(X)/\R$ induced by \eqref{eqn: distance function} is open and dense in its closure, which is denoted $\overline{X}$ and called the \emph{horofunction compactification}.

\begin{remark}
    In this paper we reserve the name horofunction for elements of $\mathcal{H}(X)$, which are functions on $X$.
    In the literature the elements of $\partial X$ are sometimes called horofunctions, even though these functions are only defined up to constants.
\end{remark}

\begin{remark}
    For a proper CAT(0) space such as $\X$ with its invariant Riemannian metric, the visual compactification identifies naturally with the horofunction compactification \cite{EberleinONeill1973visibility} via the map induced by
        \[ \left( c\colon [0,\infty) \to \X \right) \mapsto \left( x \mapsto \lim_{t \to \infty} d(x,c(t)) - d(c(0),c(t))  \right) .\]
    We write $\partial_{vis} \X \simeq \partial^{Riem} \X$. 
\end{remark}

\subsection{Horofunction compactifications of normed vector spaces}

As suggested by Theorem \ref{thm:Planche} and explained more precisely in the next subsection the study of invariant polyhedral Finsler metrics on $\X$ often reduces in some sense to the study of normed vector spaces. 

\medskip

Let $\mathfrak{a}$ denote a finite-dimensional vector space and let $N$ be a possibly asymmetric norm on $\mathfrak{a}$.
The unit ball is
    \[ B^N = \{ x \in \mathfrak{a} : N(x) \le 1 \} \]
and the dual unit ball is given by
    \[ (B^N)^\ast = \{ \alpha \in \mathfrak{a}^\ast : \forall v \in B^N, \, \alpha(v) \le 1 \} .\] 
Observe that for any $v \in \mathfrak{a}$, 
\begin{equation}\label{eqn: norm from the dual ball}
    N(v) = \sup \left\{ \alpha(v) : \alpha \in (B^N)^\ast \right\}
\end{equation}
which is realized for some $\omega \in (B^N)^\ast$.

\medskip

In the whole paper we will assume that $N$ is a polyhedral norm, except in Remark \ref{rem:smoothsquare norm} and Figure \ref{fig:smoothsquare norm}, and when we explicitly consider the Euclidean norm.

\medskip

The face $F$ of $B^N$ containing $v \in B^N$ is the set of points $v' \in B^N$ so that the segment from $v$ to $v'$ may be extended on both sides within $B^N$. 
Each face is relatively open, i.e.\ open in its affine span, in $\mathfrak{a}^*$. 
The singleton faces are also called vertices or extreme points, and $B^N$ is the convex hull of its set of vertices. The set of vertices is denoted by $\ver(F)$.

The set of faces of $B^N$ forms a partition of it, and it is a finite partition exactly when $N$ is a polyhedral norm, i.e.\ when $B^N$ is a polyhedron. In this case $\ver((B^N)^*)$ is finite and for $v\in \mathfrak{a}$:
$$N(v)=\max_{\omega\in \ver((B^N)^*)} \omega(v).$$

We will be mostly interested in \emph{exposed faces}, which are faces in $\partial B^N$. Associated to an exposed face $F$ of $B^N$ is its dual face which is an exposed face of $(B^N)^*$. This is the unique face such that:
    \begin{equation}\label{eqn: def of Fvee}
        \overline{F^\vee} \coloneqq \{ \beta \in (B^N)^* \mid \forall v \in F, \beta(v)=1 \}.
    \end{equation}

\begin{lem} \label{lem: Alternative definition Fvee}
Alternatively for any $v \in F$, we can write \begin{equation}\label{eqn: def of Fvee bis}
        \overline{F^\vee} \coloneqq \{ \beta \in (B^N)^* : \beta(v)=1 \} 
    \end{equation} 

Equivalently, $\overline{F^\vee}$ is exactly the convex hull of the linear forms $\omega\in \ver((B^N)^*))$ such that $\omega=1$ on $F$.
\end{lem}

Note again that we consider relatively open faces. 

\begin{proof}
We first prove the first equivalence. Let $v,v'\in F$ and suppose that $\beta \in (B^N)^*$ satisfies $\beta(v)=1$ but $\beta(v')<1$ then for $\epsilon>0$ small enough, let $\beta_0=\beta-\epsilon\beta'\in F$. We have $\beta(v)>1$ which contradicts the fact that $\beta\in B^N$ and $v \in (B^N)^*$. 
We conclude that $\beta(v)=1$ is equivalent to $\beta'(v)=1$ for $v\in B^N$.

\medskip

We now prove the second characterization. Indeed $\overline{F^\vee}$ is equal to the convex hull of its extreme points, and its extreme points are extreme points of $(B^N)^*$.
\end{proof}

One can define similarly $(F')^\vee$ for a face $F'$ of $(B^N)^*\subset \mathfrak{a}^*$ as is the unique face such that: 
\begin{equation}\label{eqn: def of F'vee}
        \overline{(F')^\vee} \coloneqq \{ v \in B^N \mid \forall \beta \in F', \beta(v)=1 \}.
    \end{equation} 

\begin{prop}
\label{prop:Bijection vee}
    Let $N$ be a polyhedral norm. The map $F\mapsto F^\vee$ is a bijection between the faces of $B^N$ and those of $(B^N)^*$ with inverse $F'\mapsto (F')^\vee$. A face of dimension $p$ of $B^N$ is mapped to a face of codimension $p+1$ of $(B^N)^*$. 
 \end{prop}

 \begin{proof}

  Using the definition, one immediately gets that $\overline{F}\subset \overline{(F^\vee)^\vee}$. For the converse inclusion we just need to check that the codimension of $F^\vee$ is $p+1$ if the dimension of $F$ is $p$. With a similar argument we obtain that the dimension of $(F^\vee)^\vee$ is also $p$.

  \medskip

  For a polyhedral norm, faces $F$ can be written as $\lbrace v\in \mathfrak{a}\mid \forall \omega\in I, \, \omega(v)=1;\; \forall \omega\in J, \omega(v)\leq 1\rbrace$ where $I\cup J$ is a partition of $\ver((B^N)^*)$. Moreover one has $I=\ver(F^\vee)$. Hence, the dimension of $F$ is equal to the codimension of the span of $\omega\in \ver(F^\vee)$, which is equal to the codimension of $F^\vee$ minus one. 
 \end{proof}

\begin{rem}
\label{rem:smoothsquare norm}
    For general norms on vector spaces, $(F^\vee)^\vee$ may have larger dimension than $F^\vee$, see Figure \ref{fig:smoothsquare norm}. These raises technical difficulties in the general case. This is part of the reason why we restrict ourselves to polyhedral norms.
\end{rem}
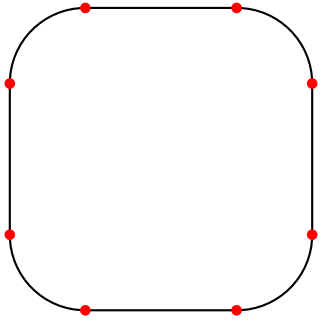
\begin{figure}[h]
    \begin{center}
    \begin{tikzpicture}

    \def\side{4}     
    \def\rad{1}     
    \def\offset{3}

    \draw[thick] (0,0) [rounded corners=\rad cm] -- (\side,0) -- (\side,\side) -- (0,\side) -- cycle;
    
    \foreach \x/\y in {
        \rad/0, \offset/0,           
        \side/\rad, \side/\offset,   
        \offset/\side, \rad/\side,   
        0/\offset, 0/\rad            
    } {
        \fill[red] (\x,\y) circle (2pt);
    }

\end{tikzpicture}
\end{center}
\caption{Unit ball of a non-polyhedral Finsler metric on $\R^2$ with $8$ marked points on its boundary corresponding to faces $F$ such that $(F^\vee)^\vee\neq F$. For these faces $F^\vee$ is a point and is equal to $F_0^\vee$ where $F_0$ is the adjacent edge.}
\label{fig:smoothsquare norm}
\end{figure}

We introduce standard horofunctions. This will allow us to understand all horofunctions for the norm $N$ when $N$ is a polyhedral norm.

\begin{definition}\label{def: standard horofunction}
    Let $F$ be a face of $B^N$ and let $f \in F$ be any point (recall that, by convention, $F$ is open in its affine span). 
    The \emph{standard horofunction associated to $F$} is
        \[h_F(x) = \lim_{t \to \infty} N(tf-x) -t .\]
\end{definition}

Let us write $\ver F$ the set of vertices of $F$, \ie faces of dimension $0$, or in other words the set of extreme points of its closure.

\begin{prop}
\label{prop:Formula horofunction standard}
The standard horofunction is independent of the choice of $f\in F$ and is also given by:
\begin{equation}
    \label{eq:Formula standard horofunctions}
        h_F(x) = \sup_{\beta\in F^\vee} -\beta(x)=\max_{\omega\in \ver(\overline{F^\vee})} -\omega(x)
\end{equation}  
\end{prop}

In particular, the standard horofunction does not depend on the choice of $f$ in $F$. 

\begin{proof}
Here we use \eqref{eqn: norm from the dual ball} to write $h_F(x)$ as the limit for $t\to \infty$ of:
$$\max_{\beta\in (B^N)^*}\beta(tf-x)-t.$$
For each $t>0$ let $\beta_t\in (B^N)^*$ be a linear form at which the maximum is reached. Let $f\in F$. Then for all $\beta\in F^\vee$ :
$$ t(\beta_t(f)-1)-\beta_t(x)=\beta_t(tf-x)-t\geq \beta(tf-x)-t=-\beta(x). $$
Since $\beta_t(x)$ is bounded by $N(x)$ and $(\beta_t(\tau)-1)\leq 0$, then $\beta_t(\tau)=1-o(t)$. Therefore if we chose $\beta_\infty$ to be any limit point of $(\beta_t)$, $\beta_\infty(f)=1$ and $h_F(x)=-\beta_\infty(x)$. Note moreover that for any $\beta\in F$ we have $h_F(x)\geq -\beta(x)$. Hence $h_F(x)=\max_{\beta\in \overline{F^\vee}} -\beta(x)$.
\end{proof}

We may now recall the following general result applied to polyhedral norms:
\begin{theorem}[{\cite[Theorem 4.7]{JS2024polyhedral}}]
\label{thm: Ji-Schilling}
    The horofunction compactification $\overline{\mathfrak{a}}^N$ of $(\mathfrak{a},N)$ for a polyhedral metric is homeomorphic to the dual of the unit ball $(B^N)^\ast$, by a homeomorphism identifying the translates of $[h_F]\in \overline{\mathfrak{a}}^N$ with the face $F$ of the dual unit ball of $N$.
\end{theorem}

\begin{figure}[h]
\begin{center}
\begin{tikzpicture}[>=Stealth, thick, scale=.6]

    \coordinate (A) at (90:2.5);
    \coordinate (B) at (210:2.5);
    \coordinate (C) at (330:2.5);
    \draw (30:2.5) -- (150:2.5) -- (270:2.5)-- cycle;
    \fill (0,0) circle (1.5pt);
    \draw[dotted] (A) -- (B) -- (C) -- cycle;

    \coordinate (BC1) at ($(B)!0.5!(C)$);

    \coordinate (AC1) at ($(A)!0.5!(C)$);

    \coordinate (AB1) at ($(A)!0.5!(B)$);

    \fill[blue] (B) circle (2.5pt);
    \draw[blue] ($(B)+(210:1)+(300:1)$) -- ($(B)+(210:1)+(120:1)$);
    \draw[blue, ->, dotted] ($(B)+(210:1)$) -- ($(B)+(210:1.5)$);

    \fill[red] (C) circle (2.5pt);
    \draw[red] ($(C)+(330:1)+(60:1)$) -- ($(C)+(330:1)+(240:1)$);
    \draw[red, ->, dotted] ($(C)+(330:1)$) -- ($(C)+(330:1.5)$);

    \fill[olive] (A) circle (2.5pt);
    \draw[olive] ($(A)+(90:1)+(180:1)$) -- ($(A)+(90:1)+(0:1)$);
    \draw[olive, ->, dotted] ($(A)+(90:1)$) -- ($(A)+(90:1.5)$);
    
    \fill[blue!50!red] (BC1) circle (2.5pt);
    \draw[blue!50!red] ($(BC1)+(270:2.0)+(120:1)$) -- ($(BC1)+(270:2.0)$) --($(BC1)+(270:2.0)+(60:1)$);
    \draw[blue!50!red, ->, dotted] ($(BC1)+(270:2.0)$) -- ($(BC1)+(270:2.5)$);
    
        \fill[olive!50!blue] (AB1) circle (2.5pt);
    \draw[olive!50!blue] ($(AB1)+(150:2.0)+(0:1)$) -- ($(AB1)+(150:2.0)$) --($(AB1)+(150:2.0)+(300:1)$);
    \draw[olive!50!blue, ->, dotted] ($(AB1)+(150:2.0)$) -- ($(AB1)+(150:2.5)$);

        \fill[olive!50!red] (AC1) circle (2.5pt);
    \draw[olive!50!red] ($(AC1)+(30:2.0)+(240:1)$) -- ($(AC1)+(30:2.0)$) --($(AC1)+(30:2.0)+(180:1)$);
    \draw[olive!50!red, ->, dotted] ($(AC1)+(30:2.0)$) -- ($(AC1)+(30:2.5)$);

\end{tikzpicture}\hspace{1cm}\begin{tikzpicture}[>=Stealth, thick, scale=.6]

    \coordinate (A) at (90:2.5);
    \coordinate (B) at (210:2.5);
    \coordinate (C) at (330:2.5);
    \fill (0,0) circle (1.5pt);
    \draw[dotted] (A) -- (B) -- (C) -- cycle;

    \coordinate (BC1) at ($(B)!0.3!(C)$);
    \coordinate (BC2) at ($(B)!0.7!(C)$);

    \fill[blue] (B) circle (2.5pt);
    \draw[blue] ($(B)+(210:1)+(300:1)$) -- ($(B)+(210:1)+(120:1)$);
    \draw[blue] ($(300:2.1)$) -- ($(120:2.1)$);
    \draw[blue, ->, dotted] ($(B)+(210:1)$) -- ($(B)+(210:1.5)$);

    \fill[red] (C) circle (2.5pt);
    \draw[red] ($(C)+(330:1)+(60:1)$) -- ($(C)+(330:1)+(240:1)$);
    \draw[red] ($(60:2.1)$) -- ($(240:2.1)$);
    \draw[red, ->, dotted] ($(C)+(330:1)$) -- ($(C)+(330:1.5)$);

    \fill[blue!30!red] ($(B)!0.6!(C)$) circle (2.5pt);
    \draw[blue!30!red] ($(BC1)+(270:1.7)+(120:1)$) -- ($(BC1)+(270:1.7)$) --($(BC1)+(270:1.7)+(60:1)$);
    \draw[blue!30!red] ($(-.1,0)+(240:1.5)+(120:3.5)$) -- ($(-.1,0)+(240:1.5)$) --($(-.1,0)+(240:1.5)+(60:3.5)$);
    \draw[blue!30!red, ->, dotted] ($(BC1)+(270:1.7)$) -- ($(BC1)+(270:2.2)$);
    
    \fill[blue!70!red] ($(B)!0.4!(C)$) circle (2.5pt);
    \draw[blue!70!red] ($(BC2)+(270:1.7)+(120:1)$) -- ($(BC2)+(270:1.7)$) --($(BC2)+(270:1.7)+(60:1)$);
    \draw[blue!70!red] ($(.1,0)+(300:1.5)+(120:3.5)$) -- ($(.1,0)+(300:1.5)$) --($(.1,0)+(300:1.5)+(60:3.5)$);
    \draw[blue!70!red, ->, dotted] ($(BC2)+(270:1.7)$) -- ($(BC2)+(270:2.2)$);

\end{tikzpicture}
    \end{center}
    \caption{ Illustration of the horofunction compactification and the behavior of horofunctions for a metric whose unit ball is a triangle. The dual unit ball  for the metric is drawn as a dotted triangle. We illustrate for each point on this dual unit ball a part of a level set of the corresponding horofunction defined up to a constant, with an arrow pointing towards a directions on which the horofunction is increasing. In the second picture we additionally draw a level curve passing close to the origin, to visualise the convergence of the horofunctions up to a constant.}
    \label{fig:Horofunction compactification}
\end{figure}
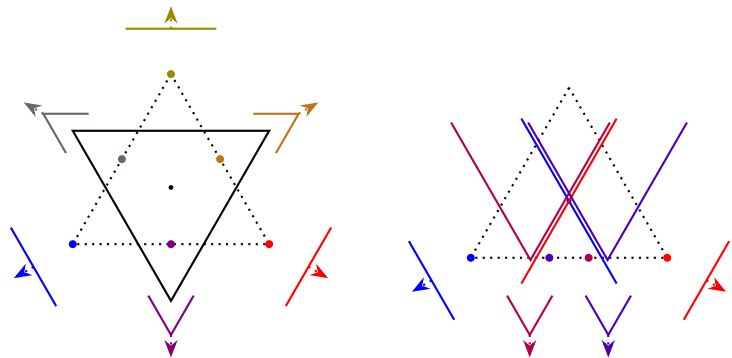

In particular, for a polyhedral Finsler metric each horofunction is the limit point of some Euclidean geodesic ray. Note that Euclidean geodesic rays are geodesic rays in $(\mathfrak{a},N)$ up to rescaling the speed. Polyhedral metrics $(\mathfrak{a},N)$ are never uniquely geodesic. 

It will be useful to understand precisely when two translates of standard horofunctions agree. For that purpose we introduce the following vector subspaces associated to a face $F$ of the unit ball of $N$.

\begin{defn}\label{defn:Lin and Lin_0}

Let $\Span(F)$ denote the linear span, i.e.\ the smallest vector subspace of $\mathfrak{a}$ containing $F$.

Let $\Span_0(F)$ denote the linear part of the affine span of $F$, or equivalently the linear span of differences $f-f'$ for $f,f'\in F$.

\medskip 

Dually we define $\Span(F^\vee)$ as the linear span of $\mathfrak{a}$ containing $F^\vee$ and $\Span_0(F^\vee)$ the linear part of the affine span of $F^\vee$. For a vector subspace $V\subset \mathfrak{a}$ we write $V^\vee=\lbrace \beta\in \mathfrak{a}^*\mid \forall v\in V,\, \beta(v)=0\rbrace$.
\end{defn}

For polyhedral metrics these subspaces have the following property:

\begin{prop}\label{prop:basic properties Lin}
Let $F$ be a face of $B^N$, where $N$ is a polyhedral norm. We have $\Span(F^\vee)=(\Span_0(F))^\vee$ and $\Span_0(F^\vee)=(\Span(F))^\vee$.
\end{prop}

\begin{proof}
Let us first prove that in general, for any norm $N$, $\Span(F)\subset \Span_0(F^\vee)^\vee$. Let $f\in F$ and $\alpha, \beta\in F^\vee$ so that $\alpha-\beta \in \Span_0(F^\vee)$. By definition of $F^\vee$, $\alpha(f)=\beta(f)=1$ so $(\alpha-\beta)(f)=0$. Hence $f\in \Span_0(F^\vee)^\vee$, and $\Span(F)\subset\Span_0(F^\vee)^*$.

Similarly one has for any norm $N$ that $\Span_0(F)\subset \Span(F^\vee)^\vee$. Let $f,g\in F$ so $f-g\in \Span_0(F)$ and $\alpha\in F^\vee$. By definition of $F^\vee$, $\alpha(f)=\alpha(g)=1$ so $\alpha(f-g)=0$. Hence $f-g\in \Span(F^\vee)^\vee$, and $\Span_0(F)\subset\Span(F^\vee)^\vee$.

\medskip

In order to check the equality, it suffices to compute that the dimensions coincide. As in the proof of Proposition \ref{prop:Bijection vee}, if $F$ has dimension $p$, $F^\vee$ has codimension $p+1$, so the dimension of $\Span_0(F)$ and the codimension of $\Span(F^\vee)$ coincide. Similarly the dimension of $\Span(F)$ and the codimension of $\Span_0(F^\vee)$ coincide.
\end{proof}

As we mentioned, these subspaces have the following property with respect to the standard horofunctions.

\begin{lemma}\label{lem:equivalence of translated horofunctions}
    Let $F$ be a face of $B^N$ and let $v \in \mathfrak{a}$.
    The standard horofunction $x\mapsto h_F(x)$ coincides on $\mathfrak{a}$ with $x\mapsto h_F(x+v)$, if and only if $v\in \Span_0(F)$. 
    These two functions differ by a constant function if and only if $a\in \Span(F)$.
    
    \medskip
    
    Let $F,F'$ be two faces of the unit ball for $N$. 
    The difference $h_F-h_{F'}$ of the associated standard horofunctions is bounded on $\mathfrak{a}$ if and only if $F=F'$.
\end{lemma}

\begin{proof}
Let $f\in F$, we have by Proposition \ref{prop:basic properties Lin}:
$$\Span_0(F)=\Span(F^\vee)^*=\bigcap_{\beta\in F^\vee} \Ker(\beta).$$
Clearly translations in this subspace leave invariant the standard horofunction $h_F$ given by equation \eqref{eq:Formula standard horofunctions}.
Note that $\langle f\rangle \oplus \Span_0(F)=\Span(F)$.
For all $x\in \mathfrak{a}$ one has $h_F(x)-h_F(tf+x)=t$, so translating $h_F$ by an element in $\Span(F)$ only changes the horofunction by a constant.

\medskip

Conversely suppose that translating by $v\in \mathfrak{a}$ changes the horofunction by a constant. Up to translating $v$ by a multiple of $f$ with $f\in F$ we can assume that $h_F(x)=h_F(x+v)$ for all $x\in \mathfrak{a}$. In particular since $h_F(0)=0$, one has $h_F(v)=h_F(-v)=0$ so in particulat by \eqref{eq:Formula standard horofunctions}, for all $\beta \in \overline{F}$, $\beta(v)\leq 0$ and $\beta(-v)\leq 0$ so $\beta(v)=0$ and hence $v\in \Span_0(F)$.

\medskip 

Now for the second point, note that for $v\in \mathfrak{a}$ and $t\in \R$, $h_{F}(tv)=th_F(v)$, and the same holds for $F'$, hence $h_F(v)-h_{F'}(v)=\frac{1}{t}(h_F(tv)-h_{F'}(tv))$ which converges to $0$ when $t$ goes to $+\infty$, meaning that $h_F=h_{F'}$ if the difference is bounded. If we still have $F\neq F'$ then there is up to exchanging $F$ and $F'$ a $\beta\in F^\vee$ such that $\beta\notin \overline{F'^\vee}$. Since $\beta \in F^\vee$ then $h_F\geq -\beta$, but since $\beta\notin \overline{F'^\vee}$, one can find $v\in \mathfrak{a}$ such that $-\beta(v)> 0$ and $-\beta'(v)< 0$ for all $\beta'\in \overline{F'^\vee}$, meaning that $h_{F'}(v)<-\beta(v)$. Hence $h_{F}\neq h_{F'}$. 
\end{proof}

\subsection{The stratification of the space of horofunctions}

Let $N$ be a $W$-invariant norm on $\mathfrak{a}$ and let $d^N$ denote the distance on $\X$ induced by the corresponding polyhedral Finsler metric. 

In this case, we have the following theorem relating the horofunction compactification of $(\X,d^N)$ to the horofunction compactification of $(\mathfrak{a},N)$.

\begin{theorem}[{\cite{haettel2018horofunctioncompactificationssymmetricspaces}}]
\label{thm:Comparaison Horofunction Flat and symmetric space HSWW}
    The horofunction compactification of $(\mathfrak{a},N)$ embeds into the horofunction compactification of $(\X,d^N)$, and every horofunction in $\X$ up the action of $G$ lies in the image of this embedding. 
\end{theorem}

In particular by Theorem \ref{thm:Comparaison Horofunction Flat and symmetric space HSWW} every horofunction $h$ on $\X$ is a Busemann point.
A more precise result follows from Theorem \ref{thm: Ji-Schilling}: for any horofunction $h$ there exists a geodesic ray $\gamma \colon \R\to \X$ for the Riemannian metric on $\X$, which is also a geodesic ray for $d^N$ up to a rescaling the speed by a constant, such that for all $x\in \X$:
$$ h(x)=\lim_{t\to +\infty} d^N(x,\gamma(t))-d^N(\gamma(0),\gamma(t)) .$$ 

\medskip

As a consequence, the stratification of the set of horofunctions $\mathcal{H}^N(\X)$ into $G$-orbits is reduced to $\mathcal{H}^N(\mathfrak{a})$ modulo the affine Weyl group $W \ltimes \mathfrak{a}$.

Let us set some notation.
For a face $F$ of the unit ball $B^N$ let $\mathcal{O}_F$ be the set of horofunctions in the $G$-orbit of the standard one $h_F\in \mathcal{H}^N \mathfrak{a}\subset \mathcal{H}^N\X$. 
We write $\text{Faces}(N)$ for the set of faces of $B^N$ up to the action of the Weyl group.
Then the previous discussion amounts to the following:

\begin{prop}
\label{prop:Description space Horofunctions}
    The set $\mathcal{H}^N(\X)$ of all horofunctions on $\X$ for the metric $d^N$ decomposes into $G$-orbits:
    $$\mathcal{H}^N(\X)=\bigcup_{[F]\in \text{Faces}(N)}\mathcal{O}_F .$$

    Similarly, the horoboundary decomposes as 
        \[ \partial^N(\X) = \bigcup_{[F]\in \text{Faces}(N)}\mathcal{O}_F/\R . \]
\end{prop}

We note that there are finitely many such orbits precisely when $N$ is polyhedral. 

\begin{remark}
    For any invariant norm $N$ and orbit $\mathcal{O}_F$ of horofunctions of $d^N$, there exists a polyhedral norm $N'$ and a corresponding face $F'$ so that $\mathcal{O}_F \simeq \mathcal{O}_{F'}$, but the global structure of the compactifications is different.
\end{remark}

Among all faces of  the unit ball $B^N$, one face plays a particularly important role.
Recall that $\rho\in \mathfrak{a}^*$ denotes the half sum of all positive roots with multiplicity. 
We define $F_\rho$ to be the face of $B^N$ such that the maximum of $\rho$ on $B^N$ is reached exactly on $\overline{F_\rho}$. 
One can check that one has the following other characterization:

\begin{prop}
    The face $F_\rho$ is characterized by the fact that for some $\lambda>0$, $\lambda\rho\in F_\rho^\vee$. 
\end{prop}

\subsection{The stabilizer of a horofunction}
\label{subsec:The stabilizer of a horofunction}

Let $R_F$ be the stabilizer in $G$ of the standard horofunction $h_F\in\mathcal{O}_F \subset \mathcal{H}^N(X)$, and let $S_F$ be the stabilizer of this $[h_F]$, which is $h_F$ up to a constant, i.e.\ $S_F$ is the stabilizer of the corresponding point in $\partial^N \X$.
In this subsection we give a precise Lie theoretic description of $R_F$ and $S_F$.
By the preceding subsection this is sufficient to determine the stabilizer of an arbitrary horofunction. 
The Lie theory involved is somewhat technical: we encourage the reader to refer to the example in Section \ref{subsec:Horofunctions for SLn} that illustrates all the objects we introduce.

\medskip

First we recall some notations.
We let $\mathfrak{a}$ denote a fixed choice of Cartan subspace, endowed with a $W$-invariant asymmetric norm $N$.
We let $B^N$ denote the unit ball, and $(B^N)^\ast$ denotes the dual unit ball. 
For a face $F$ of $B^N$, recall that the linear subspaces ${\Lin_0}(F)$ and ${\Lin}(F)$ were defined in Definition \ref{defn:Lin and Lin_0}. In the case the metric is polyhedral, $\Lin(F)$ is the span of $F$ and $\Lin_0$ is the linear part of the affine span of $F$.
There is a set of simple restricted roots $\Delta$ and corresponding positive Weyl chamber $\mathfrak{a}^+$. 
We recall that $Faces(N)$ denotes the faces of $B^N$ modulo the action of $W$. 
Each element of $Faces(N)$ is represented by an element $F$ intersecting $\mathfrak{a}^+$.

Associated to a face $F$ meeting $\mathfrak{a}^+$, we introduce some distinguished subsets of simple roots.
Let $\Theta_F^{\stable}$ be the set of simple roots $\alpha\in \Delta$ such that the associated reflection in the Weyl group $\sigma_\alpha$ stabilizes $F$. 

\medskip
For every such $\alpha$, the symmetry $\sigma_\alpha$ either \emph{fixes} $F^\vee$ pointwise, or it \emph{flips} $F^\vee$, idem for $F$. 
When the metric is polyhedral, $F$ is fixed if and only if $F^\vee$ is flipped and vice-versa. 
One goal is to relate the horofunction compactification to flag manifolds, which means associated to each face $F$ of $B^N$ the appropriate subset of simple roots. 
We need to keep track not only of the walls $\ker \alpha$ which meet the face $F$, but also whether the reflection in that wall fixes the face or flips it. 
This motivates the following definition due to Guichard, Kassel, and the third author \cite{GuichardKasselWienhard}.

\begin{defn}
    The set of simple roots whose associated symmetry fixes $F$ (and hence $F^\vee$) will be denoted by $\Theta_F^\text{stable}$. This set decomposes as :
    $$\Theta_F^{\stable}=\Theta_F^{\fixed}\sqcup \Theta_F^{\flipped}.$$
    Here $\Theta_F^{\fixed}$ is the set of simple roots $\alpha\in \Theta_F^{\stable}$ such that $\alpha \in \Span_0(F^\vee)$ and $\Theta_F^{\flipped}$ the set of simple roots $\alpha\in \Theta_F^{\stable}$ with $\alpha^\vee\in \Span_0(F)$. 
\end{defn}

Note that indeed the orthogonal reflection associated to a root $\alpha \in \mathfrak{a}^*$ preserves $F$ if and only if either $\alpha \in \Span_0(F^\vee)$, in which case $F$ and $\Lin(F)$ are \emph{fixed} pointwise by the reflection, or $\alpha^\vee \in \Span_0(F)$ in which case the reflection has a fixed point in $F$, and it \emph{flips} the face. 
Note here that since $\alpha$ is a root, the reflection is in the Weyl group and necessarily preserves the unit ball $B^N$. 

\begin{rem}
We will see in Proposition \ref{prop: Existence measure} and Corollary \ref{cor: O_F_rho has invariant measure} that $F_\rho$, the face that we will be mostly interested in, as well as its subfaces, always satisfy $\Theta_{F_\rho}^{\flipped}=\emptyset$. 
Hence the following discussion becomes simpler for this orbit of horofunctions.
\end{rem}

\begin{lemma}\label{lem:flipping roots orthgonal to fixing roots}
    Each root in $\Theta_F^{\fixed}$ is orthogonal to each root in  $\Theta_F^{\flipped}$.
\end{lemma}

\begin{proof}
    For $\beta \in \Theta_F^{flipped}$ and $\alpha \in \Theta_F^{fixed}$ it holds that $\beta^\vee \in \Lin_0(F) \subset \Lin (F) \subset \ker(\alpha)$. Hence $\beta$ and $\alpha$ are orthogonal. 
\end{proof}

\medskip
This decomposition induces two subgroups of the Levi subgroup $L_F$ associated with the parabolic subgroup $P_{\Delta\setminus \Theta_F^{stable}}$.

\medskip

Let us denote by $\mathfrak{g}_F^{fixed},\mathfrak{g}_F^{flip}\subset \mathfrak{g}$ the Lie subalgebras generated respectively by $\mathfrak{g}_{\alpha}\oplus \mathfrak{g}_{-\alpha}$ for $\alpha$ in $\Theta_F^{\rm fixed}$, respectively $\alpha$ in $\Theta_F^{\rm flip}$. 
For some $\mathfrak{m}_F^{fixed},\mathfrak{m}_F^{\flipped}\subset \mathfrak{m}$ and $\mathfrak{a}_F^{fixed},\mathfrak{a}_F^{\flipped}\subset \mathfrak{a}$:

\begin{equation}\label{eq:Definition Gfixed flip}
\begin{split}
&\mathfrak{g}_F^{fixed}=\mathfrak{a}_F^{fixed}\oplus \mathfrak{m}_F^{fixed}\oplus \bigoplus_{\alpha \in Span(\Theta^{\fixed}_F)} \mathfrak{g}_{\alpha}\oplus \mathfrak{g}_{-\alpha}, \\
&\mathfrak{g}_F^{\flipped}=\mathfrak{a}_F^{\flipped}\oplus \mathfrak{m}_F^{\flipped}\oplus \bigoplus_{\alpha \in Span(\Theta^{\flipped}_F)} \mathfrak{g}_{\alpha}\oplus \mathfrak{g}_{-\alpha}.
\end{split}\end{equation}

We have the following :

\begin{prop}
    \label{prop:a is span of coroots}
    The Lie subalgebras $\mathfrak{a}^{\fixed}_F$ and $\mathfrak{a}^{\flipped}_F$ are respectively the spans of the co-roots $\alpha^\vee$ associated to $\alpha\in \Theta^{\fixed}_F$, repsectively $\alpha\in \Theta^{\flipped}_F$.
\end{prop}

\begin{proof}
Note that for all $\alpha\in \Sigma^+$, $\alpha^\vee \in[\mathfrak{g}_\alpha,\mathfrak{g}_{-\alpha}]$, so the spans are contained in $\mathfrak{a}^{\fixed}_F$ and $\mathfrak{a}^{\flipped}_F$ respectively. 

Moreover $[\mathfrak{g}_\alpha,\mathfrak{g}_{-\alpha}]\subset \langle \alpha^\vee \rangle \oplus \mathfrak{m}$. 
In particular, we see that in the expression \eqref{eq:Definition Gfixed flip}, replacing $\mathfrak{a}^{\fixed}_F$ and $\mathfrak{a}^{\flipped}_F$ by the respective spans of coroots still gives a Lie subalgebra. 
Hence we have the desired result.
\end{proof}

These Lie algebras are preserved by the Cartan involution associated to the basepoint: indeed each is generated by a subset that is preserved by the Cartan involution.
As a consequence these Lie algebras are reductive. 

Let $G^{\fixed}_F,G^{\flipped}_F\subset G$ be the connected subgroups of $G$ associated to the Lie subalgebras $\mathfrak{g}^{\fixed}_F,\mathfrak{g}^{\flipped}_F$ of $\mathfrak{g}$. 
To explain their structure more precisely we need a few preliminary results.

\begin{lemma}\label{lem: tau in ker flipped}
    There exists $\tau \in F$ so that \( \Theta_F^{\stable} = \{ \alpha \in \Delta : \alpha(\tau)=0 \} .\)
\end{lemma}

\begin{proof}
    Any $\tau \in F$ satisfies $\alpha(\tau)>0$ when $\alpha \in \Delta \setminus \Theta_F^{\stable}$ and $\alpha(\tau)=0$ when $\alpha \in \Theta_F^{\fixed}$.
    Replace any $\tau \in F$ by its average $\tau'$ by the subgroup of $W$ stabilizing $F$, or equivalently by its average under the Weyl subgroup corresponding to $\Theta_F^{\flipped}$. 
    Then $\alpha(\tau') = 0$ for all $\alpha \in \Theta_F^{\flipped}$.
\end{proof}

Now we discuss how these parabolic subgroups are related to the stabilisers $R_F$ of the standard horofunction associated with $F$ and the stabilizer $S_F$ of the standard horofunction up to constant.

\begin{prop}

The subgroups $R_F\subset S_F$ are contained in the parabolic subgroup $P_{\Delta\setminus \Theta_F^{\stable}}$ and contain the horospherical subgroup $U_{\Delta\setminus \Theta_F^{\stable}}$.

\end{prop}

\begin{proof}
    Write $h_F$ as the limit of distance functions along $\gamma_\tau (t) = e^{t \tau}K$, with $\tau$ from Lemma \ref{lem: tau in ker flipped}, normalized to be zero at the basepoint.
    
    We first show that $U_{\Delta\setminus \Theta_F^{\stable}} $ is contained in $R_F$. 
    Note that for any $n \in U_{\Delta\setminus \Theta_F^{\stable}}$, the Riemannian geodesic ray $\gamma = \gamma_\tau$ is strongly asymptotic to its translate $n\cdot \gamma$, because $\alpha(\tau)=0$ for all $\alpha\in\Theta_F^{\stable}$.
    Being strongly asymptotic means that:
    $$ \abs{d^N(x,n \gamma(t)) - d^N(x,\gamma(t))} \to 0.$$
    Hence the limit of distance functions along $\gamma$ and $n\cdot \gamma$ are equal, and in particular $h_F \circ n = h_F$.

    \medskip
    
    Now let $g$ be an element of $S_F$, and suppose that it does not belong to the parabolic $P_{\Delta\setminus \Theta_F^{\stable}}$. 
    Hence $[\gamma]$ and $[g \gamma]$ are distinct points of $\partial_{vis}\X$. 
    Let $A$ be a maximal flat containing both  $[\gamma]$ and $[g \gamma]$ in its visual boundary, and let $\gamma'(t) = e^{t \tau'}K$ be the geodesic ray in the standard flat asymptotic to $g \gamma$ with $\gamma'(0)=K$.
    Choose $g' \in G$ fixing $[\gamma]$, taking the standard maximal flat to $A$, and satisfying $g' \gamma' = g \gamma$.
    
    Since $g' \gamma \sim \gamma$, their associated horofunctions $h_{g' \gamma}$ and $h_\gamma = h_F$ are at bounded distance.
    In particular $h_{g' \gamma}$ is at bounded distance from $h_{g \gamma} = h_F \circ g = h_F$ since $g \in S_F$. 
    These are translates of the horofunctions $h_{\gamma'}$ and $h_\gamma$ by $g'$, so those are also at bounded distance. 
    But since these correspond to distinct faces, this contradicts Lemma \ref{lem:equivalence of translated horofunctions}.
\end{proof}

Note that $\mathfrak{a}_F^{\fixed}\perp {\Lin}(F)$ since all the roots $\alpha\in \Theta^{\fixed}_F$ are in $\Span_0(F^\vee)$. 
Moreover $\mathfrak{a}_F^{\flipped}\subset {\Lin_0}(F)$ since all the coroots $\alpha^\vee$ for $\alpha\in \Theta^{\flipped}_F$ are in ${\Lin_0}(F)$. 

\begin{prop}\label{prop:comparing RF and SF to Gflip and Gfix}
    The subgroups $R_F\subset S_F$ contain $G^{\flipped}_F$ and the intersection of the maximal compact $K$ with $P_{\Delta\setminus \Theta^{\stable}_F}$. 
    The intersection of $R_F$ and $S_F$ with $G_F^{\fixed}$ is exactly the maximal compact of $G_F^{\fixed}$. 
\end{prop}

\begin{proof}

We again write $h_F$ as the Busemann point associated to $\tau$ from Lemma \ref{lem: tau in ker flipped}, and again take $\gamma_\tau$ the corresponding geodesic ray emanating from the basepoint. 
Then every element $g\in K\cap P_{\Delta\setminus \Theta^{\stable}_F}$ fixes the geodesic ray $\gamma_\tau$, and hence preserves $h_F$.

\medskip

Now thanks to the Cartan decomposition for $G^{\flipped}_F$ and $G^{\fixed}_F$, since the maximal compacts subgroups are in $R_F$ we just need to check that the restricted Cartan subgroup of $G^{\flipped}_F$ is in $R_F$ and that no non-trivial element of the restricted Cartan subgroup are in $S_F$.

\medskip

Let $\tau \in \mathfrak{a}^{\flipped}_F$, then it is a combination of co-roots  in $\Theta^{\flipped}_F$ by Proposition \ref{prop:a is span of coroots}, hence it belongs to $\Lin_0(F)$. 
In particular $h_F\circ e^{\tau}=h_F$ on the model flat by Lemma \ref{lem:equivalence of translated horofunctions}. 
Since $h_F$ is in the horoboundary of the model flat and since $e^\tau$ preserves the model flat, this implies that the horofunctions are equal on the whole symmetric space by Theorem \ref{thm:Comparaison Horofunction Flat and symmetric space HSWW}.

\medskip

Now let $\tau \in \mathfrak{a}^{\fixed}_F$. 
It is a combination of co-roots in $\Theta^{\fixed}_F$ by Proposition \ref{prop:a is span of coroots}. 
Lemma \ref{lem:flipping roots orthgonal to fixing roots} implies that $\tau$ is orthogonal to $\Lin(F)$. Hence by Lemma \ref{lem:equivalence of translated horofunctions} the translated horofunction $h\circ e^\tau$ is to equal up to a constant to $h$ on the model flat, so $e^\tau\notin S_F$.
\end{proof}

\begin{lemma}
\label{lem:Cartan subgroup of stabiliazer}
    We have $\Lie(R_F) \cap \mathfrak{a} = \Lin_0(F)$ and $\Lie(S_F) \cap \mathfrak{a} = \Lin(F)$.
\end{lemma}

\begin{proof}
    Applying Theorem \ref{thm:Comparaison Horofunction Flat and symmetric space HSWW}, the horofunctions $h_F$ and $h_F \circ e^\tau$ for $\tau\in \mathfrak{a}$ are equal if and only they coincide on the model flat. 
    The result then follows from Lemma \ref{lem:equivalence of translated horofunctions}.
\end{proof}

Now we can fully understand the orbits of the horofunction compactification. They are bundles over a flag manifold, with fiber a symmetric space $\X_F$.

\begin{prop}
\label{prop:Space horofunctions}
    The $G$-homogeneous space $\mathcal{O}_F$ fibers over the flag manifold $\mathcal{F}_\Theta=G/P_\Theta$ where $\Theta=\Delta\setminus \Theta_F^{\stable}$. The fiber at each point is a copy of $\X_F$, where $\X_F$ is the  symmetric space associated to the group $G^{\fixed}_F\times e^{V}$. Here $V\subset \mathfrak{a}$ is the orthogonal complement of $\mathfrak{a}^{\fixed}_F$ in $\Lin_0(F)^\perp$.
\end{prop}

For the space $G/S_F$ of horofunctions up to a constant the fiber is associated to the group generated by $G^{\fixed}_F\times e^{V'}$ where $V'$ is the orthogonal complement of $\mathfrak{a}^{\fixed}_F$ in $\Lin(F)^\perp$.

\begin{proof}
    We have $\mathcal{O}_F = G/R_F$ with $R_F$ in $P_\Theta$, so $G/R_F$ maps to $G/P_\Theta$.
    The fiber is $P_\Theta/R_F$.

    \medskip

    The parabolic group $P_\Theta$ is generated by  the groups $K\cap P_\Theta$, $U_\Theta$, $G^{\flipped}_F$, $G^{\fixed}_F$ and $e^\mathfrak{a}$. By Proposition \ref{prop:comparing RF and SF to Gflip and Gfix} we know that 
    $$ \left(K\cap P_\Theta\right)/\left(R_F\cap K\right), U_\Theta/\left(U_\Theta\cap R_F\right), \text{ and } G^{\flipped}_F/\left(G^{\flipped}_F\cap R_F\right) $$
    are trivial, and that $G^{\fixed}_F/\left(G^{\fixed}_F\cap R_F\right)=G^{\fixed}_F/\left(G^{\fixed}_F\cap K\right)$.
    
    We use Lemma \ref{lem:Cartan subgroup of stabiliazer} and again part of Proposition \ref{prop:comparing RF and SF to Gflip and Gfix} to see that $$\langle e^\mathfrak{a}, G^{\fixed}_F\rangle/\left(\langle e^\mathfrak{a}, G^{\fixed}_F\rangle\cap R_F\right)=\X_F. $$
    Hence $P_\Theta/R_F$ is naturally identified with $\X_F$.
\end{proof}

We now describe precisely the Lie algebras of $R_F$ and $S_F$.
We write 
$$ \mathfrak{k}_\alpha = \mathfrak{k} \cap (\mathfrak{g}_\alpha \oplus \mathfrak{g}_{-\alpha}), \quad  \mathfrak{p}_\alpha = \mathfrak{p} \cap (\mathfrak{g}_\alpha \oplus \mathfrak{g}_{-\alpha}).$$

Proposition \ref{prop:comparing RF and SF to Gflip and Gfix} and Lemma \ref{lem:Cartan subgroup of stabiliazer} imply directly the following:

\begin{prop}
\label{prop:Lie algebra of stabilizer}
$${\Lie}(R_F)={\Lin_0}(F)\oplus \mathfrak{m} \oplus \bigoplus_{\alpha \in \Sigma^+ \setminus {\Span}(\Theta_F^{\stable})}\mathfrak{g}_\alpha \bigoplus_{\alpha \in {\Span}(\Theta_F^{\fixed})}\mathfrak{k}_\alpha\bigoplus_{\alpha \in {\Span}(\Theta_F^{\flipped})}\mathfrak{g}_{\alpha}\oplus\mathfrak{g}_{-\alpha},$$
$${\Lie}(S_F)={\Lin}(F)\oplus \mathfrak{m} \oplus \bigoplus_{\alpha \in  \Sigma^+ \setminus {\Span}(\Theta_F^{\stable})}\mathfrak{g}_\alpha  \bigoplus_{\alpha \in {\Span}(\Theta_F^{\fixed})}\mathfrak{k}_\alpha\bigoplus_{\alpha \in {\Span}(\Theta_F^{\flipped})}\mathfrak{g}_{\alpha}\oplus \mathfrak{g}_{-\alpha}.$$

\end{prop}

\begin{rem}
    The roots in $\Theta_F^{\fixed}$ and $\Theta_F^{\flipped}$ are orthogonal (Lemma \ref{lem:flipping roots orthgonal to fixing roots}), hence every root in ${\Span}(\Theta_F^{\stable})$ belongs to exactly one of ${\Span}(\Theta_F^{\fixed})$ or ${\Span}(\Theta_F^{\flipped})$.
\end{rem}

\subsection{Relating horofunctions to Riemannian Busemann functions}

In order to study the bisector of two horofunctions, will need to consider configurations of a pair of horofunctions $h_1,h_2$ and a point $x$ of $\X$.
Since this triple may fail to all lie in (the closure of) a common maximal flat, we need to give a more invariant description of the horofunctions.

\medskip

We start by relating the polyhedral Finsler distance $d^N$ to certain Riemannian Busemann functions. 
For $o \in \X$ and $\xi \in \partial_{vis} \X$ we let $b_{o,\xi}$ denote the Riemannian Busemann function corresponding to the geodesic ray emanating from $o$ asymptotic to $\xi$. 

If $A$ is a maximal flat containing $o$ and $\xi$, then the restriction of $b_{o,\xi}$ to $A$ is linear and has gradient of norm $1$. 
On the other hand, if $\omega \in \mathfrak{a}^*$, is nonzero, then there is a unique ray emanating from the basepoint whose associated Busemann function restricts to $-\omega / \Abs{\omega}$ on $\mathfrak{a}$, identified with the model flat in $\X$ by the exponential map. 
In particular, there is a natural inclusion $\mathfrak{a}^\ast \setminus \{0\} \to \mathcal{H}^{Riem}(\X)$, the space of Riemannian Busemann functions. 
We will write $\omega \mapsto b_\omega$.
Note that the $G$-orbit of $b_\omega$ in $\mathcal{H}^{Riem}(\X)$ is independent of the choice of maximal flat and the action of the Weyl group.

We will restrict this map to vertices
of the dual unit ball $(B^N)^\ast$. 
Let us index the vertices of the dual unit ball $\ver((B^N)^\ast) = \{ \omega_i : i \in I \}$.

\begin{proposition}\label{prop: dN from busemanns}
    Let $o,x,y \in \X$. 
    Then:
        \[ d^N(x,y) = \max_{i \in I, \xi \in G \cdot [b_{\omega_i}]} \Abs{\omega_i}(b_{o,\xi}(x) - b_{o,\xi}(y)) .\]
    In particular by setting $o=x$ or $o=y$ we obtain:
       \[ d^N(x,y) = \max_{i \in I, \xi \in G \cdot [b_{\omega_i}]} -\Abs{\omega_i} b_{x,\xi}(y)  = \max_{i \in I, \xi \in G \cdot [b_{\omega_i}]} \Abs{\omega_i} b_{y,\xi}(x).\]
\end{proposition}

\begin{proof}
    Note that the expression is independent of $o \in \X$.
    Up to applying an isometry, we may assume that $x=o$ and $y=e^a o$ for $a \in \mathfrak{a}^*$, so that 
        \[ d^N(x,y) = N(\vec{d}(x,y)) = N(\vec{d}(o,e^a\cdot o)) = N(a) .\]
    If $a$ belongs to a face $F$ of $B^N$, then 
        \[ N(a) = \sup \left\{ \omega_i(a) : \omega_i \in vert(F^\vee) \right\} \]
    which is realized by a vertex $\omega_j \in vert(F^\vee)$ since the set of vertices is finite, so 
        \[ \omega_j(a) = \Abs{\omega_j}(b_{o,[\omega_j]}(o) - b_{o,[\omega_j]}(e^a \cdot o)) .\]

    For the other bound, note that each $\Abs{\omega_i}b_{o,\xi}$ for $\xi \in G \cdot [\omega_i]$ is a horofunction of $d^N$, and hence is $1$-Lipschitz with respect to $d^N$. 
\end{proof}

Our next goal is to write a horofunction of $d^N$ in terms of Riemannian Busemann functions. 
As we saw in the previous proof, it is possible to restrict considerations to the face $F$ containing the tangent vector from $x$ to $y$ and the vertices of its dual face. 

\medskip

Recall from Section \ref{subsec:The stabilizer of a horofunction} that $\Theta_F^{\stable}$ denotes the set of simple roots whose reflections preserve $F$, and set $\Theta = \Delta \setminus \Theta_F^{\stable}$, and recall that $P_\Theta$ denotes the corresponding parabolic subgroup. 
We also let $R_F$ denotes the stabilizer in $G$ of the standard horofunction $h_F$. 
Let $K_\Theta = K \cap P_\Theta$ be the intersection of the maximal compact and $P_\Theta$, or equivalently the maximal compact subgroup of the Levi subgroup $L_\Theta$. 
From Proposition \ref{prop:comparing RF and SF to Gflip and Gfix} we have $K_\Theta \subset R_F \subset P_\Theta$.

\begin{lem}
\label{lem:P Theta better description}
For every root $\alpha\in \Sigma$, every $f\in F$,  and every $\omega_i\in vert(F^\vee)$, if $\alpha(f)<0$ then $\langle \omega_i, \alpha\rangle=\omega_i(\alpha^\vee)\leq 0$, where $\alpha^\vee$ is the coroot of $\alpha$.
\end{lem}

\begin{proof}
   Suppose that for a given root $\alpha \in \Sigma$, $\alpha(f)<0$. 
   Since $N$ is a norm that is invariant with respect to the Weyl group, given $f\in F\cap \mathfrak{a}^+$, 
   its reflection
   \[ f-2\alpha (f)\alpha^\vee ,\]
   where $\alpha^\vee$ is the coroot of $\alpha$,
   is still inside the unit ball of $N$. 
   Hence 
   \[ \omega_i(f-2\alpha (f)\alpha^\vee)\leq 1=\omega_i(f).\] 
   Therefore $\alpha(f)\omega_i(\alpha^\vee)\geq 0$. Since $\alpha(f)<0$, $\omega_i(\alpha^\vee)\leq 0$.
\end{proof}

We introduce the following sets. First let $I_F\subset I$ be the set of indices corresponding to vertices of $F^\vee$.

 \begin{defnprop}
\label{defnprop:Definition E_i^F}
 Let $i\in I_F$ be so that $\omega_i\in vert(F^\vee)$, and  consider $[b_{\omega_i}]\in \partial^{Riem}\X$. 
 Let $E_i^F\subset \partial^{Riem}\X$ be the following set:
 $$E_i^F:=K_\Theta\cdot [b_{\omega_i}]=P_\Theta\cdot [b_{\omega_i}].$$
 \end{defnprop}

\begin{proof}
We have to show that the two definitions agree.
The compact group $K_\Theta$ intersects every connected component of $P_\Theta$, so 
it suffices to check that the orbits agree locally. 
    
The Lie algebra $\mathfrak{p}_\Theta$ of $P_\Theta$ is the sum of $\mathfrak{g}_0$ plus the sum of all root spaces $\mathfrak{g}_\alpha$ such that $\alpha(f)\geq 0$ for every $f\in F\cap \mathfrak{a}^+$. 
Let $\tau_i \in \mathfrak{a}$ be the vector orthogonal to $\Ker(\omega_i)$ so that $\omega_i(\tau)>0$ and $N(\tau_i)=1$. 
The Lie algebra $\mathfrak{p}_i$ of the stabilizer of $[b_{\omega_i}]$ is the sum of $\mathfrak{g}_0$ plus the sum of all root spaces $\mathfrak{g}_\alpha$ with $\alpha(\tau_i)\leq 0$.
Let us write $\mathfrak{m}_i$ for the sum of all root spaces such that $\alpha(\tau_i)>0$ so that $\mathfrak{g}=\mathfrak{p}_i\oplus \mathfrak{m}_i$.

\medskip

In particular, for every $v\in \mathfrak{p}_\Theta$, we write $v=v_1+v_2$ with $v_1\in \mathfrak{p}_i$ and $v_2\in \mathfrak{m}_i$. 
By Lemma \ref{lem:P Theta better description}  if $\alpha(f)<0$ then $\alpha(\tau_i)\leq0$, hence $v_2$ lies inside the sum of all roots spaces $\mathfrak{g}_\alpha$ such that $\alpha$ vanishes on $F$. 
This sum of root spaces is preserved by the Cartan involution $\theta$ corresponding to $\mathfrak{g} = \mathfrak{k} \oplus \mathfrak{p}$. 
The projection of $v$ to $\mathfrak{g}/\mathfrak{p_i}\simeq T_{[b_{\omega_i}]}\left(G\cdot [b_{\omega_i}]\right)$ is therefore equal to the projection of $\theta(v)$. 
Hence $T_{[b_{\omega_i}]}\left(P_\Theta\cdot [b_{\omega_i}]\right)=T_{[b_{\omega_i}]}\left(K_\Theta\cdot [b_{\omega_i}]\right)$. 
This is sufficient to see that the corresponding homogeneous spaces coincide.
\end{proof}

Inside the symmetric space $\X$ there is a totally geodesic sub-symmetric space consisting of points $o\in \X$ fixed by $K_\Theta$.
Recall from Theorem \ref{thm:Comparaison Horofunction Flat and symmetric space HSWW} that the standard horofunction $h_F$ extends uniquely to a horofunction on $\X$, which we continue to denote by $h_F$.

\begin{prop}
    \label{prop:Explicit formula for Horofunctions}
    Let $o\in \X$ be a point fixed by $K_\Theta$. Let $h_F\in \mathcal{O}_F$ be the model horofunction on $\X$, associated to the standard horofunction $h_F$ on $\mathfrak{a}$ via the embedding $\partial ^N\mathfrak{a}\subset \partial^N \X$, and with $h(o)=0$. We have the following:
\begin{equation}\label{eq:Formula horofunction in symmetric space}
h(x)=\max_{i\in I_F,\, p\in E^F_i}\lVert \omega_i\rVert b_{p,o}(x).\end{equation}
\end{prop}

\begin{proof}

Let $f\in F\cap \mathfrak{a}^+$. Let $[f]\in \partialvis \X$ be the class of the ray from $o$ with initial derivative $f$. By definition, $h$ is equal to the following limit, where $f\in F$:
$$h(x)=\lim_{t\to +\infty}d^N(x, exp_o(tf ))-d^N(o,exp_o(tf));$$
$$h(x)=\lim_{t\to +\infty}\max_{i\in I, \, p\in G\cdot [\omega_i]}\lVert \omega_i \rVert \left(b_{p,o}(x)-b_{p,o}(exp_o(f t))\right)-tN(f).$$

Recall that $N(f)=1$ since $f\in F$ is in a face the unit ball of $N$. We individually consider the functions $\phi_{i,p}(t)=-\lVert \omega_i \rVert b_{p,o}(exp_o(tf))-t$ for $i\in I_F$ and $p\in G\cdot [\omega_i]$. These functions are negative and decreasing: indeed $\lVert\omega_i\rVert b_{p,o}$ is $1$-Lipschitz for $d^N$ by Proposition \ref{prop: dN from busemanns}. Hence their limits are well defined in $\R\cup \lbrace -\infty\rbrace$.

\medskip

Morally we are now trying to prove that for this term to appear in the maximum in the limit, we need $[p]$ to be in the $E_i$ previously defined. Else we show that this function converges to $-\infty$ and hence cannot count in the maximum.

\medskip

Let $p\in G\cdot [\omega_i]$ for some $i\in I$. Let $P_\Theta$ be the same parabolic subgroup as before. We can find $g\in P_\Theta$ such that $g\cdot p=p_0\in \partialvis \mathfrak{a}^+\subset \partialvis \X$. Let $\pi_0\in \mathfrak{a}^+$ be the vector of unit norm for $N$ corresponding to $p_0\in \partialvis\mathfrak{a}^+$.  The function $\phi_{i,t}=-\lVert \omega_i\rVert b_{p_{0},g\cdot o}(exp_o(t\pi_0))-t$ is equal up to a constant to:
$$\widetilde{\phi}_{i,p}(t)=-\lVert \omega_i\rVert b_{p_{0},o}(exp_o(t\pi_0))-t=\omega_i(t\pi_0)-t$$ 
Here the constant bounding the difference depends only on the distance between $o$ and $g\cdot o$. 

The function $\phi_{i,p}$ converges uniformly to $-\infty$ away from the set of points $p_0$ such that the associated points satisfy $\pi_0=\tau_i$, where $\tau_i$ as previously is the vector of norm one for $N$ orthogonal to $\Ker(\omega_i)$ with $\omega_i(\tau_i)>0$. This property  $\pi_0=\tau_i$ occurs excatly when $i\in I_F$ and $p\in E^F_i$

\medskip

We write :
\begin{equation}
    \label{eq: formula horofunction inside the proof}
h(x)=\lim_{t\to +\infty}\max_{i\in I, \, p\in G\cdot [\omega_i]} \left( \lVert \omega_i \rVert b_{p,o}(x)+\phi_{i,p}(t)\right).\end{equation}

Let $(i_t)_{t\geq 0}$ and $(p_t)_{t\geq 0}$ be elements of $I$ and $p_t\in G\cdot [\omega_t]$ respectively such that the maximum in \eqref{eq: formula horofunction inside the proof} is reached for this value of $t\geq 0$. Up to taking a subsequence, one can assume that $i_t=i'$ for $t$ large enough, since $I$ is finite.

\medskip
Since $h(x)\neq -\infty$, we obtain that $i'\in I_F$ and $p_\infty\in E_{i'}^F=P_\Theta\cdot [\omega_{i'}]$ so that $\widetilde{\phi}_{i_t,p_t}(t)$ does not diverge to $-\infty$, as previously .

As a consequence we obtain, as in \ref{prop:Formula horofunction standard}: 

$$h(x)=\max_{i\in I_F, \, p\in E^F_i} \left( \lVert \omega_i \rVert b_{p,o}(x)+\lim_{t\to +\infty}\phi_{i,p}(t)\right).$$

By Proposition \ref{defnprop:Definition E_i^F}, if $p\in P_\Theta\cdot [\omega]$ then $p\in E_i^F=K_\Theta\cdot [\omega_i]$. Hence we can assume that $g\in K_\Theta$, and therefore $g\cdot o=o$ which implies $\phi_{i,p}=\widetilde{\phi}_{i,p}$. 
\end{proof}
    
\begin{remark}
    A consequence of the previous proposition is that in the formula for $d^N$ only the Riemannian Busemann functions for directions ``near'' the ray extending $xy$ need to be considered. 
\end{remark}

\begin{corollary}
    Let $x,y \in \X$.
    If the segment $xy$ meets a face $F$ of the unit ball in a maximal flat containing $xy$, then 
        \[ d^N(x,y) = \max_{ \omega_i \in vert(F^\vee), p \in E_i^F} \Abs{\omega_i} (b_{o,p}(x) - b_{o,p}(y)).\]
\end{corollary}

\subsection{Example  \texorpdfstring{$\X=\mathbb{H}^2 \times \mathbb{H}^2$}{X=H2*H2}}
\label{subsec:H2H2}
We now illustrate the notions introduced in the section with concrete examples, with $G=\PSL(2,\R)\times \PSL(2,\R)$. We keep the notations introduced in Section \ref{subsec:Symmetric spaces examples}.

Let $G=\PSL(2,\R)\times \PSL(2,\R)$, so that $\X=\mathbb{H}^2 \times \mathbb{H}^2$. We consider the following norms on $\mathfrak{a}^+\simeq \mathbb{R}^+\times \mathbb{R}^+$: 
$$N_\infty(x,y)=max(|x|,|y|),$$
$$N_1(x,y)=|x|+|y|,$$
$$N_{a,b}(x,y)=a|x|+b|y|.$$

In the third example we will take any real numbers $a>b>0$. The half sum of all positive roots $\rho$ is a positive multiple of the limear functional $(x,y)\mapsto x+y$. Therefore the face $F_\rho$ on the closure of which $\rho$ is maximum is a vertex for $N=N_\infty$ and $N=N_{a,b}$, but it is an edge for $N=N_1$. Recall that the Weyl group in this case is the group of tansformations $(x_1,x_2)\mapsto (\pm x_1,\pm x_2)$. We write $b_{\gamma,o}:\mathbb{H}^2\to \R$ for the Busemann function centered at $\gamma\in \partial \mathbb{H}^2$ which is equal to zero at $o\in \mathbb{H}^2$.

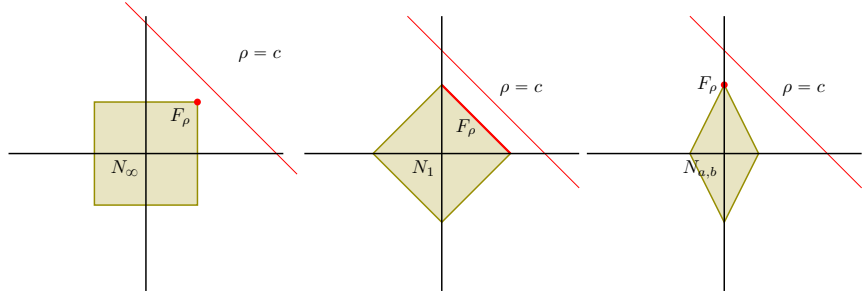
\begin{figure}[h!]
\begin{center}
\resizebox{0.9\textwidth}{!}{

\begin{tikzpicture}[scale=1.3]

    \draw[olive, fill=olive!20, thick] (.75,.75) -- (-.75,.75)--(-.75,-.75) --(.75,-.75)-- cycle;

    \draw[red] (2.2,-0.3) -- (-0.3,2.2);
    \fill[red] (.75,.75) circle (.05);

    \draw[thick] (2,0) -- (-2,0);
    \draw[thick] (0,2) -- (0,-2);
    \node[below left] (A) at (.75,.75) {$F_\rho$};
    \node[below left] (A) at (0,0) {$N_\infty$};
    \node[above right] (A) at (1.25,1.25) {$\rho=c$};
\end{tikzpicture}
\begin{tikzpicture}[scale=1.3]

    \draw[olive, fill=olive!20, thick] (1,0) -- (0,1)--(-1,0) --(0,-1)-- cycle;

    \draw[red] (2,-0.5) -- (-0.5,2);
    \draw[red, very thick] (1,0) -- (0,1);

    \draw[thick] (2,0) -- (-2,0);
    \draw[thick] (0,2) -- (0,-2);
    \node[below left] (A) at (.6,.6) {$F_\rho$};
    \node[below left] (A) at (0,0) {$N_1$};
    \node[above right] (A) at (.75,.75) {$\rho=c$};
\end{tikzpicture}
\begin{tikzpicture}[scale=1.3]
    \draw[olive, fill=olive!20, thick] (.5,0) -- (0,1)--(-.5,0) --(0,-1)-- cycle;

    \draw[red] (2,-0.5) -- (-0.5,2);
    \fill[red] (0,1) circle (.05);

    \draw[thick] (2,0) -- (-2,0);
    \draw[thick] (0,2) -- (0,-2);
    \node[left] (A) at (0,1) {$F_\rho$};
    \node[below left] (A) at (0,0) {$N_{a,b}$};
    \node[above right] (A) at (.75,.75) {$\rho=c$};
\end{tikzpicture}}
\end{center}
\caption{Illustration of the unit balls for the norms $N_\infty$, $N_1$ and $N_{a,b}$, together with the face $F_\rho$.}
\label{fig:Unit ball Norms H2}
\end{figure}

\medskip

For $N=N_\infty$, the unit ball has $3$ faces up to the Weyl group action: a vertex $F_\rho$ and two edges $F_1=\lbrace (1,x_2)\mid x_2\in [-1,1]\rbrace$ and $F_2=\lbrace (x_1,1)\mid x_1\in [-1,1]\rbrace$. The orbit of horofunctions $\mathcal{O}_{F_\rho}$ is exactly the space of functions $h:\mathbb{H}^2 \times \mathbb{H}^2\to \R$ of the form $h(p_1,p_2)=\max(b_{\alpha,o}(p_1),b_{\beta,o'}(p_2))$ for some $(\alpha,\beta)\in \partial \mathbb{H}^2 \times \partial \mathbb{H}^2$ and $(o,o')\in \mathbb{H}^2 \times \mathbb{H}^2$.  This  orbit $\mathcal{O}_{F_\rho}$ is a plane bundle over $\partial \mathbb{H}^2\times \partial \mathbb{H}^2$.

The other two orbits $\mathcal{O}_{F_1}$ and $\mathcal{O}_{F_2}$ consist respectively of horofunctions of the form $h(p_1,p_2)= b_{\alpha, o}(p_1)$ and $h(p_1,p_2)= b_{\alpha, o}(p_2)$ for $\alpha \in \partial \mathbb{H}^2$ and $o\in \mathbb{H}^2$. These orbits $\mathcal{O}_{F}$ are line bundles over $\partial \mathbb{H}^2$.

\medskip

For $N=N_1$, the unit ball has $3$ faces up to the Weyl group action: an edge $F_e=F_\rho$ and two vertices $F_1=\lbrace (1,0)\rbrace$, $F_2=\lbrace (0,1)\rbrace$. The orbit of horofunctions $\mathcal{O}_{F_\rho}$ is exactly the space of functions $h:\mathbb{H}^2 \times \mathbb{H}^2\to \R$ of the form $h(p_1,p_2)=b_{\alpha,o}(p_1)+b_{\beta,o}(p_2)$ for some $(\alpha,\beta)\in \partial \mathbb{H}^2 \times \partial \mathbb{H}^2$ and $o\in \mathbb{H}^2$. This  orbit $\mathcal{O}_{F_\rho}$ is a line bundle over $\partial \mathbb{H}^2\times \partial \mathbb{H}^2$.

The other two orbits $\mathcal{O}_{F_1}$ and $\mathcal{O}_{F_2}$ consist respectively of horofunctions of the form $h(p_1,p_2)= b_{\alpha, o}(p_1)+d(o',p_2)$ and $h(p_1,p_2)= d(o',p_1)+b_{\alpha, o}(p_2)$ for $\alpha\in \partial \mathbb{H}^2$ and $(o,o')\in \mathbb{H}^2 \times \mathbb{H}^2$. These orbits $\mathcal{O}_{F_i}$ are bundles over $\partial \mathbb{H}^2$ with fiber $\mathbb{H}^2 \times \R$.

\medskip

For $N=N_{a,b}$ and $a>b>0$, the faces are similar to the ones for $N_1$, but the difference is that $F_\rho=F_2$ contains the horofunctions of the form $h(x,y)=a \times d(o,x)+b\times b_{\beta,o'}(y)$ for $\beta \in \partial \mathbb{H}^2$ and $o,o'\in \mathbb{H}^2\times \mathbb{H}^2$.

\medskip

This discussion about $\mathbb{H}^2\times \mathbb{H}^2$ is prolonged in Section \ref{subsec:ExamplesV2}, where we apply the main results of the paper about ideal Poisson--Voronoi tessellations.

\subsection{Example \texorpdfstring{$G=\SL(d,\R)$}{G=SL(d,R)}}
\label{subsec:Horofunctions for SLn}

    Consider the case when $G= \SL(d,\R)$ and the Finsler metric is $d_{\omega_1}$. The horofunction compactification $\overline{\mathcal{X}}^{d_{\omega_1}}$ decomposes into $G$-orbits $\mathcal{O}_{F_1}, \cdots, \mathcal{O}_{F_{d-1}}$ where for $1\leq k\leq d-1$ the face $F_k\subset \mathfrak{a}$ is the following:
    $$F_k=\left\lbrace \Diag(1,\cdots, 1,\lambda_{k+1},\lambda_{k+2},\cdots, \lambda_{d})\mid \lambda_{k+1}+\cdots +\lambda_{d}+k=0,\, \lambda_i\leq 1\right\rbrace.$$

    \begin{rem}
        The conventions for labelling the faces are opposite to that of Section \ref{sec:Selberg} for a good reason: the Selberg invariant is close to the metric $d_{\iota\omega_1}$ and not $d_{\omega_1}$.
    \end{rem}

    In this case, the set of roots stabilizing $F_{k}$ is $\Theta_{F_k}^{\stable}=\Delta \setminus \{\alpha_k\}$. 
    Indeed the symmetry associated to the root $\alpha_\ell=\lambda_\ell^*-\lambda_{\ell+1}^*:(\lambda_1,\cdots, \lambda_d)\mapsto \lambda_\ell-\lambda_{\ell+1}$ for $1\leq \ell\leq d-1$ exchanges the $\ell$-th entry and the $(\ell+1)$-th entry, which preserves $F_k$ except if $\ell=k$. The corresponding parabolic subgroup $P_{\Delta\setminus \Theta_{F_k}^{\stable}}$ is the group of elements of the following form in a block decomposition of dimensions $(k,d-k)$ :
    $$ \begin{bmatrix}
        * & * \\
    0 &  * 
    \end{bmatrix}.$$

    The corresponding flag manifold $\mathcal{F}_{\Delta\setminus \Theta^{\stable}_{F_k}}$ is naturally the Grassmanian of $k$-dimensional subspaces.

    \medskip

    Among these roots, the ones whose associated symmetry fixes $F_k$ pointwise are $\Theta_{F_k}^{\fixed}=\{\alpha_1,\dots,\alpha_{k-1}\}$. 
    Indeed the dual face, described in the dual basis $(\lambda_i^*)$ of $\mathfrak{a}$ is the following:
    $$F_k^\vee=\left\lbrace (\mu_1,\mu_2,\cdots,\mu_k, -1,\cdots,-1) \mid \mu_1+\cdots \mu_k=k,\, \mu_i\geq -1\right\rbrace\subset \mathfrak{a}^*.$$
    We see that $\Lin_0(F_k)=\left\lbrace (\mu_1,\mu_2,\cdots,\mu_k, 0,\cdots,0) \mid \mu_1+\cdots \mu_k=0\right\rbrace$ contains $\alpha_\ell$ for $\ell< k$.

    The other roots whose associated symmetry fixes $F_k$ pointwise are $\Theta_{F_k}^{\flipped}= \{\alpha_{k+1},\dots,\alpha_{d-1}\}$. Indeed $\alpha_\ell\in \Lin(F_k)$ for $\ell \geq k+1$.

    \medskip

 The Lie subgroup $G_{F_k}^{\fixed}$ can be written as the group of elements of the following form for $U\in \SL(k,\R)$:
 $$ \begin{bmatrix}
        U & 0 \\
        0 &  \Id_{d-k} 
    \end{bmatrix}.$$
Its maximal compact $K_{F_k}^{\fixed}$ is obtained by setting $U\in \SO(k,\R)$.

The Lie subgroup $G_{F_k}^{\flipped}$ can be written as the group of elements of the following form for $B\in \SL(d-k,\R)$:
 $$ \begin{bmatrix}
        \Id_{k} & 0 \\
        0 &  B
    \end{bmatrix}.$$

    In our case, $\Lin_0(F_k)$ is already equal to the span of all the roots in $\Theta_{F_k}^\text{fixed}$. The orthogonal complement of $\Lin_0(F_k)$ in $\Lin(F_k)$ is the line generated by  $v=( d-k,\cdots, d-k, k, \cdots, k)\in \mathfrak{a}$. This characterises the Lie algebra of $R_{F_k}$ and $S_{F_k}$. However these groups are not exactly the exponential of their Lie algebras\footnote{They are a degree $2$ cover over it.}: one needs to use Proposition \ref{prop:comparing RF and SF to Gflip and Gfix} to see that they contain all of $\SO(n,\R)\cap P_{\Delta\setminus \Theta^{\stable}_{F_k}}$. 

    As a consequence, the stabiliser $R_{F_k}$ of the standard horofunction in $\mathcal{O}_{F_k}$ for our choice of model flat and baseoint in $\X$ is the group of elements of the form:
    $$ \begin{bmatrix}
        U  & A \\
        0 &  B 
    \end{bmatrix}.$$
    Here $(U,B)\in S(O(k,\R)\times\GL(n-k,\R))$ and $A$ is any $k\times (n-k)$ matrix. The elements in $S_{F_k}$ have the form:
    $$ \begin{bmatrix}
       \lambda^{n-k} U  & A \\
        0 &  \lambda^kB 
    \end{bmatrix}.$$
    Here $(U,B)\in S(O(k,\R)\times\GL(n-k,\R))$, $A$ is any $k\times (n-k)$ matrix and $\lambda\in \R_{> 0}$. 

    We provide a more concrete description of the homogeneous spaces $\mathcal{O}_{F_k}=G/R_{F_k}$ in Section \ref{subsec:Selberg}.

 \section{Bisectors for Finsler horofunctions}\label{sec:bisectors}
We now turn to the investigation of bisectors. 
Let us fix any basepoint $o\in \X$, and a Finsler metric $d^N$ on $\X$ associated to a polyhedral norm $N$ on $\mathfrak{a}$. 
Let $F$ be a face of the unit ball $B^N$. 
Recall that $\mathcal{O}_F$ is a fiber bundle over a flag manifold $\mathcal{F}_\Theta$, where $\Theta=\Delta\setminus \Theta_F^\text{stable}$.

\begin{defn}
    We say that two horofunctions $h_1,h_2\in \mathcal{O}_F$ are in \emph{generic position} if the corresponding flags $f_1,f_2\in \mathcal{F}_\Theta$ are in generic position, \ie their relative position in $W_\Theta\backslash W/W_\Theta$ is the maximal one for the Bruhat order.
\end{defn}

Now given two horofunctions $h_1,h_2\in \mathcal{O}_F$ in generic position we consider their \emph{bisector}:
$$ \Bis(h_1,h_2)=\lbrace x\in \X\mid h_1(x)=h_2(x)\rbrace. $$

A crucial property for the behavior of ideal Poisson--Voronoi diagrams is the behavior of $h_1$ (or $h_2$) on $\Bis(h_1,h_2)$. 
Namely we will ask whether or not $h_1$ is proper and bounded from below on $\Bis(h_1,h_2)$. 

We give characterizations of this property, in particular in the case of the orbit of horofunctions associated with the face $F_\rho$.
We then prove that bisectors for horofunctions in generic position have empty interior in $\X$, despite the fact that bisectors in a polyhedral normed space $(\mathfrak{a},N)$ may have non-empty interior.

\subsection{Characterizing properness on bisectors}

We consider a face $F$ which meets the model Weyl chamber $\mathfrak{a}^+$. 
Recall that $w_0$ denotes the longest element of the Weyl group.

\begin{defn}
    We say that $\mathcal{O}_F$ has the \emph{proper-on-bisector property} if the model horofunction $h_F$ on $\mathfrak{a}$ is proper and bounded from below on the following bisector:
    $$\lbrace v \in \mathfrak{a}\mid h_F(v)=h_F(w_0(v))\rbrace.$$
\end{defn}

\begin{prop}
    \label{prop: proper_on-bisector Plus}
    Equivalently $\mathcal{O}_F$ has the proper-on-bisector property if and only if for no non-zero $v\in \mathfrak{a}$, $h_F(v)\leq 0$ and $h_F(w_0(v))\leq 0$.
\end{prop}
\begin{proof}
    We need to show that if such a $v$ exists, then one can find a $v'\in\mathfrak{a}$ such that $h_F(v')=h_F(w_0(v'))\leq0$. Note that for all $f\in F$, $h_F(f)<h_F(w_0(f))$ and $h_F(f)<0$. Hence one can find such a $v'$ in the union of two segments in $\mathfrak{a}$: $[f,v]\cup [v, w_0f]$ such that $h_F(v')=h_F(w_0(v'))$, and automatically $h_F(v')\leq 0$ if $v'$ is in the first segment and $h_F(w_0(v'))\leq 0$ if $v'$ is in the second segment.
\end{proof}

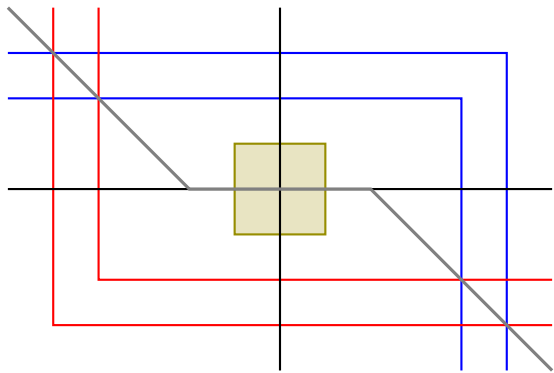
\begin{figure}
\begin{center}
\begin{tikzpicture}[scale=1.2]

    \draw[olive, fill=olive!20, thick] (.5,.5) -- (-.5,.5)--(-.5,-.5) --(.5,-.5)-- cycle;

    \draw[blue, thick] (-3,1) -- (2,1) -- (2,-2);
    \draw[blue, thick] (-3,1.5) -- (2.5,1.5) -- (2.5,-2);
    \draw[red, thick] (3,-1) -- (-2,-1) -- (-2,2);
    \draw[red, thick] (3,-1.5) -- (-2.5,-1.5) -- (-2.5,2);

        \draw[thick] (3,0) -- (-3,0);
    \draw[thick] (0,2) -- (0,-2);
    \draw[gray, very thick] (-3,2)-- (-1,0) -- (1,0) -- (3,-2);

\end{tikzpicture}
\end{center}
\caption{Illustration of two opposite horofunctions in the orbit $F_\rho$ for the metric $N_\infty$ in the case $G=\SL(2,\R)\times \SL(2,\R)$,  illustrating an orbit with the proper-on-bisector property. The bisector is the gray skew line. Several level sets of $h_F$ and $h_F(w_0(\cdot))$ are illustrated in color.}
\end{figure}

This property that we formulated only in the model flat $\mathfrak{a}$ gives us information on the behavior of horofunctions on bisectors in $\X$.

\begin{prop}
\label{prop:MainProp of proper horofunctions}
    If $\mathcal{O}_F$ has the proper-on-bisector property, then for any horofunctions $h_1,h_2\in \mathcal{O}_F$ in generic position, there exist $A,B>0$ such that for all $x\in\Bis(h_1,h_2)$:
    $$h_1(x)=h_2(x)\geq Ad^N(o,x)-B.$$
If $\mathcal{O}_F$ does not have the proper-on-bisector property, then for any pair of horofunctions $h_1,h_2\in \mathcal{O}_F$ in generic position, $h_1$ is not proper and bounded from below on $\Bis(f,g)$.
\end{prop}

We first prove a result concerning the Bruhat order (from Section \ref{sec:Flag manifolds}) and the standard horofunction. 

\begin{lem}
\label{lem:Bruhat order}
Let $F$ be a face of the unit ball for $B^N$ that intersects $\overline{\mathfrak{a}^+}$. 
Let $x\in \overline{\mathfrak{a}^+}$ and let $w_1,w_2\in W$ such that $w_1\leq w_2$ for the Bruhat order. 
Then $h_F(w_1\cdot x)\leq h_F(w_2\cdot x)$.
\end{lem}

\begin{proof}

The vertices of $f^\vee$ define linear forms $\omega_i\in \mathfrak{a}^*$ so that  for all $a \in \mathfrak{a}$, $h_F(a) = \max_i \{-\omega_i(a)\}$. 

Each of these functionals is maximized among its Weyl group orbit on $v$, so lies in $W_v \cdot (\mathfrak{a}^*)^+$, where $(\mathfrak{a}^*)^+$ denotes the dominant weights.
By \cite[Lemma 3.4]{KapovichLeebPorti2018dynamicsonflagmanifolds}, it holds that each $\omega_i \in (\mathfrak{a}^*)^+$ satisfies $\omega_i( w_1 x) \ge \omega_i(w_2 x)$. 
Since the set $\{\omega_i\}$ is preserved by $W_v$, this concludes the proof.
\end{proof}

With this we can now prove Proposition \ref{prop:MainProp of proper horofunctions}. 
\begin{proof}[{Proof of Proposition \ref{prop:MainProp of proper horofunctions}}] 

We begin by proving that the proper-on-bisector property implies that generic pairs of horofunctions are proper on their bisectors.

Suppose that there exists a diverging sequence $(x_n)$ in $\X$ such that $h_1(x_n)=h_2(x_n)$ for all $n\in \mathbb{N}$ and $h_1(x_n)/d^N(o,x_n)$ converges to $\ell\leq 0$. Both $h_1$ and $h_2$ are convex for the Riemannian metric as a maximum of Riemannian Busemann funcctions (see Proposition \ref{prop:Explicit formula for Horofunctions}). Up to taking a subsequence, we can assume that the Riemannian segments $[o,x_n]$ converge to a Riemannian geodesic ray $\gamma:\R_{\geq 0}\to \X$. We have in particular that $h_1(\gamma(t))\leq h_1(o)$ and $h_2(\gamma(t))\leq h_2(o)$ for all $t\in \R_{\geq 0}$.

Let $c^o_1,c^o_2$ be the flags associated to $h_1,h_2$.
Since they are in generic position, there exist opposite
full flags $c_1,c_2\in\mathcal{F}_\Delta$ projecting to
$c^o_1,c^o_2$, respectively.
Choose a Weyl chamber $c$ at infinity containing
$\xi=\gamma(\infty)$, and let
$u\in\overline{\mathfrak{a}^+}$ be the unit vector
representing the type of $\xi$.
Write
\[
w_i=\operatorname{pos}(c_i,c)\in W
\qquad(i=1,2).
\]

We claim that
\[
h_i(\gamma(t))=t\,h_F(w_i u)+O(1).
\]
Indeed, choose a maximal flat whose boundary contains
$c_i$ and $c$.
A ray defining $h_i$, chosen with the type supplied by
Lemma \ref{lem: tau in ker flipped}, and the ray $\gamma$
can both be replaced by asymptotic rays in this flat.
These replacements change the relevant function values
by uniformly bounded quantities.
In coordinates identifying $c_i$ with the model chamber,
the resulting horofunction is a translate of $h_F$,
and the direction of the second ray is $w_i u$.
The claim follows from the fact that $h_F$ is positive homogeneous and Lipschitz.

Since both $h_i$ are bounded above along $\gamma$, we have
\[
h_F(w_1u)\leq0,\qquad h_F(w_2u)\leq0.
\]
Proposition \ref{prop:Tits angle}, applied to the opposite
full flags $c_1,c_2$ and the flag $c$, gives
\[
w_0w_1\leq w_2.
\]
By Lemma \ref{lem:Bruhat order},
\[
h_F(w_0w_1u)\leq h_F(w_2u)\leq0.
\]
Thus the nonzero vector $v=w_1u$ satisfies
$h_F(v)\leq0$ and $h_F(w_0v)\leq0$, contradicting
Proposition \ref{prop: proper_on-bisector Plus}.

\medskip

Now suppose that the proper-on-bisector property is not satisfied.  Take $(h_1,h_2)$ a generic pair of horofunctions. They can be both seen as a limit of renormalized distance functions along geodesics. Up to acting by $G$ we can suppose that these geodesic are at bounded distance from the model flat and hence up to acting by this element of $G$,  $h_1=h_F$ and $h_2$ is at distance from $h_F\cdot w_0$ uniformly bounded by some $A>0$. In the model flat, the failure of the proper-on-bisector property guarantees that there is a $v\in \mathfrak{a}$ such that $h_F(v)\leq 0$ and $h_F(w_0\cdot v)\leq 0$. The difficulty that remains is that the bisector $\Bis(h_1,h_2)$ is not necessarily at bounded distance from the bisector  $\Bis(h_F,h_F\circ w_0)$.

To circumvent this difficulty we find other poins on the bisector. If $f\in F$, $h_F(w_0(f))-h_F(f)> 0$ since $w_0(f)\notin \overline{F}$. In particular for $r$ large enough $h_1(rf)<h_2(rf)$ and $h_1(rw_0f)>h_2(rw_0f)$. In particular there is an element $x$ in the union of two segments: $[rf,rv]\cup [rv,rw_0f]\subset \mathfrak{a}$ on the bisector $\Bis(h_1,h_2)$. But $h_F(v)\leq 0$ and $h_F(w_0\cdot v)\leq 0$ and $h_F(f)<0$. Hence the element $x$ satisfies $h_F(x)\leq 0$ or $h_F(w_0\cdot x)\leq 0$ which implies that $h_1(x)=h_2(x)\leq A$. When $r$ goes to $+\infty$ the segments in which $x$ lie uniformly go to $+\infty$ so we can find a diverging sequence of such points in the bisector. Hence $h_1$ is not proper and bounded from below on the bisector. 

\end{proof}

The rest of this subsection and the next are devoted to understanding when the $G$ orbit $\mathcal{O}_{F_\rho}$ has the proper-on-bisector property. 
We begin by a result that is valid for all faces $F$ of the unit ball $B^N$ intersecting $\mathfrak{a}^+$.

Recall that the dual face $F^\vee$ is the subset of $(B^N)^\ast$ which is maximized on $F$.

\begin{prop}\label{prop: proper bisectors}
    The orbit $\mathcal{O}_F$ has the proper-on-bisector property if and only if:
    \begin{equation}\label{eq:Convex combination}
    \Cone(F^\vee \cup w_0\cdot F^\vee) =\mathfrak{a}^* ,\end{equation} 
    where $\Cone(S)$ denotes the smallest convex cone in $\mathfrak{a}^\ast$ containing $S$.
\end{prop}

\begin{proof}

Suppose that we don't have the proper-on-bisector property, so that for some non zero $v\in \mathfrak{a}$, $-\omega_i(v)\leq 0$ and $-w_0\omega_i(v)\leq 0$ for all $\omega_i \in \ver(F^\vee)$ by Proposition \ref{prop: proper_on-bisector Plus}.

Then for every $\omega\in \mathfrak{a}^*$ such that $\omega(v)<0$, $\omega$ does not lie in the cone from \eqref{eq:Convex combination}, so cone is not all of $\mathfrak{a}^*$.

\medskip

On the other hand, if \eqref{eq:Convex combination} does not hold, we may consider the non-empty cone
\begin{equation*}
    V_F = \{ v \in \mathfrak{a} : \forall \beta \in \Cone( F^\vee \cup w_0\cdot F^\vee ), \beta(v) \ge 0\} .
\end{equation*}
Observe that $V_F$ is a convex cone and preserved by $w_0$. 
If $w_0$ is $-\Id$ on $V_F$ then $V_F$ is a vector subspace of $\mathfrak{a}$ where $h_F$ and $h_{w_0 F}$ are identically zero. 
Otherwise, there is $v \in V_F$ with $w=v+ w_0 v$ nonzero. 
Then $h_F( t w) = h_{w_0 F}(t w) \le 0 $ for all $t \ge 0$.
In either case, $h_F$ cannot be proper and bounded below on the bisector.
\end{proof}

When the opposition involution does not preserve $F$, the situation becomes more complicated.
We will now restrict attention to the distinguished face $F_\rho$.
Note that the opposition involution may fail to preserve $F_\rho$, even though it fixes $\rho$ itself, as the norm $N$ is not necessarily symmetric, \ie not necessarily fixed by the opposition involution.
Recall that $\Lin_0(F)$ denotes the linear subspace of $\mathfrak{a}$ parallel to the affine span of $F$. 

\begin{figure}[h!]
\begin{center}
\begin{tikzpicture}[scale=1.2]

    \draw[olive, fill=olive!20, thick] (1,0) -- (1.5*.5,1.5*0.866) -- (-.5,0.866) --(-1.5,0)--(-.5,-0.866) -- (1.5*.5,-1.5*0.866) -- cycle;

    \draw[red] (2*.5+.6*.866,-2*0.866+.3) -- (-2*.5+.6*.866,2*0.866+.3);
        \fill[red] (1.5*.5,1.5*0.866) circle (.05);

    \draw[thick] (2*.5,2*0.866) -- (-2*.5,-2*0.866);
    \draw[thick] (2*.5,-2*0.866) -- (-2*.5,2*0.866);
    \draw[thick] (2,0)--(-2,0);
    
    \node[above left] (A) at (1.5*.5,1.5*0.866) {$F_\rho$};
    \node[below left] (A) at (-.1,0) {$N$};
    \node[right] (A) at (-2*.5+.6*.866,2*0.866+.3) {$\rho=c$};
\end{tikzpicture}
\hfill
\begin{tikzpicture}[scale=1.2]

    \draw[olive, fill=olive!20, thick] (1,0) -- (1.5*.5,1.5*0.866) -- (-.5,0.866) --(-1.5,0)--(-.5,-0.866) -- (1.5*.5,-1.5*0.866) -- cycle;

    \draw[red, thick] (-2,.46) -- (.83,1.44) -- (1.5,-2);
    \draw[red, thick] (-2,.62) -- (.95,1.65) -- (1.65,-2);

    \draw[blue, thick] (.65,2) -- (-1.66,0) -- (.65,-2);
    \draw[blue, thick] (.4,2) -- (-1.91,0) -- (.4,-2);

    \draw[thick] (2*.5,2*0.866) -- (-2*.5,-2*0.866);
    \draw[thick,gray] (2*.5,-2*0.866) -- (-2*.5,2*0.866);
    \draw[thick] (2,0)--(-2,0);
    
    \node[right] (A) at (2*.5,-2*0.866) {$h_1=h_2$};
    \node[right] (A) at (1.2,.75) {$h_2=c$};
    \node[left] (A) at (-.75,-.75) {$h_1=c$};
    \node[below left] (A) at (-.1,0) {$N$};

\end{tikzpicture}
\end{center}
\caption{
Illustration of $F_{\rho}$ for a hexagonal norm $N$ for $\X=\SL(3,\R)/\SO(3,\R)$, as well as the level lines of two opposite horofunctions on the orbit $\mathcal{O}_{F_\rho}$.}
\end{figure}
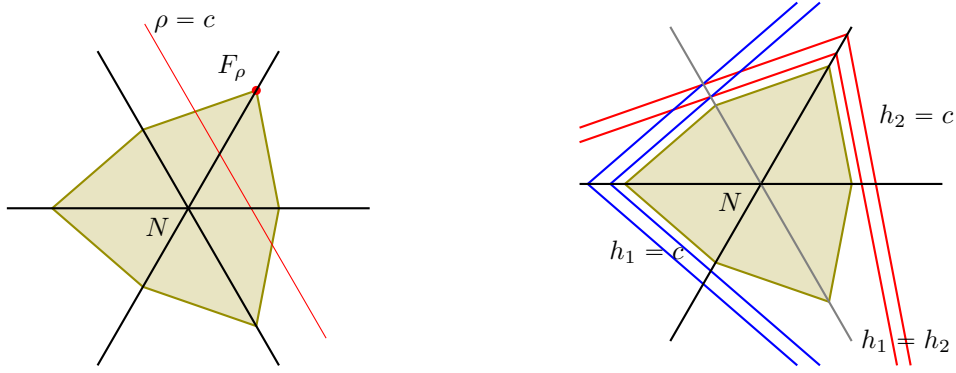

\begin{prop}
\label{prop:Caracterisation proper bisectors Frho}
The orbit $\mathcal{O}_{F_\rho}$ has the proper-on-bisector property if and only if: $$w_0\cdot \Lin_0(F_\rho)\cap \Lin_0(F_\rho)=\lbrace 0\rbrace.$$ 
In particular if $F_\rho$ is a vertex, $\mathcal{O}_{F_\rho}$ has the proper-on-bisector property.

\end{prop}

\begin{proof}[{Proof of Proposition \ref{prop:Caracterisation proper bisectors Frho}}]
    Let $h_{F_\rho}$ be a standard horofunction. Its restriction to $\mathfrak{a}$ can be written as the maximum of finitely many functionals $\omega_i$, for $i\in I_{F_\rho}$.
If $ w_0\cdot \Lin_0(F_\rho)\cap \Lin_0(F_\rho)\neq\lbrace 0\rbrace$, then the linear span of ${F_\rho}^\vee$ and $w_0\cdot {F_\rho}^\vee$ is not equal to $\mathfrak{a}^*$, and hence \eqref{eq:Convex combination} is not satisfied.

Conversely, by definition of $F_\rho$, $\rho$ is in the interior of $F_\rho^\vee$. 
Then also $-\rho = w_0 \rho$ is in the interior of $w_0 F_\rho^\vee$. 
It follows that $\rho$ is in the relative interior of the convex hull of $F_\rho^\vee$ and $2\rho + w_0 F_\rho^\vee$.
When $ \Lin_0(F_\rho) \cap \Lin_0(w_0 F_\rho) = \{0\}$, it holds that the span of $F_\rho^\vee \cup w_0 F_\rho^\vee$ is $\mathfrak{a}^\ast$. 
It follows that the origin lies in the interior of the convex hull of $F_\rho$ and $w_0 F_\rho$, so we may apply Proposition \ref{prop: proper bisectors} to see that $\mathcal{O}_{F_\rho}$ has the proper-on-bisector property.

\end{proof}

\begin{remark}
    The "if" part of Proposition \ref{prop:Caracterisation proper bisectors Frho} does not hold for a general face $F$. 
    For example, take Figure \ref{fig:A3 example} with the roles of $B^N$ and $(B^N)^\ast$ exchanged. 
    Then the affine span is one-dimensional, but this intersection is zero, and the convex cone is still contained in a proper half-space.
\end{remark}

\begin{corollary}\label{cor: proper bisectors when opposition involution is trivial}
    If the opposition involution is trivial, then the orbit $\mathcal{O}_{F_\rho}$ has the proper-on-bisector property if and only if 
    $F_\rho$ is a vertex.
\end{corollary}

In the case when the opposition involution is not trivial, we can construct many examples of polyhedral norms for which $\mathcal{O}_{F_\rho}$ has the proper-on-bisector property.

\begin{prop}
Let $P$ be any polyhedron in $\mathfrak{a}^+\cap \lbrace \rho=1\rbrace $. 
There exist a polyhedral norm $N$ that is Weyl group invariant on $\mathfrak{a}$ such that $F_\rho$ is the relative interior of $P$.
\end{prop}

Recall that in our convention faces of polyhedra are relatively open, but $P$ can have empty interior in $\mathfrak{a}^+\cap \lbrace \rho=1\rbrace$. 

\begin{proof}
 Let us fix a finite family of linear forms $(\omega_i)_{i\in I_0}$ such that $\omega_i\leq 1$ on $P$ and $P\cap \lbrace \rho=1\rbrace=\lbrace x\in \mathfrak{a}\mid \rho(x)=1,\, \forall I\in I_0,\, \omega_i(x)\leq 1\rbrace$. 
 We moreover assume that one of these linear forms is $\rho$. 

We define the norm $N_t$ for all $0\leq t<1$ to be equal for $x\in \mathfrak{a}$ to:
$$N_t(x)=\max_{i\in I_0,w\in W} (1-t)\omega_i(w\cdot x)+t\rho(w\cdot x).$$

For all $0\leq t<1$, $N(x)\geq 1$ for $x\in P_0$. Note that for all $x\in \mathfrak{a}$, and $w\in W$, $\rho(x)\geq \rho(w\cdot x)$ with equality if and only if $x=w\cdot x$. For $t$ close enough to $1$, for $x$ in a neighborhood of $\mathfrak{a}^+\cap P_0$ one has $N_t(x)=\max_{i\in I_0}(1-t)\omega_i(x)+t\rho(x)$. Hence for this norm $F_\rho=P_0$.
\end{proof}

\begin{example}
    It can happen that $\mathcal{O}_{F_\rho}$ has the proper-on-bisector property while $F_\rho$ is not a vertex when the opposition involution is non-trivial. 

    In the $D_3 \simeq A_3$ root system, i.e.\ $\R^3$ with 
        \[ \Delta = \{ (1,-1,0),(0,1,-1),(0,1,1)\} \]
    and $\rho$ is (a multiple of) the sum of fundamental weights $(2,1,0)$, take $\omega = (2,2,1)$, see Figure \ref{fig:A3 example}.
    Then the face $F_\rho$ of $B^{N_\omega}$ is an edge joining two rescaled fundamental coweights. 
    The dual ball $(B^N)^\ast$ is a truncated simplex, and one checks that the origin lies in the interior of $\Cone(F_\rho^\vee \cup w_0 F_\rho^\vee)$.
    Hence $F_\rho$ is not a vertex and $\mathcal{O}_{F_\rho}$ has the proper-on-bisector property. 

    Notice also that $F_\rho \cup w_0 F_\rho$ spans $\mathfrak{a}$, but generates a convex cone in a proper half-space. 
    So the dual face of the dual norm does not have the proper-on-bisector property. 
\end{example}

\begin{figure}
    \centering
    \includegraphics[width=0.45\linewidth]{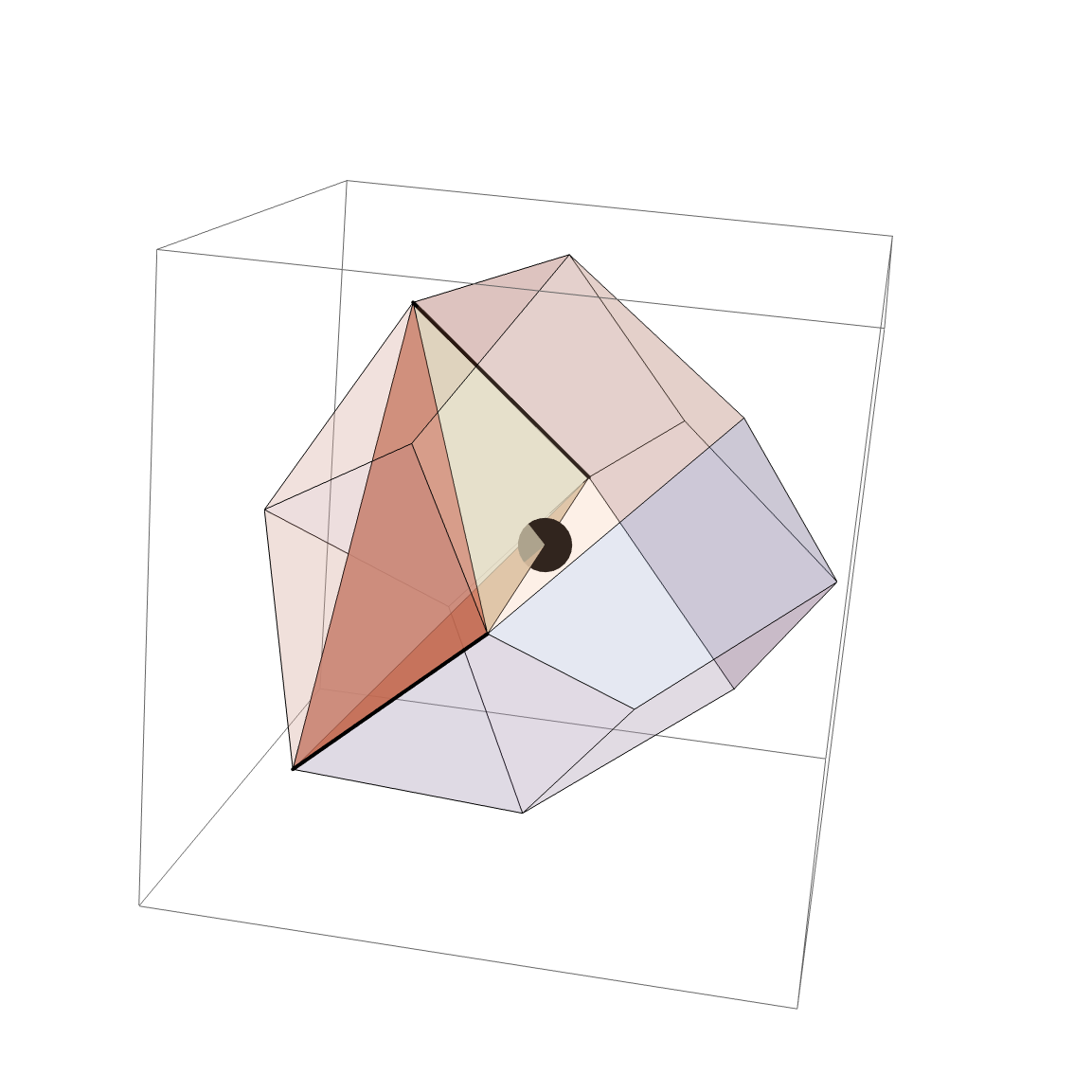}
    \hfill
    \includegraphics[width=0.45\linewidth]{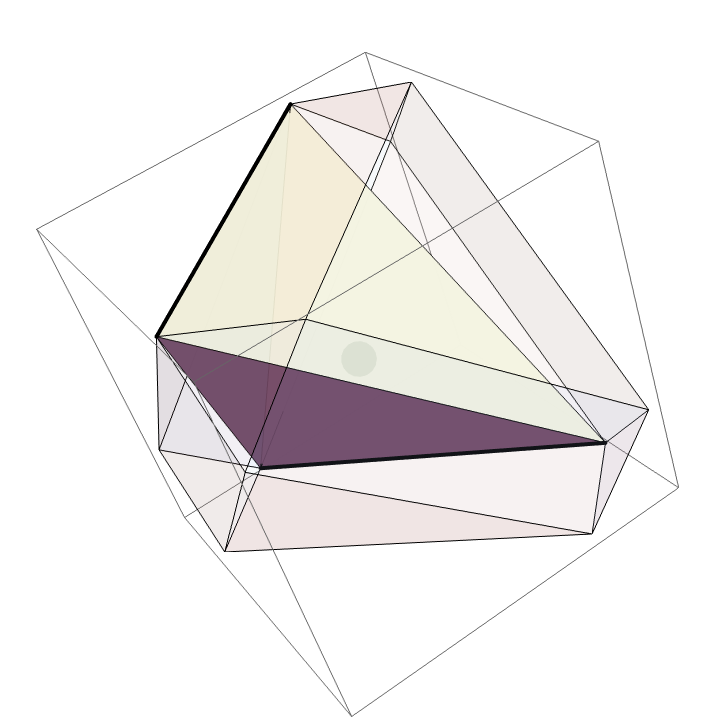}
    \caption{(Left) $B^N$ and the convex hull of $F_\rho\cup w_0 F_\rho$. (Right) $(B^N)^\ast$ and the convex hull of $F_\rho^\vee \cup w_0 F_\rho^\vee$.}
    \label{fig:A3 example}
\end{figure}

\subsection{Polyhedral norms associated to a single linear form}
\label{subsection:Tau_i}
We now investigate in more detail whether, for a given polyhedral norm, $F_\rho$ is a vertex. For this we restrict to  norms associated to a single linear form $\omega\in \mathfrak{a}^*$:
    \[ N = N_\omega(x) = \max \{\omega (wx) : w \in W\}. \]
We may assume that $\omega$ is dominant or in other words $\omega\in (\mathfrak{a}^*)^+$, i.e. it is dual to an element of $\mathfrak{a}^+$, up to acting by the Weyl group.

This situation is already interesting: we will see that $F_\rho$ may have any dimension from $0$ to $\dim(\mathfrak{a})-1$, and this dimension depends on the position of $\omega$ with respect to a decomposition of ${(\mathfrak{a}^*)}^+$.

\medskip

Let $\{\varpi_i\}_{i \in \Delta}$ be the fundamental weights, defined so that for $i,j\in \Delta$:
    \[ 2\frac{\langle \alpha_i, \varpi_j \rangle}{\langle \alpha_i,\alpha_i \rangle} = \delta_{ij}. \]
The fundamental coweights $\{\varpi_i^\vee\}_{i \in \Delta}$ are defined by
    \[ \beta(\varpi_i^\vee) = \langle \beta, \varpi_i \rangle .\]

The convex cone generated by the fundamental coweights is precisely $\mathfrak{a}^+$. 
Let us define a renormalisation of the fundamental coweights using $\rho$  for $i\in \Delta$:
$$\tau_i=\frac{1}{\varpi_i^\vee(\rho)}\varpi^\vee_i.$$

The description of $F_\rho$ for the metric $N_\omega$ depends on the position of $\omega$ in the Voronoi decomposition of $\mathfrak{a}^*$ with respect to the linear foms on $\mathfrak{a}^*$ defined by the $(\tau_i)$.

\begin{prop}
\label{prop:Proper bisectors polyhedral dominant}
Let $N=N_\omega$ with $\omega$ dominant and let $F_\rho$ be the face of $B^{N_\omega}$ maximizing $\rho$. Let $\Theta\subset \Delta$ be the set of indices $i$ such that 
$\omega(\tau_i) =\min_{j\in \Delta}\omega(\tau_j)=c$. Then $F_\rho$ is the convex hull of the $\frac{1}{\omega(\tau_i)}\tau_i= \frac{\varpi_i^\vee}{\omega(\varpi_i^\vee)}$ for $i\in \Theta$, or in other words:
$$F_\rho=\Cone(\tau_i,\, i\in \Theta) \cap \lbrace a\mid \rho(a)=\frac{1}{c}\rbrace,$$
$$F_\rho=\Cone(\tau_i,\, i\in \Theta) \cap \lbrace a\mid \omega(a)=1\rbrace.$$
In particular  $\dim(F_\rho)= \#\Theta-1$.
\end{prop}

\begin{proof}
Since $\omega$ is dominant, any $a \in \mathfrak{a}^+$ satisfies 
        \[ N_\omega(a) = \max_{w\in W} \omega(w a) = \omega(a).\]
    In particular, for each $i \in \Delta$, 
        \[ \frac{\varpi_i^\vee}{\omega(\varpi_i^\vee)} \in \partial B^N = \left\{a \in \mathfrak{a} : N_\omega(a)=1 \right\} .\]

         Moreover, the intersection of $\partial B^N$ with $\mathfrak{a}^+$ is precisely the convex hull of these rescaled fundamental coweights
    \[ \frac{\varpi_i^\vee}{\langle \omega, \varpi_i \rangle}, i \in \Delta. \]

    The face $F_\rho$ is contained in $\mathfrak{a}^+$ since $\rho(w a)$ is uniquely maximized (as $w$ ranges over $W$), on the singleton $W \cdot a \cap \mathfrak{a}^+$. 
    In particular, the closure of $F_\rho$ is the convex hull of 
        \[ \left\{ \frac{\varpi_i^\vee}{\omega(\varpi_i^\vee)} : \frac{\rho(\varpi_i^\vee)}{\omega(\varpi_i^\vee)} \text{ is maximized among } i \in \Delta \right\} .\] 

Note that the inverse of $\frac{\rho(\varpi_i^\vee)}{\omega(\varpi_i^\vee)}$ is precisely $\omega(\tau_i)$, and that $\frac{1}{\omega(\tau_i)}\tau_i=\frac{\varpi_i^\vee}{\omega(\varpi_i^\vee)}$.
\end{proof}

\medskip

In the cases when the opposition involution $\iota = -w_0$ is not trivial, the following proposition tells us when $F_\rho$ has the proper-on-bisector property. 

\begin{prop}
\label{prop:ProperbissectorLinear forms}
Let $\omega$ be a dominant linear form on $\mathfrak{a}$ and let $\Theta\subset \Delta$ be the set of indices $i$ such that 
$\omega(\tau_i) =\min_{j\in \Delta}\omega(\tau_j)$. The face $F_\rho$ for the metric $N_\omega$ has the proper-on-bisector property if and only if $\#(\Theta\cap \iota(\Theta))\leq 1$.
\end{prop}

Hence this property only depends on the set of cell closures in which $\omega$ belongs in the Voronoi decomposition of $\mathfrak{a}^+$. 

\begin{proof}
Proposition \ref{prop:Proper bisectors polyhedral dominant} implies that $\Lin_0(F_\rho)=\langle \tau_i-\tau_j\mid i,j\in \Theta\rangle\cap \Ker(\omega)$, and since $\iota(\hat{\tau_i})=\hat{\tau_{\iota(i)}}$ we have $\Lin_0(w_0\cdot F_\rho)=\langle \tau_i-\tau_j\mid i,j\in \iota(\Theta)\rangle\cap \Ker(\omega)$.

\medskip

We then apply Proposition \ref{prop:Caracterisation proper bisectors Frho}. We need to determine when these two vector subspaces intersect trivially, knowing that the $(\tau_i)$ form a basis and that $\rho(\tau_i)\neq 0$ for $i\in \Delta$. If $i\neq j\in \Theta\cap \iota(\Theta)$ then $\tau_i-\tau_j\neq 0$ belongs to the intersection, so the face $F_\rho$ doesn't have the proper-on-bisector property. On the contrary if $\#(\Theta\cap \iota(\Theta))=0$ the intersection is clearly trivial. If $\#(\Theta\cap \iota(\Theta))=1$ let $i_0$ be the intersection and let $H$ be the span of $\tau_i$ for $i\neq i_0$ in $\Delta$. Since $\omega(\tau_{i_0})\neq 0$ the projection of $\Lin_0(F_\rho)$ and $\Lin_0(w_0\cdot F_\rho)$ parallel to $\tau_{i_0}$ are injective. However their image clearly intersect trivially, so these subspaces intersect trivially.

\end{proof}

\begin{proposition}\label{prop: rho in terms of fundamental weights}
    The restricted Weyl vector $\rho$, see (\ref{eqn: weyl vector}), satisfies 
        \[ \rho = \sum_{\alpha \in \Delta} (m_\alpha + 2 m_{2\alpha}) \varpi_\alpha \]
\end{proposition}

As an easy consequence, note that $\rho$ depends on $\mathfrak{g}$, not just the restricted root system. 
Moreover, $\rho$ is always strictly dominant, i.e.\ when written as a linear combination of the fundamental weights, every coefficient is positive.

\begin{proof}
    If $\alpha$ is simple then the reflection $s_\alpha$ negates $\alpha$ and permutes the other positive roots. 
    Hence $s_\alpha(\rho)=\rho-m_\alpha\alpha-m_{2 \alpha}2 \alpha$ and 
    \begin{align*}
        \rho(\widehat{\alpha}) & = (s_\alpha \rho)(s_\alpha \widehat{\alpha}) = (\rho-m_\alpha\alpha-m_{2\alpha}2\alpha)(-\widehat{\alpha}) \\
        & = - \rho(\widehat{\alpha}) + m_\alpha\alpha(\widehat{\alpha}) + m_{2\alpha}2\alpha(\widehat{\alpha}) = - \rho(\widehat{\alpha}) + 2 m_\alpha + 4 m_{2 \alpha} 
    \end{align*}
     implies $\rho(\widehat{\alpha})=m_\alpha+2m_{2\alpha}>0$.
\end{proof}

\begin{remark}
    Proposition \ref{prop: rho in terms of fundamental weights} allows one to explicitly determine the fractions $\frac{\rho(\varpi_i^\vee)}{\omega(\varpi_i^\vee)}$ for a given $\omega$.
    The dimensions $m_\beta = \dim \mathfrak{g}_\beta$ are recorded in e.g.\ \cite[p 532]{helgason1979differential} and the pairings $\langle \varpi_i, \varpi_j \rangle$ depend only on the root system and may be read off from the tables in \cite[p 295]{onishchikvinberg2012liegroups}. 
\end{remark}

\begin{remark}
\label{rem:drho}
Using Proposition \ref{prop:ProperbissectorLinear forms}, one can see that for the linear form $\omega=\rho$, the set $\Theta$ is equal to $\Delta$. Hence in this particular case $\mathcal{O}_{F_\rho}$ does not have the proper-on bisector property in higher rank. 
In this case $F_\rho$ is a codimension one face, so $\mathcal{O}_{F_\rho}$ corresponds to a $G$-orbit of Riemannian Busemann functions. 
This $G$-orbit is exactly the $G$-orbit of horofunctions for the Riemannian metric that supports the measure constructed by \cite{FMW23poissonvoronoi}. 

\end{remark}

\subsection{Bisectors have empty interior in \texorpdfstring{$\X$}{X}}

 In a vector space with polyhedral norm, the bisector between two points may have interior. 
 We will show this doesn't happen for polyhedral Finsler metrics on $\X$.

\begin{thm} 
\label{thm:Empty Interior}
Suppose that $G$ is semi-simple.
    Let $h_1,h_2 \in \mathcal{O}_F$ be horofunctions in generic position.
    The set $\Bis(h_1,h_2) \subset \X$ has empty interior.
\end{thm}

We will prove the theorem by establishing several lemmata.

 By Proposition \ref{prop:Explicit formula for Horofunctions} we can write $h_1$ and $h_2$ (up to acting by $G$) as follows,  for some $o_1, o_2\in \X$ :
\begin{equation}
    \label{eq:horofunction 1}
    h_1(x)=\max_{i\in I_F,\, p\in E^1_i}\lVert \omega_i\rVert b_{p,o_1}(x), 
\end{equation}
\begin{equation}
    \label{eq:horofunction 2}
h_2(x)=\max_{i\in I_F,\, p\in w_0\cdot E^2_i}\lVert \omega_i\rVert b_{p,o_2}(x).
\end{equation}

Here $E_i^1=P^1_\Theta\cdot [\omega_i]$ and $E_i^1=P^2_\Theta\cdot [\omega_i]$ for some conjugates $P^1_\Theta, P^2_\Theta$ of $P_\Theta$. Since $h_1$ and $h_2$ are opposite horofunctions, we can assume $P_\Theta^1=P_\Theta$ and $P_\Theta^2=w_0P_\Theta w_0$ where $w_0\in G$ is a representative of the longest element of the Weyl group.

\medskip

 Let $M^1(x)\subset \bigcup_{i\in I_F} E^1_i\subset \partialvis \X$ and $M^2(x)\subset \bigcup_{i\in I_F}  E_i^2\subset \partialvis \X$ be respectively the sets of points on which the maximum is reached in \eqref{eq:horofunction 1}, respectively \eqref{eq:horofunction 2}, at the point $x\in \X$. 

 \medskip
Since the $\omega_i$ may have different norms, we need to decompose the sets $M^k(x)$. Let $C_1>C_2>\cdots >C_m$ be the distinct values of $\lVert \omega_i\rVert$ for $i\in I_F$. For $1\leq \ell\leq m$ and $k=1,2$ we consider the function obtained by keeping only the terms of weight at most $C_j$ in \eqref{eq:horofunction 1} and \eqref{eq:horofunction 2}:
$$h^\ell_k(x)=\max_{i\in I_F,\, \lVert\omega_i\rVert\leq C_\ell,\, p\in E^k_i}\lVert \omega_i\rVert b_{p,o_k}(x).$$
In particular $h^1_k=h_k$ for $k=1,2$. For such a function $h^\ell_k$ and $x\in \X$ we denote by $M(h^\ell_k,x)$ the set of points $p$ at which the maximum defining $h^\ell_k(x)$ is reached. 

\medskip

Let $\ell_0$ be the smallest index such that $M(h^{\ell_0}_1,x_0)\neq \emptyset$. In particular this means that by the compactness of the $E_i^1$, in a neighborhood of $x$, $h^{\ell_0}_1=h_1$.

\medskip

We need the following: 
\begin{lem}
    If $h_1=h_2$ on a neighborhood of $x_0\in \X$, $M(h^{\ell_0}_1,x)=M(h^{\ell_0}_2,x)$ for all $x$ close enough to $x_0$. 
\end{lem}
\begin{proof}
Suppose for the sake of contradiction that there is a point $p\in M(h^{\ell_0}_1,x)\setminus M(h^{\ell_0}_2,x)$. 
Let $\gamma\colon \R_{\geq 0}\to \X$ be the geodesic ray with Riemannian speed one starting from $x_0$ going in the direction opposite to $p$.
Since Riemannian Busemann functions have unit gradient, we have that $b_{p,o}(\gamma(t))=b_{p,o}(o)+t$ but for all $q\neq p$ in $\partialvis \X$,  $b_{q,o}(\gamma(t))\leq b_{q,o}(o)+(1-\epsilon) t$ with $\epsilon>0$ depending continuously on $q$. 

In particular since $p\in M^1$, $h_1(\gamma(t))=h_1(x_0)+C_{\ell_0}t$ and $h_2(\gamma(t))<h_1(o)-C_{\ell_0}\epsilon t$ for some $\epsilon>0$. Hence if $h_1$ and $h_2$ coincide on a neighborhhod of $o$,  $M(h^{\ell_0}_1x)\subset M(h^{\ell_0}_2,x)$. 

\medskip

For the converse inclusion, we obtain simillarly $M(h^{\ell_0'}_2,x)\subset M(h^{\ell_0'}_1,x)$ where $\ell_0'$ is the smallest index such that $M(h^{\ell_0'}_2,x_0)\neq \emptyset$. As a consequence of the txo inclusions $\ell_0=\ell_0'$ and we have the desired equality.
\end{proof}

\begin{lem}
    If $h_1=h_2$ on a neighborhood of $x_0\in \X$, then, for any small enough neighborhood $U$ of $x_0$, the sets $\bigcup_{x\in  U} M(h^{\ell_0}_1,x)$ and $ \bigcup_{x\in  U} M(h^{\ell_0}_1,x)$  are open subsets of $\bigcup_{i\in I_F}E^1_i$ and $\bigcup_{i\in I_F} E^2_i$ respectively.
\end{lem}

\begin{proof}
The stabilizer $R_F$ of the horofunction $h_1$ contains $K_\Theta$ which acts transitively on $E^1_i$ by the Definition/Proposition \ref{defnprop:Definition E_i^F} and hence also locally transitively on  $E_i$, for all $i\in I_F$. Hence let $p\in E_i$ be close enough to $p_0\in M^1(o)$, then there exist $g\in R$ such that $g\cdot o\in U$ and $g\cdot p_0=p$. (Note that $K_\Theta$ preserves $o_1$ but not necessarily $o$.) Since $g\cdot h_1=h_1$ we also have $M(h_1^{\ell_0},g\cdot o)=g\cdot M(h_1^{\ell_0},o)$. 

We conclude by applying the same argument to $h_2$. 
\end{proof}

 In particular, since $\bigcup_{i\in I_F}E_i^1$ and $\bigcup_{i\in I_F}E_i^2$. are finite unions of analytic sets that coincide on the non-empty open set $\bigcup_{x\in  U} M(h^{\ell_0}_1,x)$, they share at least one component. 
 Hence there must be some $i,j\in I_F$ such that $E_i^1$ and $E^2_j$ have a common analytic subset. This is contradicted by the following:

\begin{lem}
    Let $G$ be a semisimple Lie group, let $P$ be a parabolic subgroup, and let $w_0$ be a lift of the longest element of the Weyl group.
    Then $P$ and $w_0 P w_0^{-1}$ do not share an orbit in $G/P$.
\end{lem}

\begin{proof}
    Choose $\zeta^+ \in \partial_{vis} \X$ so that $P$ is the stabilizer of $\zeta^+$. 
    Then $\zeta^+$ is in every orbit closure: $\zeta^+ \in \overline{P\cdot \zeta}$ for any $\zeta \in G \cdot \zeta^+$, and similarly $w_0 \zeta^+ \in \overline{w_0Pw_0^{-1}\cdot \zeta}$ for any $\zeta \in G \cdot \zeta^+$.
    Now suppose for the sake of contradiction that $\zeta$ in $G \cdot \zeta^+$ satisfies $P \cdot \zeta = w_0 P w_0^{-1} \cdot \zeta$. 
    By definition, the Bruhat order corresponds to inclusions of Schubert varieties, i.e.\ closures of orbits of $P$. 
    From 
        \[ \overline{P \cdot \zeta} = \overline{P \cdot \overline{P \cdot \zeta}} = \overline{P \cdot \overline{w_0 P w_0^{-1} \cdot \zeta}} \supset \overline{P \cdot w_0 \zeta^+} = G \cdot \zeta^+ \]
    we see that $\overline{P \cdot \zeta} = G \cdot \zeta^+$, but this implies that $\zeta$ is maximal in the Bruhat order, hence is given by $w_0 \zeta^+$. 
    But this point is fixed by $w_0 P w_0^{-1}$, which yields a contradiction.
\end{proof}

This concludes the proof that the bisector of a pair of horofunctions in generic position has empty interior, and thus the proof of Theorem~\ref{thm:Empty Interior}.

\begin{figure}
\begin{center}
\begin{tikzpicture}[scale=1.2]

\draw[olive, fill=olive!20, thick] (1,0) -- (-.5,0.866) --(-.5,-0.866)  -- cycle;

\fill[gray,opacity=0.4] (2,0) -- (2*.5,2*0.866) --(-2*.5,2*0.866) --(0,0) -- cycle;

    \draw[red, thick]  (-2+.5,2*0.866) -- (1.5,0) --(-2+.5,-2*0.866);
    \draw[red, thick] (-1.5,0.866) -- (0,0) -- (-1.5,-0.866);

    \draw[blue, thick] (1.9+.5*.5,0-.866*.5) -- (-.5*1.4,0.866*1.4) -- (-1.9*.5+.5*.5,-1.9*.866-.866*.5);
    \draw[blue, thick] (1+.5,0-.866) -- (0,0) -- (-.5+.5,-.866-.866);

    \draw[very thick, gray] (2,0) -- (0,0) --(-2*.5,2*0.866);
    \draw[very thick, gray] (0,0) -- (-2*.5,-2*0.866);

\end{tikzpicture}
\end{center}
\caption{Illustration of two opposite horofunctions with a bisector in $\mathfrak{a}$ having non-empty interior, but with the proper-on-bisector property.}
\end{figure}
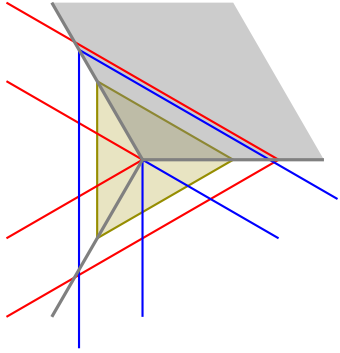

\section{Measures on the space of horofunctions}
\label{sec:Measures}

Riemannian symmetric spaces carry a natural volume measure $\Leb$ associated to the Riemannian metric.
Up to scaling, it is the unique measure invariant by the isometry group.
We describe below a family of rescaled measures $\frac1{m_t}(\iota_t)_\ast \Leb$ and their limits in horofunction compactifications. 
The main result we prove in this section is that they limit to a measure on the space of horofunctions with support in the distinguished stratum $\mathcal{O}_{F_\rho}$. 

Recall that for $t\in\mathbb{R}$, we have the embedding $\iota_t^N \colon \X\to \mathcal{D}^N(\X)$ of $\X$ into the space of distance-like functions
$$ x\mapsto d^N_\X(\cdot, x)-t.$$

Pushing forward the volume measure $\Leb$ on $\X$ by $\iota_t$ yields a 1-parameter family of $G$-invariant measures $(\iota_t)_*\Leb$. 
In order to renormalize these measures we consider an exhaustion of $\overline{\mathcal{D}^N(\X)}$ by compact sets.

\begin{defn}
\label{defn:cylinders}
For a base-point $o\in \X$ and $r, R\in [-\infty, +\infty]$ the \emph{cylinder centered at $o$} is given by:
$$C^o_{r,R}=\left\lbrace f \in \overline{\mathcal{D}^N(\X)}\mid r\leq f(o)\leq R\right \rbrace.$$
\end{defn}

Note that if $-\infty<r$ and $R<+\infty$, the cylinder $C^o_{r,R}$ is compact by the Arzelà--Ascoli theorem.

\medskip

The total mass of $(\iota_t)_\ast \Leb$ is infinite.
In order to renormalize the measures, consider:
$$ m_t = (\iota_t)_*\Leb(C^o_{-\infty,0})=\Leb \left(\lbrace x\in \X\mid d^N(o,x)\leq t\rbrace\right). $$
Note that $m_t$ is simply the volume of the ball of radius $t$ for $d^N$ in $\X$, hence independent of $o \in \X$ by $G$-invariance.

\medskip

In this section we study the vague limit of the sequence of measures $\frac{1}{m_t}(\iota_t)_*\Leb$, seen as measures on $\overline{\mathcal{D}^N(\X)}$. 
Recall that the vague limit of a sequence $(\nu_n)$ is $\nu$ if for all compactly supported continuous functions $f$, $\int f \mathrm{d}\nu_n$ converges to $\int f \mathrm{d}\nu$.

\medskip

When $d_\X^N$ is replaced by the $G$-invariant Riemannian metric induced by the Killing form, these measures and their  limit were studied by Fr{\k{a}}czyk-Mellick-Wilkens. 
They proved \cite[Corollary 3.4 and Theorem 3.6]{FMW23poissonvoronoi} that the corresponding measures limit vaguely to the unique (up to scale) $G$-invariant measure on the space of Riemannian Busemann functions centered at $\rho$, \ie the space of Busemann functions $b \colon \X \to \R$ admitting some maximal flat $\mathfrak{a} \to \X$ so that $b_{|\mathfrak{a}}=-\rho/\Abs{\rho}$.

We will prove the following analog for polyhedral Finsler metrics. 
Recall that $B^N$ denotes the unit ball of $(\mathfrak{a},N)$ and $F_\rho$ is the relative interior of the exposed face of $B^N$ where $\rho$ is maximized. 
This face necessarily lies entirely in $\overline{\mathfrak{a}^+}$ and $\Theta_{F_\rho}^\text{flip}=\emptyset$.
We also write $\rho_{max}$ for the maximum value of $\rho$ on $B^N$.

\begin{thm}
\label{thm: Convergence measure}
    Let $N$ be polyhedral.
    The $1$-parameter family of measures
     $$\mu_t = \frac1{m_t} (\iota_t)_\ast \Leb$$
    converges vaguely on $\overline{\mathcal{D}^N(\X)}$ as $t\to \infty$ to a measure $\mu_N$ supported on $\mathcal{O}_{F_\rho}$. 
   
    This measure $\mu_N$ is the unique $G$-invariant measure on $\mathcal{O}_{F_\rho}$ such that $\mu_N(C^o_{-\infty,0})=1$. 
    
    Moreover this measure is \emph{$2\rho_{\max}$-homogeneous}, in the sense that for any measurable set $E\subset \mathcal{O}_{F_\rho}$ and real number $\lambda\in \R$, $\mu_N(\lambda\cdot E)=e^{\lambda2\rho_{\max}}\mu_N(E)$, where we consider the $\R$ action on the space of horofunctions given by $\lambda\cdot h(x):=h(x)+\lambda$.
\end{thm}

The homogeneity implies that $\mu_N(C^o_{-\infty,R})=e^{2\rho_{\max}R}$.

\medskip

We will give the proof of Theorem~\ref{thm: Convergence measure} at the end of Section \ref{subsec:ComputationLimit}, after some preparation.

\subsection{Volume asymptotics}

We first need to compute the asymptotics of the volume of balls for Finsler metrics. 

\medskip 
Using the standard horofunctions (see Definition \ref{def: standard horofunction}) we introduce the following cone in $\mathfrak{a}$: 
\begin{equation}\label{eq:defn ConeCrho}C_\rho=\lbrace -x\in \mathfrak{a} \mid  h_{F_\rho}(x)\leq 0 \rbrace= \lbrace x\in \mathfrak{a} \mid \forall \beta \in F_\rho^\vee, \, \beta(x)\leq 0\rbrace.
\end{equation}

The above equality is a consequence of Proposition \ref{prop:Formula horofunction standard}. This unbounded convex cone $C_\rho$ is a cone with base $\Lin_0(F_\rho)$, and for $f\in F_\rho$, the cone $f+C_\rho$ is equal to $B^N$ in a neighborhood of $F_\rho$.

\medskip

Let $\Lin_0(F_\rho)^\perp\subset \mathfrak{a}$ be the orthogonal complement of $\Lin_0(F_\rho) \subset \mathfrak{a}$. 
We obtain the following:

\begin{lem}
\label{lem:ConeBallN}
There exists $A>0$ such that for all $v\in C_\rho\cap \Lin_0(F_\rho)^\perp$:
 $$ \rho(v)\leq - A N(v).$$
 Moreover for some $B>0$ the codimension one volume of the intersection  $C_\rho\cap \Lin_0(F_\rho)^\perp\cap  \lbrace \rho=-t\rbrace$ is at most
 $Bt^{\dim(\mathfrak{a})-1}$ for $t \ge 1$.
\end{lem}
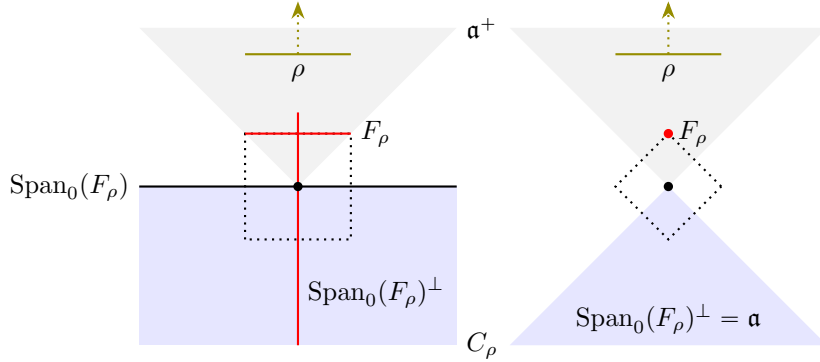
\begin{figure}[h!]
\begin{center}
\begin{tikzpicture}[>=Stealth, thick, scale=.7]

    \coordinate (A) at (-1,-1);
    \coordinate (B) at (-1,1);
    \coordinate (C) at (1,1);
    \coordinate (D) at (1,-1);
    
    \fill[gray!10] (-3,3) -- (0,0) -- (3,3);
    \fill[blue!10] (-3,-3) -- (-3,0) -- (3,0) -- (3,-3);
    \node[right] () at (3,-3) {$C_\rho$};
    \node[right] () at (3,3) {$\mathfrak{a}^+$};
    \draw[dotted] (A) -- (B) -- (C) -- (D)-- cycle;
    \draw[thick,red] (B) -- (C);
    \draw (-3,0) --(3,0);
    \draw[red, thick] (0,1.4) -- (0,-3);
    \node[right] () at (0,-2) {$\Lin_0(F_\rho)^\perp$};
    \node[left] () at (-3,0) {$\Lin_0(F_\rho)$};
    \node[right] () at (C) {$F_\rho$};
    \coordinate (BC1) at ($(B)!0.5!(C)$);
    \fill (0,0) circle (2.5pt);
    \coordinate (AC1) at ($(A)!0.5!(C)$);

    \coordinate (AB1) at ($(A)!0.5!(B)$);

    \draw[olive] ($(90:2.5)+(180:1)$) -- ($(90:2.5)+(0:1)$);
    \draw[olive, ->, dotted] ($(90:2.5)$) -- ($(90:2.5)+(90:1)$);
    \node[below] () at ($(90:2.5)$) {$\rho$};
\begin{scope}[shift={(7,0)}]

    \coordinate (A) at (1,0);
    \coordinate (B) at (0,1);
    \coordinate (C) at (-1,0);
    \coordinate (D) at (0,-1);
    
    \fill[gray!10] (-3,3) -- (0,0) -- (3,3);
    \fill[blue!10] (-3,-3) -- (0,0)  -- (3,-3);
    \draw[dotted] (A) -- (B) -- (C) -- (D)-- cycle;
    \fill[red] (B) circle (2.5pt);
    \node[right] () at (B) {$F_\rho$};
    \node () at (0,-2.5) {$\Lin_0(F_\rho)^\perp=\mathfrak{a}$};
    \coordinate (BC1) at ($(B)!0.5!(C)$);
    \fill (0,0) circle (2.5pt);
    \coordinate (AC1) at ($(A)!0.5!(C)$);

    \coordinate (AB1) at ($(A)!0.5!(B)$);

    \draw[olive] ($(90:2.5)+(180:1)$) -- ($(90:2.5)+(0:1)$);
    \draw[olive, ->, dotted] ($(90:2.5)$) -- ($(90:2.5)+(90:1)$);
    \node[below] () at ($(90:2.5)$) {$\rho$};
\end{scope}
\end{tikzpicture}
\end{center}
\caption{Illustration of Lemma \ref{lem:ConeBallN} for two metrics (with dashed unit balls) on $\mathbb{H}^2\times \mathbb{H}^2$.}
    \label{fig:Volume asymptotics part 1}
\end{figure}
\begin{proof}

\begin{figure}
    \centering
    \begin{tikzpicture}[>=Stealth, thick, scale=.7]

\fill[gray!10] (-3,3) -- (0,0) -- (3,3);

\begin{scope}[rotate={-45}]

    \coordinate (A) at (1,0);
    \coordinate (B) at (0,2);
    \coordinate (C) at (-1,0);
    \coordinate (D) at (0,-2);
    
    \fill[blue!10] (-1.5,-3) -- (0,0)  -- (1.5,-3);
    \draw[dotted] (A) -- (B) -- (C) -- (D)-- cycle;
    \fill[red] (B) circle (2.5pt);
    \node[above] () at (B) {$F_\rho$};

    \coordinate (BC1) at ($(B)!0.5!(C)$);
    \fill (0,0) circle (2.5pt);
    \coordinate (AC1) at ($(A)!0.5!(C)$);

    \coordinate (AB1) at ($(A)!0.5!(B)$);
    
    \fill[pattern={Lines[distance=5pt, angle=90,yshift=-4.1cm, line width=.5pt]}, pattern color=red] (0,-3) rectangle (3,4.2);
    \draw[line width=.5pt];
\end{scope}
\draw[olive] ($(90:2.5)+(180:1)$) -- ($(90:2.5)+(0:1)$);
    \draw[olive, ->, dotted] ($(90:2.5)$) -- ($(90:2.5)+(90:1)$);
    \node[below] () at ($(90:2.5)$) {$\rho$};
\end{tikzpicture}
    \caption{Illustration of $H_F$ in an example with $\Theta_{F_\rho}^\text{fixed}\neq \emptyset$. The dashed half space is the complement of $H_F$.}
    \label{fig:Picture for HF}
\end{figure}
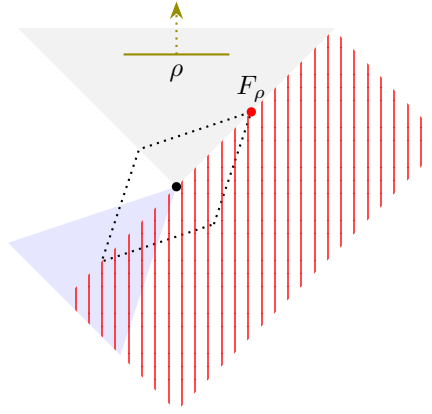

Since $\rho$ belongs to the interior of $F_\rho^\vee$, if we take any ray $\R_{\ge 0} v\subset C_\rho$, then $\rho(v)\leq 0$ with equality if and only if $\beta(v)=0$ for all $\beta\in F_\rho^\vee$, \ie if and only if $v \in \Lin_0(F_\rho)$. If moreover $v\in \Lin_0(F_\rho)^\perp$, then $\rho(v)<0$. This implies that there is some $A>0$ such that for all $\lambda \ge 0$, $ \rho(\lambda v)\leq- A N(\lambda v)$. These constants $A$ can be chosen uniformly as the set of rays in $C_\rho\cap \Lin_0(F_\rho)^\perp$ is compact.

 \medskip

The second point is a direct consequence of the first one: the diameter of the intersection of level sets of $\rho$ and $C_\rho$ grows at most linearly in $t$.
\end{proof}

\medskip

We now compute the volume of balls for the Finsler metric in the symmetric space. 

 Let $H_F=\bigcap_{\alpha\in \Theta_{F_\rho}^\text{fixed}}\lbrace \alpha\geq 0\rbrace \subset \mathfrak{a}$, where $\Theta^\text{fixed}_F$ is the set of simple roots $\alpha\in \Delta$ such that $\alpha\in \Lin_0(F^\vee)$.  
 Recall that $\Theta_{F_\rho}^\text{flip}=\emptyset$. 
 Note that we have always $\Lin_0(F)\subset H_F$.

\begin{lemma}
\label{lemma: volume of balls}
    Let $B^N_{\X}(o,t)$ denote a ball of radius $t$ in $(\X,d^N)$.
    Then 
        \[ m_t = vol(B^N_{\X}(o,t)) \sim c t^{\dim F_\rho} e^{2 t \rho_{\max}} \]
    where 
        \[ c = \text{Vol}(F_\rho)\int_{v \in \Lin_0(F_\rho)^\perp} e^{2\rho(v)}\1_{C_\rho\cap H_F} (v) {\prod_{\alpha\in \Span(\Theta_{F_\rho}^{\fixed})}}(1-e^{-2\alpha(v)})^{m_\alpha} \mathrm{d}\Leb(v) .\]
\end{lemma}

Note that $\rho \le 0$ on $C_\rho$. 
Before proving this we prove a technical lemma.

 \begin{lem}
 \label{lem: Indicator function equalities}
    There exists $\epsilon >0$, depending only on $N$, so that the following holds:
     Let $v_F$ be the unique intersection point of the affine spaces $F_\rho+\Lin_0(F_\rho)$ and $\Lin_0(F_\rho)^\perp$. Let $v\in \Lin_0(F_\rho)^\perp$. Let us write :
     $$ E(v)= \left\lbrace u\in F_\rho+\Lin_0(F_\rho) \mid \rho(v)\leq -\epsilon d^N(u,F_\rho) \right\rbrace. $$ 
     For all pairs $(u,v)\in \Lin_0(F_\rho)\times\Lin_0(F_\rho)^\perp$ and all $t\geq 1$ we have:
     $$\1_{\mathfrak{a}^+\cap B^N}(v/t+u+v_F) \leq \1_{C_\rho\cap H_F}(v)\1_{E(v/t)}(u+v_F).$$

     Moreover for almost every $(u,v) \in \Lin_0(F_\rho)\times\Lin_0(F_\rho)^\perp$:
     $$\1_{\mathfrak{a}^+\cap B^{N}}(v/t+u+v_F)\xrightarrow{t\to \infty}\1_{C_\rho\cap H_F}(v)\1_{\overline{F_\rho}}(u+v_F).$$
 \end{lem}

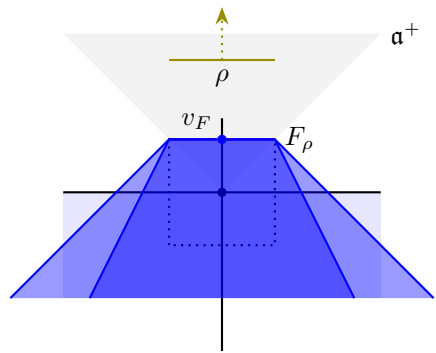
\begin{figure}[h!]
\begin{center}
\begin{tikzpicture}[>=Stealth, thick, scale=.7]

    \coordinate (A) at (-1,-1);
    \coordinate (B) at (-1,1);
    \coordinate (C) at (1,1);
    \coordinate (D) at (1,-1);
    
    \fill[gray!10] (-3,3) -- (0,0) -- (3,3);
    \fill[blue!10] (-3,-2) -- (-3,0) -- (3,0) -- (3,-2);
    \node[right] () at (3,3) {$\mathfrak{a}^+$};
    \draw[dotted] (A) -- (B) -- (C) -- (D)-- cycle;
    \draw[thick,red] (B) -- (C);
    \draw (-3,0) --(3,0);
    \draw (0,1.4) -- (0,-3);
    \node[right] () at (C) {$F_\rho$};
    \coordinate (BC1) at ($(B)!0.5!(C)$);
    \fill (0,0) circle (2.5pt);
    \fill[blue] (0,1) circle (2.5pt);
    \fill[blue, opacity=.4] (-4,-2) -- (-1,1) -- (1,1) -- (4,-2) -- cycle;
    \draw[thick,blue] (-4,-2) -- (-1,1) -- (1,1) -- (4,-2) ;
    \fill[blue, opacity=.4] (-2.5,-2) -- (-1,1) -- (1,1) -- (2.5,-2) -- cycle;
    \draw[thick,blue] (-2.5,-2) -- (-1,1) -- (1,1) -- (2.5,-2);
    \node[above left] () at (0,1) {$v_F$};
    \coordinate (AC1) at ($(A)!0.5!(C)$);

    \coordinate (AB1) at ($(A)!0.5!(B)$);

    \draw[olive] ($(90:2.5)+(180:1)$) -- ($(90:2.5)+(0:1)$);
    \draw[olive, ->, dotted] ($(90:2.5)$) -- ($(90:2.5)+(90:1)$);
    \node[below] () at ($(90:2.5)$) {$\rho$};
    
\end{tikzpicture}
\end{center}
\caption{Illustration of Lemma \ref{lem: Indicator function equalities}. The two trapezoids are respectively the set of elements of $\mathfrak{a}$ of the form $u+v_F+v$ with $(u,v)\in\Lin_0(F_\rho)\times\Lin_0(F_\rho)^\perp$ such that $u+v_F\in E(v)$ and $u+v_F\in E(v/2)$.}
    \label{fig:Volume asymptotics part 2}
\end{figure}

Here we write $\1_X(x)=1$ if $x\in X$ and $\1_X=0$ for $x\notin X$.

 \begin{proof}

 Let us prove the first part of the statement, namely we prove that if $v/t+u+v_F\in \mathfrak{a}^+\cap B^N$, then $v\in C_\rho\cap H_F$ and $u+v_F\in E(v/t)$.

Let us assume that $v/t+u+v_F\in \mathfrak{a}^+\cap B^N$. 
First we see that $v\in C_\rho$. 
Indeed let $\beta\in F^\vee$, so $\beta(v_F)= 1$ since $v_F\in F+\Lin_0(F_\rho)$, but $\beta(v/t+u+v_F)\leq N(v/t+u+v_F)\leq 1$. Since $u\in \Lin_0(F_\rho)$, $\beta(u)=0$ so $\beta(v)\leq 0$. Since this happens for all $\beta\in F^\vee$, then $v\in C_\rho$.

Now since $v/t+u+v_F\in \mathfrak{a}^+$, for all $\alpha \in \Theta_F^\text{fixed}$, $\alpha(u+v_F)=0$ as $u+v_F\in \Lin(F)$ so $\alpha(v/t+u+v_F)=\alpha(v/t)\geq 0$. Hence $v\in H_F$.

Since $N$ is convex and piecewise linear, there exists $c>0$ so that $N(u+v_F) - 1 \ge c d(u+v_F, F_\rho)$. 
Let $C$ be the maximum of $N(x)/N(-x)$ for $x$ nonzero, and let $A$ be the constant from Lemma \ref{lem:ConeBallN}.
Then 
    \[ c d(u+v_F,F_\rho) \le N(u+v_F)-1 \le N(-v/t) \le \frac{C}{A} \rho(-v/t) \]
shows that we may take $\epsilon$ to be $cA/C$.

\medskip

Now for the second part note that the intersection of $E(v/t)$ for $t\to +\infty$ is $\overline{F_\rho}$, so we already have that the right hand side is an upper bound on the limsup everywhere. Now assume that $u+v_F\in F_\rho$ and that $v\in C_\rho\cap H_F$.
For every simple root $\alpha \in \Delta\setminus  \Theta_F^\text{fixed}$, we have $\alpha(u+v_F)>0$ and hence $\alpha(v/t+u+v_F)>0$ for $t$ large enough. 
For the other roots $\alpha\in \Theta_F^\text{fixed}$, $\alpha(u+v_F)=0$ but $\alpha(v)\geq 0$ so $\alpha(v/t+u+v_F)\geq 0$ for all $t>0$. As a consequence for $t$ large enough $v/t+u+v_F\in \mathfrak{a}^+$.

Finally let us write $N(v/t+u+v_F)=\max_{\omega\in \ver((B^N)^*)}\omega(v/t+u+v_F)$. 
Note that if $\omega \in \ver((B^N)^*)\setminus \ver(F^\vee)$ then $\omega(u+v_F)< 1$ so for $t$ large enough $\omega(v/t+u+v_F)< 1$. 
For $\omega \in  \ver(F^\vee)$, $\omega(v)\leq 0$ since $v\in C_\rho$ so $\omega(v/t+u+v_F)\leq 1$ for all $t>0$.  As a consequence for $t$ large enough $v/t+u+v_F\in B^N$. This proves that for almost every pair $(u,v)$, meaning that $u\notin \partial F_\rho-v_F$, we have this convergence.
\end{proof}

We now prove the asymptotics of the volume of balls for the Finsler metric.

\begin{proof}[Proof of Lemma \ref{lemma: volume of balls}]
     First we recall the formula for volume and rescale into the unit ball.
   Recall that $B^N$ denotes the unit ball of $(\mathfrak{a},N)$, and let $B^N_t$ denote the ball of radius $t$ in $(\mathfrak{a},N)$.
        \begin{align*}
            m_t & = \text{vol}(B_{\X}^N(o,t)) \\
                & = \int_{\mathfrak{a}^+} \1_{B^N_t}(x) d \mu_{\mathfrak{a}^+}(x) \\
                & = \int_{\mathfrak{a}^+} \1_{B^N_t}(x) \prod_{\alpha\in \Sigma^+} \left( 2\sinh(\alpha(x)) \right)^{m_\alpha} \mathrm{d} \Leb(x) \\ 
                & = \int_{\mathfrak{a}^+}  e^{2\rho(x)} \Phi(x) \1_{B^N_t}(x) \mathrm{d} \Leb(x).
        \end{align*} 

    Here we lightened the notation by writing as in Remark \ref{rmk: def of Phi}:
    \[ \Phi(x) = e^{-2\rho(x)}\prod_{\alpha\in \Sigma^+} \left( 2\sinh(\alpha(x)) \right)^{m_\alpha}, \] 
    Recall that $\Phi(x) \le 1$ for $x \in \mathfrak{a}^+$, and $\Phi(x)\to 1$ whenever $\alpha(x)\to +\infty$ for all $\alpha\in \Sigma^+$. 

    In order to study the limit of this function we decompose an element $x$ of $\mathfrak{a}$ as $u_0+v_0$ with $(u_0,v_0)\in \Lin_0(F_\rho)\times \Lin_0(F_\rho)^\perp$. We fix the intersection $v_F\in \Lin_0(F_\rho)^\perp\cap (F+\Lin_0(F))$, so that in particular $\rho(v_F)=\rho_{\max}$ and apply an affine change of variable $(u,v)\in \Lin_0(F_\rho)\times \Lin_0(F_\rho)^\perp$ depending on $t$ : $u=u_0/t$, $v=v_0-tv_F$.
    Note that if $u_0\in \Lin_0(F_\rho)$, $\rho(u_0)=0$.
    Then 
    \begin{align*}
        m_t & = \int\limits_{\Lin_0(F_\rho)^\perp} \int\limits_{\Lin_0(F_\rho)} e^{2\rho(v_0)}\Phi\left( u_0+v_0 \right)\1_{\mathfrak{a}^+ \cap B_t^N}\left(u_0+v_0\right)  \mathrm{d} \Leb(u_0) \mathrm{d} \Leb(v_0) \\
        & = t^{\dim F_\rho}\int\int e^{2\rho(t(u+v_F)+v)}\Phi\left( t(u+v_F)+v \right)\1_{\mathfrak{a}^+ \cap B_t^N}\left(v+t(u+v_F)\right)  \mathrm{d} \Leb(u) \mathrm{d} \Leb(v) \\
        & =e^{2t\rho_{\max}}t^{\dim F_\rho}\int\int e^{2\rho(v)}\Phi\left( t(u+v_F)+v \right)\1_{\mathfrak{a}^+ \cap B_t^N}\left(v+t(u+v_F)\right)  \mathrm{d} \Leb(u) \mathrm{d} \Leb(v)
    \end{align*}

Now we apply Lemma \ref{lem: Indicator function equalities}.
For $t\geq 1$, if we write $E(v)=\lbrace u\in \mathfrak{a} \mid \rho(v)\leq - \epsilon d(u,F_\rho)\rbrace$ for some $\epsilon>0$ small enough, we obtain : 
$$\Phi\left( v+t(u+v_F) \right)\1_{\mathfrak{a}^+ \cap B_t^N}\left(v+t(u+v_F)\right)\leq \1_{C_\rho\cap H_F}\left(v\right)\1_{E(v)}(u+v_F) .$$

Note that for all $t\geq 1$, $E(v/t)\subset E(v)$. Hence one can apply the dominated convergence theorem to compute the limit of $m_t e^{-2t\rho_{\max}}t^{-\dim F_\rho}$ if the following integral is finite :
\[ \int\int e^{2\rho(v)}\1_{C_\rho\cap H_\rho}\left(v\right)\1_{E_{v}}(u+v_F)  \mathrm{d} \Leb(u) \mathrm{d} \Leb(v).\]

But this integral decomposes as:

\[ \int e^{2\rho(v)}\1_{C_\rho\cap H_\rho}\left(v\right)\text{vol}(E(v))\mathrm{d} \Leb(v) <+\infty.\]

This integral is finite as $\text{vol}(E(v))$ is polynomial in $-\rho(v)$, as the volume of the $\rho(v)/\rho_{\max}$-neighborhood of $F$ is $F+\Lin_0(F)$, and Lemma \ref{lem:ConeBallN} implies that the volume of the level set of $\rho$ in $C_\rho\cap \Lin_0(F)^\perp$ also grows polynomially in $-\rho(v)$.

\medskip

Now let us compute the limit. Under the assumption that $u+v_F\in F_\rho$, then the limit of $\alpha(v+t(u+v_F))$ is equal to $+\infty$, for all roots, except the ones that vanish on $F_\rho$, which are exactly the roots in $\Span(\Theta_{F_\rho}^\text{fixed})$ , for which $\alpha(u+v_F)=0$ (recall that $\Theta_{F_\rho}^{\flipped}=\emptyset$).
Using the second part of Lemma \ref{lem: Indicator function equalities} and this remark we obtain that for almost every $(u,v)\in \Lin_0(F_\rho)\times \Lin_0(F_\rho)^\perp\cap H_F$, when $t\to +\infty$:
$$\Phi\left( v+t(u+v_F) \right)\1_{\mathfrak{a}^+ \cap B_t^N}\left(v+t(u+v_F)\right)$$
$$\to \prod_{\alpha\in \Span(\Theta_{F_\rho}^{\fixed})}(1-e^{-2\alpha(v)})^{m_\alpha}\1_{C_\rho\cap H_F}\left(v\right)\1_{F_\rho}(u+v_F).$$

\medskip

The limit of $m_t e^{-2t\rho_{\max}}t^{-\dim F_\rho}$ is then the following integral, using the dominated convergence theorem:
\[ \int\int e^{2\rho(v)}\1_{C_\rho\cap H_\rho}\left(v\right)\1_{F_\rho}(u+v_F) {\prod_{\alpha\in \Span(\Theta_{F_\rho}^{\fixed})}}(1-e^{-2\alpha(v)})^{m_\alpha}\mathrm{d} \Leb(u) \mathrm{d} \Leb(v) = c .\]

\end{proof}

\begin{rem}
    Note that the peculiar expression ${\prod_{\alpha\in \Span(\Theta_{F_\rho}^{\fixed})}}(1-e^{-2\alpha(v)})^{m_\alpha}$ does not change when $v$ is replaced by $v+tv_F$ for $t\in \R$: it is invariant by translation in the direction $v_F$.
\end{rem}

Now we can prove the following.

\begin{prop}
\label{prop:Existence Limiting measure}
    The family of measures $\mu_t = \frac1{m_t}(\iota_t)_*\Leb$ converges as $t \to \infty$ up to subsequence to a $\sigma$-finite measure $\mu_N$ supported on the space of horofunctions $\mathcal{H}^N(\X)$. 
    Every such limit is $G$-invariant.
\end{prop}
 
\begin{proof}
    For all $-\infty<r<R<+\infty$, the sequence $\left(m_{t+R}/m_{t+r}\right)$ converges to $e^{2\rho_{\max}(R-r)}$ by Lemma \ref{lemma: volume of balls}. Therefore, once restricted to $C^o_{r,R}$ and up to scaling the measures by a sequence of real numbers converging to a positive number, the sequence $\frac1{m_t}(\iota_t)_*\Leb$ restricts to a sequence of probability measures on the compact $C^o_{r,R}$. Hence, up to taking a subsequence, one can assume that we have a weak-* convergence of these measures on all $C^o_{r,R}$ for $-\infty<r<R<+\infty$, and therefore up to subsequence the sequence $\frac1{m_t}(\iota_t)_*\Leb$ converges vaguely to a measure $\mu^N$. This measure is $G$-invariant since the measures $(\iota_t)_*\Leb$ are $G$-invariant.
\end{proof}

Proposition \ref{prop:Existence Limiting measure} already shows that any limiting measure $\mu_N$ up to subsequence is $G$-invariant and hence the support is a union of $G$-orbits $\mathcal{O}_{F}$.  
However we will see that the only orbit on which the limit measure has non-zero mass is $\mathcal{O}_{F_\rho}$.
Intuitively, this corresponds to the facts that $F_\rho$ is the face of the unit ball $B^N$ where $\rho$ is maximized, and $\rho$ is the direction where the volume of $\X$ concentrates.

\subsection{Computation of the limit}
\label{subsec:ComputationLimit}

As we prepare the proof of Theorem \ref{thm: Convergence measure} we give a more precise description of where the mass of the following measures concentrates: 
    \[ \mu_t = \frac1{m_t}(\iota_t)_*\Leb_{\X} \]
Let $K$ be the stabilizer in $G$ of a point $o \in \X$, and let $\mathfrak{a}$ be a Cartan subspace so that $\exp(\mathfrak{a})o$ is a maximal flat. 
Consider the identification:
    \[ \Psi \colon (\mathfrak{a} \times \R)/W \xrightarrow{\cong} \mathcal{D}^N(\X)/K \]
induced by $(v,t) \mapsto \iota_t(e^v\cdot o)$.
    
Each measure $\mu_t$ on $\mathcal{D}^N(\X)$ pushes forward to $\mathcal{D}^N(\X)/K$ and then to $(\mathfrak{a} \times \R)/W$ via $\Psi^{-1}$, yielding a measure we denote by $\mu_t^o$ on $(\mathfrak{a} \times \R)/W$.
Since $\mu_t$ is $G$-invariant (and in particular $K$-invariant), the measure $\mu_t$ is recovered from $\mu_t^o$ as the unique $K$-invariant measure agreeing on $K$-invariant measurable sets.

We need to play a similar game in the closure $\overline{\mathcal{D}^N(\X)}$. 

\begin{proposition}\label{prop: identifying a/Lin0(F)/W_F^fixed with O_F/K}
    For each $F \in Faces(N)$, we have a natural identification 
        \[ \Psi_F \colon (\mathfrak{a}/\Lin_0(F))/W_F^{\fixed} \to \mathcal{O}_F /K .\]
    Moreover, these identifications naturally extend $\Psi$ in the following sense:
   If $(v_n,t_n)$ is a sequence in $\mathfrak{a} \times \R$ such that $t_n \to \infty$ and $v_n - t_n f_0$ converges to $v' \in \mathfrak{a}$ for some $f_0 \in F$, then $\Psi([(v_n,t_n)])$ converges to $\Psi_F([v'])$ in $\overline{\mathcal{D}^N(\X)}/K$.
\end{proposition}

\begin{proof}

Consider the distance-like functions on $\mathfrak{a}$ associated to $(v_n,t_n)$:
    \begin{align*}
        \iota_{t_n}(v_n)(x) & = N(v_n -x) - t_n \\
        & = \max_{\beta \in (B^N)^\ast} \beta(v_n -x) - t_n \\
        & = \max_{\beta \in (B^N)^\ast} \beta( v_n - t_n f_0 -x) + t_n( \beta(f_0) -1)
    \end{align*}
For all $x\in \mathfrak{a}$ in a fixed compact, the term 
$\beta( v_n - t_n f_0 -x)$
stays bounded and the term $t_n(\beta(f_0)-1)$ goes to $-\infty$ except if $\beta \in F^\vee$, in which case this term vanishes. 
Therefore one has for all $n$ large enough: 
\[ \iota_{t_n}(v_n)(x) = \max_{\beta \in \ver(F^\vee)}\beta( v_n - t_n f_0 -x) = h_{F}(-(v_n-t_n f_0 -x)). \]
Then defining $\Psi_F([v'])$ to be the following horofunction 
    \[\Psi_F([v'])( x)\coloneqq h_{F}(x-v').\]
yields the desired continuity result. 

The fact that this map is well-defined on the quotient by $\Lin_0(F_\rho)$ is a consequence of Lemma \ref{lem:equivalence of translated horofunctions}. The minus sign in the definition of $\Psi$ comes from the minus sign in the definition of the standard horofunctions $h_{F_\rho}$.
\end{proof}

Recall that if $v_n \in \mathfrak{a}$ for all $n$ then $v'$ is in the cone $H_\rho$:
    \[ H_\rho \coloneqq \bigcap_{\alpha \in \Theta_{F_\rho}^{\fixed}} \{x : \alpha(x) \ge 0\}. \]
This cone is a fundamental domain for the action of $W_F^{\fixed}$ on $\mathfrak{a}$ and is invariant by translations in $\Lin_0(F)$. 

\medskip
Fixing $v_F$ as the intersection of $(F+\Lin_0(F))\cap \Lin_0(F)^\perp$, we consider the map:
    \begin{equation}\label{eqn: def of pi}
        \pi \colon (\mathfrak{a} \times \R)/W \to (\mathfrak{a}/\Lin_0(F))/W_F^{\fixed}, \quad \pi([v,s]) = [v-sv_F] .
    \end{equation}

In this expression, we chose the unique representative $v\in \mathfrak{a}^+$ of the class  $[v,s]\in (\mathfrak{a} \times \R)/W$, and we obtain an element in $H_{\rho}/\Lin_0(F_\rho)$ which is naturally identified with $(\mathfrak{a}/\Lin_0(F))/W_F^{\fixed}$. 
Intuitively, $\pi$ is a map which pushes points to the boundary, or, slightly more precisely, replaces a distance-like function with the ``best'' horofunction in the same direction, as seen from the basepoint $o$.

\begin{remark}\label{remark: vF vs f0}
    Note that for any $f_0 \in F$, the class $[v-s v_F] = [v-sf_0]$, so $\pi$ can equivalently be defined with $f_0$ in place of $v_F$. 
    This shows that still $\Psi_F$ continuously extends $\Psi$  if $v_F$ replaces $f_0$ in its definition.
\end{remark}

\medskip

Our next goal is to study the limiting behavior of the measures $\pi_\ast \nu_t^o$.
One problem with $\pi$ is that it does not preserve the cylinders $C^o_{-r,R}/K$ in general (via the identifications $\Psi^{-1}$ and $\Psi_{F_\rho}^{-1}$). 
We need in some cases to restrict the measure $\mu^o_t$ to a smaller set $S$ to obtain this behavior.

\medskip

Let $S$ be defined as the set:\footnote{As it turns out $S$ can also be described as the cone over $\lbrace x\in \mathfrak{a}^+\mid h_{F_\rho}(-x)=1,\, x\in B^N\rbrace$, which is the closure of $F_\rho$ and all faces $F'$ such that $F_\rho\subset \overline{F'}$.}
$$S=\lbrace x\in \mathfrak{a}^+\mid N(x)=h_{F_\rho}(-x)\rbrace.$$ 

Recall that $h_{F_\rho}(-x)\leq N(x)$ in general. In all the examples so far that we drew, $S=\mathfrak{a}^+$, and we illustrate in Figure \ref{fig:ImageS} the unit ball of two metrics for which $S$ is a proper subset of $\mathfrak{a}^+$. We have the following fact. 

\begin{lem}
    \label{lem:LittleLemS}
    Recall that $v_F$ is the intersection of $F_\rho+\Lin_0(F_\rho)$ and $\Lin_0(F_\rho)^\perp$. Let $w\in \Lin_0(F)$ be such that $v_F+w\in \mathfrak{a}^+$. Then $v_F+w\in S$ if and only if $v_F+w\in \overline{F_\rho}$.  If $v_F+w\in F_\rho$ then $v_F+w$ is in the interior of $S$. 
\end{lem}

\begin{proof} 

Recall that for all $\beta\in F_\rho^\vee$, $\beta(v_F+w)=1$. Hence one has $v_F+w\in \overline{F_\rho}$ if and only if $v_F+w\in \partial B^N$, which is equivalent to having $v_F+w\in S$. 

If $v_F+w\in F_\rho$, then for all $\beta\in\ver\left((B^N)^*\right)$ with $\beta\notin \ver(\overline{F^\vee})$ then $\beta(v_F+w)<1$. This remains true in a neighborhood of $v_F+w$, so $v_F+w$ is in the interior of $S$.

\end{proof}

Let $\nu^o_t$ be the measure on $(\mathfrak{a}\times \R)/W$ such that:
\begin{equation}\label{eqn: nuto}
    \mathrm{d}\nu^o_t=\1_{S\times \R}\mathrm{d}\mu^o_t.    
\end{equation}
Here $S\times \R$ is viewed as a subset of $\overline{\mathfrak{a}^+}\times \R \cong (\mathfrak{a}\times \R)/W$.

\begin{figure}
\begin{center}
\begin{tikzpicture}[scale=.6]

\draw[olive, fill=olive!20, thick]
  ({2*cos(22.5)},{2*sin(22.5)})
  \foreach \a in {67.5,112.5,...,337.5}
    { -- ({2*cos(\a)},{2*sin(\a)}) }
  -- cycle;

\fill[gray!45, opacity=0.2]
  (0,0)
  -- ({3*cos(0)},{3*sin(0)})
  -- ({3*cos(90)},{3*sin(90)})
  -- cycle;
\node at ({3.2*cos(45)},{3.2*sin(45)}) {$S$};
  
\fill[olive!45, opacity=0.65]
  (0,0)
  -- ({3*cos(22.5)},{3*sin(22.5)})
  -- ({3*cos(67.5)},{3*sin(67.5)})
  -- cycle;
\node at (2.5,.3) {$\mathfrak{a}^+$};

\draw[red, very thick]
  ({2*cos(22.5)},{2*sin(22.5)})
  -- ({2*cos(67.5)},{2*sin(67.5)});

\draw[black] (-3,0) -- (3,0);
\draw[black] (0,-3) -- (0,3);

\fill[black] (0,0) circle (1pt);

\end{tikzpicture}
\begin{tikzpicture}[scale=.6]

\draw[olive, fill=olive!20, thick]
  ({2*cos(36)},{2*sin(36)})
  \foreach \a in {72,108,...,360}
    { -- ({2*cos(\a)},{2*sin(\a)}) }
  -- cycle;

\fill[gray!45, opacity=0.2]
  (0,0)
  -- ({3*cos(0)},{3*sin(0)})
  -- ({3*cos(90)},{3*sin(90)})
  -- cycle;
\node at ({3.2*cos(45)},{3.2*sin(45)}) {$S$};
  
\fill[olive!45, opacity=0.65]
  (0,0)
  -- ({3*cos(0)},{3*sin(0)})
  -- ({3*cos(72)},{3*sin(72)})
  -- cycle;
\node at (2.5,.3) {$\mathfrak{a}^+$};

\draw[olive, thick]
  ({2*cos(0)},{2*sin(0)})
  -- ({2*cos(36)},{2*sin(36)})
  -- ({2*cos(72)},{2*sin(72)});

    \fill[red] ({2*cos(36)},{2*sin(36)}) circle (1.5pt);

\draw[black] (-3,0) -- (3,0);
\draw[black] (0,-3) -- (0,3);

\fill[black] (0,0) circle (1pt);

\end{tikzpicture}
\end{center}
\caption{Illustration of $S$ for two polyhedral Finsler metrics on $\mathbb{H}^2\times \mathbb{H}^2$ .}
\label{fig:ImageS}
\end{figure}
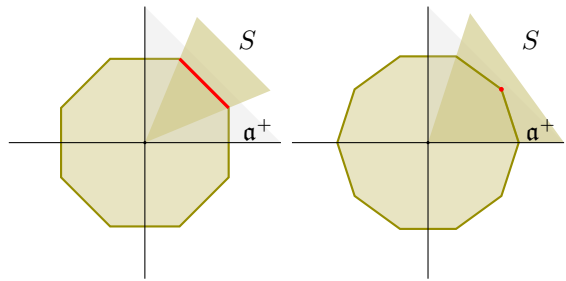

\begin{proposition}\label{prop: vague limit of measures in flat}
    Let $\Leb$ denote the pullback of the Lebesgue measure on $\Lin_0(F_\rho)^\perp$ under the identification $\Lin_0(F_\rho)^\perp \cong \mathfrak{a}/\Lin_0(F_\rho)$.
    The measures $\pi_\ast(\nu^o_t)$ converge vaguely to the measure $\mu_N^o$ satisfying: 
        \[ \mathrm{d}\mu_N^o([v]) = \1_{H_{F_\rho}}(v)e^{2\rho(v)} \frac{Vol(F_\rho)}{c} \prod_{\alpha\in \Span(\Theta_F^{\fixed})}(1-e^{-2\alpha(v)})^{m_\alpha}\mathrm{d} \Leb([v]).\]
    Here $c$ is the constant from Lemma \ref{lemma: volume of balls}.
    There is a unique $K$-invariant measure $\mu_N$ on $\mathcal{O}_{F_\rho}$ satisfying $\mu_N^o(K \cdot A) = \mu_N(K \cdot A)$ for all measurable subsets $A \subset \mathcal{O}_{F_\rho}$.
    Moreover $\mu_N(C^o_{-\infty,R}/K)= e^{2\rho_{\max}R}$ for all $R \in \R$.

\end{proposition}

The fact about the total measure will be telling us that in fact we did not loose any mass in the limit by restricting ourselves to the subset $S$, allowing us to prove the last fact.

\begin{proof}[{Proof of Proposition~\ref{prop: vague limit of measures in flat}}]
Recall that for a pair $[v,s]\in(\mathfrak{a} \times \R)/W$ to be in the support of $\mu^o_t$, and hence of $\nu$, we need to have $s=t$. 
The measure $\pi_\ast(\1_{C_\rho}\nu_t^o)$ can be written as $\phi_t(v) d \Leb([v])$ where $\phi_t(v)$ is equal to the following :
$$\frac{1}{m_t}\int_{w_0\in \Lin_0(F_\rho)}\prod_{\alpha\in \Sigma^+}(2\sinh(\alpha(v+w_0+tv_F)))^{m_\alpha}\1_{S}(v+w_0+tv_F)\mathrm{d}\Leb (w_0).$$

After a change of variable $tw=w_0$ we obtain:
$$ \int_{w\in \Lin_0(F_\rho)}\frac{t^{\dim F_\rho}}{m_t}\prod_{\alpha\in \Sigma^+}(2\sinh(\alpha(v+t(w+v_F))))^{m_\alpha}\1_{S}(v+t(w+v_F))\mathrm{d}\Leb (w).$$

We want to apply the dominated convergence to this integral, for $v$ fixed, using the estimate of $m_t$ of Lemma \ref{lemma: volume of balls}. First note that for $v$ fixed, $\1_{S}(v+t(w+v_F))\leq \1_{\mathfrak{a}^++B^N}(w+v_F)$ for all $t$ such that $N(-v)\leq t$, where $\mathfrak{a}^++B^N$ is the $1$-neighborhood of $\mathfrak{a}^+$ for $N$. 
Using this plus the control $\Phi\leq 1$ as in the proof of Lemma \ref{lemma: volume of balls} allows us to apply the dominated convergence theorem to the previous integral, for $v$ fixed, as one can bound the integral by:
$$\frac{1}{c}e^{2\rho(v)}\text{Vol}\left((\mathfrak{a}^++B^N)\cap (F+\Lin_0(F_\rho)\right).$$

Once $v$ is fixed, $\1_{S}(v+t(w+v_F))$ converges for almost all $w\in \Lin_0(F_\rho)$ towards $\1_{F_\rho}(w+v_F)$ by Lemma \ref{lem:LittleLemS}

Hence the previous integral, which depends on a parameter $v \in \mathfrak{a}$, converges to the following expression when $t\to +\infty$ via the dominated convergence theorem, as in the proof of Lemma \ref{lemma: volume of balls}:
$$\frac{1}{c}\int_{w\in \Lin_0(F_\rho)} e^{2\rho(v)} {\prod_{\alpha\in \Span(\Theta_F^{\fixed})}}(1-e^{-2\alpha(v)})^{m_\alpha}\mathrm{d}\Leb w), $$ 
which can be simplified to:
$$\text{Vol}(F_\rho)\frac{e^{2\rho(v)}}{c}\prod_{\alpha\in \Span(\Theta_F^{\fixed})}(1-e^{-2\alpha(v)})^{m_\alpha}.$$

\medskip

Finally we compute the measure of $C_{-\infty,R}^o$, for $R \in \R$, with respect to $\mu_N^o$:

\begin{align*}
    & \frac{1}{c}\int_{v\in \Lin_0(F_\rho)^\perp}e^{2\rho(v)}\1_{\lbrace h_{F_\rho}\leq R\rbrace\cap H_{F_\rho}}(v) \text{Vol}(F_\rho) {\prod_{\alpha\in \Span(\Theta_F^{\fixed})}}(1-e^{-2\alpha(v)})^{m_\alpha}\mathrm{d}\lambda(v) \\
    = & \frac{e^{2\rho(Rf_0)}}{c}\int_{v\in \Lin_0(F_\rho)^\perp}e^{2\rho(v)}\1_{\lbrace h_{F_\rho}\leq 0\rbrace\cap H_{F_\rho}}(v) \text{Vol}(F_\rho) {\prod_{\alpha\in \Span(\Theta_F^{\fixed})}}(1-e^{-2\alpha(v)})^{m_\alpha}\mathrm{d}\mathcal{L}(v) \\
    = & e^{2\rho_{\max}R}.
\end{align*}

\end{proof}

\begin{proof}[Concluding the proof of Theorem \ref{thm: Convergence measure}]

We already computed the limit of the projected measures $\pi_\ast\nu_t^o$ where $\nu^o_t$ was defined in (\ref{eqn: nuto}). 
To conclude the proof, we first check that $\mu_N^o$ is the limit of the measures $\nu_t^o$, and then we check that the differences $\mu_t^o-\nu_t^o$, which are positive measures, converge to the zero measure, so we have computed the limit of $\mu_t^o$.

\medskip

Let $f$ be a continuous and compactly supported function on $\overline{\mathcal{D}^N(\X)}/K$.
We write $\pi \colon \mathcal{D}(X)/K \to \mathcal{O}_{F_\rho}/K$ using the identifications:
    \[ \overline{\mathfrak{a}^+} \times \R \cong  \mathcal{D}(X)/K \text{ and } H_{F_\rho}/\Lin_o(F_\rho) \cong  \mathcal{O}_{F_\rho}/K.\]
Let $g=|f-f\circ \pi|$, which is a measurable function defined $\nu_t^o$ almost everywhere.

We already proved in Proposition \ref{prop: vague limit of measures in flat} that :

$$\int_{\overline{\mathcal{D}^N(\X)}/K} f\circ\pi\mathrm{d}\nu^o_t = \int_{\overline{\mathcal{D}^N(\X)}/K} f\mathrm{d}(\pi_*\nu^o_t)\to_{t\to +\infty}\int f \mathrm{d}\mu^o_N .$$

We need to prove that:
$$ \int_{\overline{\mathcal{D}^N(\X)}/K} g \mathrm{d}\nu^o_t\to_{t\to +\infty} 0. $$

Fix an $\epsilon >0$ and choose $r,R\in \mathbb{R}$ so that $C^o_{r,R}/K$ contains the supports of $f$ and $f \circ \pi$ in $S \times \R$. 

We find a compact $E_\infty \subset H_{F_\rho}/\Lin_0(F_\rho)\cong \mathcal{O}_{F_\rho}/K$ such that $\mu_N^o( E_\infty\cap C^o_{-\infty,R}/K)\geq e^{2\rho_{\max}R}-e^{2 \rho_{max}r} - \epsilon$. 
This exists since $\mu^o_N(C^o_{-\infty,R}/K)=e^{2\rho_{\max}R}$ by Proposition \ref{prop: vague limit of measures in flat}. 
Moreover we may suppose that $\mu_N^o(\partial E_\infty)=0$, and also that $r$ is small enough that $e^{2 \rho_{max}r}< \epsilon$. 

Define $E=\pi^{-1}(E_\infty)\cap (\overline{\mathfrak{a}^+} \times \R_+) \cong \mathcal{D}^N(\X)/K$, meaning the pre-image of $E_\infty$, restricted to pairs $(v,s)$ with $s\geq 0$.
Note that for each $t \in \R$, the intersection of $E$ with $\overline{\mathfrak{a}^+} \times \{t\}$ is compact, and the intersection is empty for $t<0$.

\medskip

Now we show that the closure $\overline{E}$ of $E$ in $\overline{\mathcal{D}^N(\X)}$ is included $E \cup E_\infty$.
For any sequence $(v_n,s_n)$ in $E$, the sequence $[v_n - s_n v_F]$ is in the compact set $E_\infty$, so up to subsequence $[v_n-sv_F]$ converges to some $[v_\infty]\in E_\infty$. 
If $s_n$ converges then $[v_n]$ is bounded in $\mathfrak{a}/\Lin_0(F_\rho)$, so $(\rho(v_n))$ is bounded, $v_n$ is in $\overline{\mathfrak{a}^+}$. 
Hence $(v_n)$ is bounded in $\mathfrak{a}$ and converges to a point in $E$.
If $s_n$ diverges, it goes to $+\infty$ as $s_n\geq 0$. 
Then by Proposition \ref{prop: identifying a/Lin0(F)/W_F^fixed with O_F/K} (and remark \ref{remark: vF vs f0}) the sequence converges to a point in $E_\infty$. 
Hence $\overline{E}$ is compact in $\overline{\mathcal{D}^N(\X)}$. 

\medskip

We split the following integral: 
$$\int_{\overline{\mathcal{D}^N(\X)}/K} g \1_{C^o_{-\infty,R}}\mathrm{d}\nu^o_t\leq \int_{E} g \1_{C^o_{-\infty,R}}\mathrm{d}\nu^o_t+\int_{E^c} g \1_{C^o_{-\infty,R}} \mathrm{d}\nu^o_t .$$

On $E$, one has $g(x)\to 0$ uniformly as $x\to +\infty$, again by Proposition \ref{prop: identifying a/Lin0(F)/W_F^fixed with O_F/K} and by the uniform continuity of $f$ on the compact $\overline{E}$. 

\medskip

Now on $E^c$, we show that $\nu_t^o(E^c\cap C^o_{-\infty,R}\cap S)\leq 3\epsilon $ for $t$ large enough, which controls the previous integral on $E^c$ for $t$ large enough (note  that we use here the fact that the support of $\nu^o_t$ is on $S\cap( \overline{\mathfrak{a}^+}\times \lbrace t\rbrace)$).
We use the inclusion $E^c\cap C^o_{-\infty,R}\cap S\subset \pi^{-1}(C^o_{-\infty,R}/K\setminus E_\infty)$. 
Hence the measure $\nu^o_t(E^c\cap C^o_{-\infty,R})$ is less than $\pi_*\nu^o_t(C^o_{-\infty,R}/K\setminus E_\infty)$. 
For $t$ large enough, this measure is $2\epsilon$-close to $\mu^o_N(C^o_{-\infty,R}/K\setminus E_\infty)$. 
Here we use that $r$ satisfies $e^{2\rho_{\max}r}<\epsilon$. 
By Lemma \ref{lemma: volume of balls} $\nu_t(C^o_{-\infty,R})\leq \epsilon$ for all $t$ large enough. 
Since $C^o_{r,R}$ is compact, the vague convergence of measures from  Proposition \ref{prop: vague limit of measures in flat} show that $\pi_*\nu^o_t(C^o_{r,R}/K\setminus E_\infty)$ converges to $\mu^o_N(C^o_{r,R}/K\setminus E_\infty)\leq \epsilon$ (we have chosen $E_\infty$ so that $\mu^o_N(\partial E_\infty)=0$). 

Finally $\mu^o_N(C^o_{-\infty,R}/K\setminus E_\infty)\leq \epsilon$ by construction of $E_\infty$, which concludes the proof that $\nu_t^o(E^c\cap C^o_{-\infty,R})\leq 3\epsilon $.
Hence the previous integral over $E^c$ is less than $3 \epsilon\sup(g)$ for $t$ large enough. We have this for all $R$ and $\epsilon>0$.

This concludes the proof of the fact that $\nu_t^o$ converges to $\mu_N^o$.

\medskip

Finally we compute for all $R>0$ the limit of $(\mu_t^o-\nu_t^o)(C^o_{-\infty,R})$. By the asymptotic computation of the volume of balls, $\mu_t^o(C^o_{-\infty,R})=\frac{m_{t+R}}{m_t}$ which converges to $e^{2\rho_{\max} R}$. But for all $r\in \R$, the limit of $\nu^o_t(C^o_{r,R})$ is equal to $\mu_N^o(C^o_{r,R})$ so the limit of $\nu^o_t(C^o_{-\infty,R})$ is equal at least to $\mu_N^o(C^o_{-\infty,R})$, which is equal to $e^{2\rho_{\max} R}$ by Proposition \ref{prop: vague limit of measures in flat}. This suffices to prove the vague convergence of $\mu_t^o-\nu_t^o$ to zero as any compact set in $\overline{\mathcal{D}^N(\X)}/K$ is contained in some of these cylinders.
\end{proof}

\subsection{Invariant measures on other strata}

As is well-known, invariant measures on homogeneous spaces come about in the following way.

\begin{prop}
    Let $G/H$ be a  homogeneous space with $H$ having only finitely many connected components. 
    It admits a $G$-invariant non-zero and locally finite measure if and only if for all $h\in H$, the trace of $ad_h$ on $\mathfrak{g}/\mathfrak{h}$ vanishes, in which case this measure comes from a smooth volume form and is unique up to multiplication by a scalar.
\end{prop}

\begin{proof}
    Suppose that the condition is satisfied. Pick any volume form $\omega_H$ of $\mathfrak{g}/\mathfrak{h}\simeq T_H G/H$. For any $g\in G$, define $\omega_{gH}$ to be the pushforward by left multiplication by $g$ of $\omega_H$. 
    When $H$ is connected, $\omega_H$ is $H$-invariant exactly when our assumption on traces holds.
    When $H$ has finitely many connected components $H_i$, $i\in I$, we pick any $h_i\in H_i$ for $i\in I$ and consider the volume form $\sum_{i\in I}h_i\cdot \omega_H$. This volume form is $H$-invariant, and so induces a $G$ invariant measure on $G/H$.
    
\medskip
    Conversely any $G$-invariant measure on $G/H$, together with the $H$-invariant measure on $H$ define a $G$-invariant measure on $G$, which must be the smooth Haar measure up to a scalar. 
\end{proof}

We now apply this result to the case when $H=R_F$. 
Recall that $\Aff(F)$ denotes the linear subspace of $\mathfrak{a}$ parallel to the affine span of $F$.
\begin{prop}
\label{prop: Existence measure}
    The space of horofunctions of a given type $\mathcal{O}_F\simeq G/R_F$ admits a $G$-invariant measure if and only if the following linear form on $\mathfrak{a}$ vanishes on $\Span_0(F)$:
    \begin{equation}\label{eqn: form on aff(F)}
        \sum_{\alpha \in \Sigma^+ \setminus \text{Span}(\Theta_F^\text{flip})} m_{\alpha}\alpha.    
    \end{equation}
\end{prop}

\begin{proof}
Let $h_0\in R_F$ be an element, we want to compute the trace of $\text{ad}_h$ on :
$$\mathfrak{g}/\text{Lie}(R_F)=\mathfrak{a}/\Span_0(F) \oplus \bigoplus_{\alpha \in \Sigma^+ \setminus \text{Span}(\Theta_F^\text{stable})}\mathfrak{g}_{-\alpha} \bigoplus_{\alpha \in \text{Span}(\Theta_F^\text{fixed})}\mathfrak{p}_\alpha.$$
We can write $h_0=h+k$ with $h\in \mathfrak{p}$, $k\in \mathfrak{k}$ since the trace of the adjoint action of elements in $\mathfrak{k}$ on a y invariant subspace of $\mathfrak{g}$ vanishes. Hence we it suffices to study $h$. Moreover up to conjugating $h$, one can assume that $h\in \Span_0(F)$. 
In this case, the action of $\text{ad}_h$ is diagonal in the previous decomposition. 
Recall that $\alpha \in \Theta^{\rm fixed}_F$ if and only if $\Lin(F) \subset \ker \alpha$, so
for every $\alpha \in \text{Span}(\Theta_F^\text{fixed})$, one has $\alpha(h)=0$. Hence the trace is equal to :
$$\sum_{\alpha \in \Sigma^+ \setminus \text{Span}(\Theta_F^\text{stable})} m_\alpha\alpha(h)=\sum_{\alpha \in \Sigma^+ \setminus \text{Span}(\Theta_F^\text{flip})} m_{\alpha}\alpha(h).$$
\end{proof}

One can always find some orbits satisfying this criterion. 
Let $F_{\rho}$ be the face of the unit ball of $N$ on the closure of which the linear form $\rho$ is maximal.

\begin{corollary}
\label{cor: O_F_rho has invariant measure}
    If $F$ is a subface of $F_\rho$, $\mathcal{O}_F\simeq G/R_F$ admits a $G$-invariant measure.
\end{corollary}

\begin{proof}
Recall that for a face $F$ of $B^N$, $\Theta_F^{\flipped}$ denotes the set of simple roots $\alpha\in \Delta$ whose co-root $\widehat\alpha$ is in $\Aff(F)$.
Note that when $\Theta_F^\text{flip}$ is empty, the linear form \eqref{eqn: form on aff(F)} is $2\rho$.

We first show that for a subface $F$ of $F_\rho$, it holds that $\Theta_{F}^{\flipped} = \emptyset$.
Indeed, if there is some simple root $\alpha\in \Delta$ whose co-root $\widehat{\alpha}$ is in $\Aff(F)\subset \Aff(F_\rho)$, then $\rho(\widehat{\alpha})=0$.
This is not possible since $\rho$ lies in the model Weyl chamber in $\mathfrak{a}^*$ by Proposition \ref{prop: rho in terms of fundamental weights}.

\medskip
Now that $\Theta_F^\text{flip}=\emptyset$, one can directly apply Proposition \ref{prop: Existence measure} to conclude.
\end{proof}

Note that other orbits than the one associated to $F_\rho$ can admit invariant measures. 

\begin{example}
\label{exmp:Existence Invariant Measures}
If $G=G_1\times G_2$, and $N$ a polyhedral norm on the Cartan subalgebra for $G$ defined as $\max(N_1,N_2)$ for polyhedral norms on the Cartan subalgebras of $G_1,G_2$, then the face $F_\rho^{1}$ maximizing the sum of all positive roots of $G_1$ admits an invariant measure. Indeed $\Sigma^+ \setminus \text{Span}(\Theta_{F_\rho^1}^\text{flip})=\Sigma^+_{G_1}$ and hence the sum \eqref{eqn: form on aff(F)} is maximized on $F_\rho^1$.

\medskip

    For $G=\SL(n,\R)$ and $N=N_{\omega_1}$ as in Section \ref{subsec:Symmetric spaces examples}, One can check that all the faces $F_i$, $1\leq i \leq n-1$ satisfy this property, and hence admit $G$-invariant measures. Indeed for $F=F_i$, $\Sigma^+ \setminus \text{Span}(\Theta_F^\text{flip})$ is the sum \eqref{eqn: form on aff(F)} equals a linear form $2\rho_i+\nu$ where $\rho_i$ is the half sum of all positive roots of a copy of $\SL(i,\R)$ in $\SL(n,\R)$, these roots being of the form $\lambda_j-\lambda_{k}$, $1\leq j,k\leq i$, and $\nu$ a sum of roots of the form $\lambda_j-\lambda_{k}$, $1\leq j\leq i<k$. We see that this linear form $2\rho_i+\nu$ is maximized on $F_i$.
\end{example}

\section{Ideal Poisson Voronoi tessellations}
\label{sec:IdealPoisson}
In this section we study the properties of the ideal Poisson--Voronoi tessellations associated to Finsler metrics on symmetric spaces. 
These are the tessellations associated to the Poisson point process on $\mathcal{O}_{F_\rho} \subset \mathcal{H}^N(\X)$ with respect to the limiting measure $\mu_N$ constructed in Section~\ref{sec:Measures}.

\subsection{Poisson point processes}
We recall some background on Poisson point processes. 
We refer the reader to \cite{LP17} as a reference for the results we recall. Let $(X,\mu)$ be a locally compact, second-countable metrizable space equipped with a Radon measure with no atoms. A random variable $\Pi$ with values in the set $\mathcal{S}$ of locally finite subsets of X \footnote{The $\sigma$-algebra on $\mathcal{S}$ that we consider is the one generated by the sets $\lbrace A\in \mathcal{S} \mid k=|A\cap U|\rbrace $ for all measurable set $U\subset X$ and all $k\geq 0$.} is called a \emph{Poisson point process with intensity $\mu$} if it satisfies the following two axioms:
\begin{itemize}
    \item[-] for all disjoint measurable sets $U,V\subset X$, the random variables $U\cap \Pi$ and $V\cap \Pi$ are independent,
    \item[-] for all measurable sets $U$, $\mathbb{E}(|U\cap \Pi|)=\mu(U)$.
\end{itemize}

Such a random variable exists for all Radon measure with no atoms $\mu$ on $X$, and the law of this variable is determined by $\mu$. More precisely for all measurable sets $U$ the random variable $|U\cap\Pi|$ follows a Poisson law of parameter $\mu(U)$, meaning that for all $k\geq 0$, $\mathbb{P}(|U\cap\Pi|=k)=e^{-\mu(U)}\frac{\mu(U)^k}{k!}$.

\medskip

Poisson point processes are very natural, and satisfy the Mecke formula, which we recall in Proposition~\ref{prop:Mecke}. We give the standard and the bivariate version here.

\begin{prop}\label{prop:Mecke}
Let $\Pi$ be a Poisson point process with intensity $\mu$. For every measurable positive function $\phi:\mathcal{S}\times X\to \R_{\geq 0}$ on the space of locally finite collections of points:
    $$\mathbb{E}\left(\sum_{f\in \Pi}\phi(\Pi,f)\right)=\int \mathbb{E}\left(\phi(\Pi\cup \lbrace f\rbrace,f) \right) \mathrm{d}\mu(f).$$
The bivariate version of this formula states that for every measurable non-negative function $\phi:\mathcal{S}\times X\times X\to \R_{\geq 0}$ on the space of locally finite collections of points: 
    $$\mathbb{E}\left(\sum_{f \ne g\in \Pi}\phi(\Pi,f,g)\right)=\int \int \mathbb{E}\left(\phi(\Pi\cup \lbrace f, g\rbrace,f,g) \right)d\mu(f) \mathrm{d}\mu(g).$$
\end{prop}

This will allow us to verify certain almost sure properties of Poisson point process, by fixing one or two points.

\begin{lem}
\label{lem:MeckeApplied}
    Let $\mathcal{M}$ be a measurable subset of $\mathcal{S}_X \times X \times X$. 
    Let $\Pi$ be a Poisson point process with intensity measure $\mu$ on $X$. 
    Suppose that for $(\mu \times \mu)$-almost all $(f,g) \in X \times X$, one has almost surely that $(\Pi \cup \lbrace f,g\rbrace, f,g) \in \mathcal{M}$. 
    Then almost surely, for all $f \ne g \in \Pi$, $(\Pi,f,g) \in \mathcal{M}$.
\end{lem}

\begin{proof}
    Apply Proposition \ref{prop:Mecke} to the indicator function of the complement of $\mathcal{M}$.
\end{proof}

\subsection{Tessellations of symmetric spaces}

A locally finite \emph{tessellation} of a locally compact metric space $X$, $(\mathcal{C}_h)_{h\in \Pi_0}$, is a countable collection $\mathcal{T}$ of closed sets $\mathcal{C}\subset X$ such that:
\begin{itemize}
\item[-]$X=\bigcup_{\mathcal{C}\in \mathcal{T}}\mathcal{C},$
\item[-]for every distinct $\mathcal{C}_1,\mathcal{C}_2\in \mathcal{T}$, $\mathcal{C}_1\cap \mathcal{C}_2$ has empty interior.
\item[-] for every compact $K\subset X$, $\lbrace \mathcal{C}\in \mathcal{T}\mid \mathcal{C}\cap K\neq \emptyset\rbrace$ is finite.
\end{itemize}
Let $\text{Tess}(X)$ be the space of locally finite tessellations of $\X$.

\medskip

Let $\Pi_0$ be a Poisson point process on $\mathcal{O}_{F_\rho}$ with measure intensity $\mu_N$. The corresponding Poisson--Voronoi tessellation is the collection $(\mathcal{C}_h)_{h\in \Pi_0}$ where for $h\in \Pi_0$:
$$\mathcal{C}_h=\lbrace x\in \X \mid \forall h'\in \Pi_0, h(x)\leq h'(x)\rbrace.$$

\begin{prop}\label{prop: Tessellation}
    Almost surely $(\mathcal{C}_h)_{h\in \Pi_0}$ is a locally finite tessellation of $\X$.
\end{prop}

\begin{proof}
    For any point $o\in \X$ recall that given $R\in \R$, $C^o_{-\infty,R}$ is the set of horofunctions $h\in \overline{\mathcal{D}^N(\X)}\supset \mathcal{O}_{F_\rho}$ such that $h(o)\leq R$. We know by Theorem \ref{thm: Convergence measure} that for every $R>0$, $\mu_N(C^o_{-\infty,R})<+\infty$ and that when $R\to +\infty$, $\mu_N(C^o_{-\infty,R})\to +\infty$. For every $R>0$ the random variable $\#\lbrace h\in \Pi_0\mid h(o)\leq R\rbrace$ follows a Poisson law with parameter $\mu_N(C^o_{-\infty,R})$, hence it is almost surely finite, and almost surely it is non zero for $R$ large enough. In particular there exist $h\in \Pi_0$ that minimizes $h(o)$ and hence $o\in \mathcal{C}_h$. t
    This proves that $\X=\bigcup_{h\in \Pi_0}\mathcal{C}_h$.

    \medskip

    The $G$-orbit $\mathcal{O}_{F_\rho}$ in the space of horofunction is a fiber bundle over a flag manifold $\mathcal{F}_{\Delta\setminus {\Theta^{\fixed}_{F_\rho}}}$.
    The projection of the measure is $K$-invariant and hence is comparable with the Lebesque measure. In particular almost surely every pair of points in $\Pi_0$ are opposite.
    Hence by Theorem \ref{thm:Empty Interior} almost surely for every $h,h'\in \Pi_0$, the set $\lbrace x\in \X \mid h(x)=h'(x)\rbrace$ has empty interior. In particular almost surely for all $h\neq h'\in \Pi_0$, $\mathcal{C}_h\cap \mathcal{C}_{h'}$ has empty interior.

    \medskip
    
  If we fix a compact $V\subset \X$ it is contained in the ball centered at some $o\in \X$ with radius $R'>0$. Almost surely $o\in \mathcal{C}_h$ for some $h\in \Pi_0$, and also almost surely $\Pi_0\cap C^o_{-\infty,h(o)+2R'+1}$ is finite. for every $h'\in \Pi_0$ with $h'(o)\geq R'+2h(o)+1$, the $1$-Lipschitz property of $h,h'$ implies that for all $x\in V$, $h'(x)\geq  h(x)+1$ so $x\notin \mathcal{C}_{h'}$, proving that the tessellation is locally finite.
\end{proof}

\subsection{Sides of the ideal Poisson--Voronoi tessellation.}

Given a collection of points $\Pi\subset \overline{\mathcal{D}^N(\X)}$ we define the associated \emph{Voronoi tessellation} $\lbrace\mathcal{C}_h\rbrace_{h\in \Pi}$ so that:
$$\mathcal{C}_h=\lbrace x\in \X\mid \forall h'\in \Pi,\, h'(x)\geq h(x)\rbrace.$$

\begin{thm}
\label{thm:ProperBissectro dichotomy}
   Let $F_\rho$ be the face of the unit ball of $N$ whose closure is the maximal points of the linear form $\rho$. 
   Suppose that $\mathcal{O}_{F_\rho}$ has the proper-on-bisector property, then almost surely for every pair of opposite horofunctions $f\neq g\in \Pi_N$, the corresponding side in the tessellation is bounded:
   $$\mathcal{C}_f\cap \mathcal{C}_g=\lbrace x\in \X\mid f(x)=g(x),\, \forall h\in \Pi_N,\, h(x)\geq f(x)  \rbrace.$$

   Conversely if $\mathcal{O}_{F_\rho}$ does not have the proper-on-bisector property, then almost surely for every pair $f\neq g\in \Pi_N$, the corresponding side in the tessellation is unbounded.
\end{thm}

We delay the proof of this theorem to the end of this subsection.

\begin{rem}
Two horofunctions $h_1,h_2$ are \emph{opposite} exactly if the corresponding flags in $G/P_\Theta$, $\Theta=\Delta\setminus \Theta_{F_\rho}^\text{stable}$ are in generic position. Hence the set of transverse pairs has measure $1$ for any measure equivalent to the Lebesgue measure on $\mathcal{O}_{F_\rho}$.
\end{rem}

We first see that the behavior of the intersection of a cell with a domain $E$ in the ideal Poisson Voronoi diagram is dictated by the growth of the corresponding horofunction on this domain.

We will need the following lemma, which is similar to \cite[{Theorem 3.8}]{FMW23poissonvoronoi}.
\begin{lem}
\label{lem:Matrix Coefficient Vanishing}
For all $A\subset \mathcal{O}_{F_\rho}$ of finite measure for $\mu_N$, one has $\mu_N(A\cap g_nA)\to 0$ for every diverging sequence $(g_n)$ in $G$.

\end{lem}

The proof is more complex for semisimple groups than for simple groups.

\begin{proof}
We consider the unitary representation $\mathrm{L}^2(\mathcal{O}_{F_\rho}, \mu_N)$ of $G$. Decompose $G$ into simple factors $G=G_1\times\cdots \times G_n$, up to taking a finite quotient. Let us first prove that this representation does not admit any $G_i$-invariant vector. The orbit $\mathcal{O}_{F_\rho}=G/R_{F_\rho}$ projects $G_i$-equivariantly onto $G_i/R_i$ where $R_i$ is the projection onto $G_i$ of $R_{F_\rho}$. Hence the invariant vector can be projected to this homogeneous space, and must be trivial if $R_i$ is included to either a maximal compact subgroup or a parabolic subgroup. Indeed the pushforward of $f^2\mu_N$ would give a finite $G_i$-invariant measure on the symmetric space or a flag manifold. 

We now use Propositon \ref{prop:Lie algebra of stabilizer}. If there is some root of $G_i$ in $\Delta\setminus \Theta_{F_\rho}^\text{stable}$ then $R_i$ belongs to a parabolic subgroup. Else all roots of $G_i$ are in $\Theta^\text{fixed}_{F_\rho}$ ($\Theta^\text{flipped}_{F_\rho}$ is empty) and hence $G_i\subset G^\text{fixed}$, and $R_i$ is in a maximal compact subgroup of $G_i$. 

\medskip

We proved that there is no $G_i$-invariant element in $\mathrm{L}^2(\mathcal{O}_{F_\rho}, \mu_N)$. The Lie group $G$ satisfies the hypothesis of \cite[{Theorem 9.1}]{BFGM07}, hence every matrix coefficient goes to zero for any diverging sequence $(g_n)$ in $G$, meaning that for all $f_1,f_2\in \mathrm{L}^2(\mathcal{O}_{F_\rho}, \mu_N)$:
$$\int f_1 (f_2\circ g_n^{-1})\mathrm{d}\mu_N\to 0.$$

Applying this to $f_1=f_2=\1_A$ yields the desired result.

\end{proof}

\begin{lem}
\label{lem:3 conditions}
    Let $h_0\in \mathcal{O}_{F_\rho}$ be a horofunction and let $E\subset \X$ be a subset. Let $\mathcal{C}_{h_0}$ be the cell corresponding to $h_0$ in the Voronoi diagam for $\Pi_0\cup \lbrace h_0 \rbrace$, where $\Pi_0\subset \mathcal{O}_{F_\rho}$ is a Poisson point process with measure intensity $\mu_N$.
    \begin{itemize} 
    \item[(i)] Suppose that $h_0$ has linear growth on $E$ , meaning that for some $A,B>0$, for all $x\in E$:
$$h_0(x)\geq Ad^N(o,x)-B.$$

Then almost surely $\mathcal{C}_{h_0}\cap E$ is bounded.

\item[(ii)] Suppose that $E$ is unbounded. If $h_0$ is not proper and bounded from below on $E$, then almost surely $\mathcal{C}_{h_0}\cap E$ is unbounded.

\item[(iii)]  Suppose that $-h_0$ has linear growth on $E$,  meaning that for some $A,B>0$, for all $x\in E$:
$$h_0(x)\leq -Ad^N(o,x)+B.$$

Suppose also that the $1$-neighborhood $E_1$ of $E$ has subexponential volume growth, meaning that the volume of $E_1\cap B(o,t)=o(e^{\epsilon t})$ for all $\epsilon>0$. Then almost surely $E\setminus \mathcal{C}_{h_0}$ is bounded.
\end{itemize}

\end{lem}

Part $(iii)$ of the Lemma is not used in this section, and will only be used in a different form in Lemma \ref{lem:3 conditionsbis}, in order to prove Theorem \ref{thm:Selberg Main}.

\begin{proof}

Let us first prove the point $(i)$.
    Let $(x_i)_{i\in I}$ be a set of points in $E$ such that every point in $E$ is at distance at most $2$ from some $x_i$, and so that for $i\neq j$, $d(x_i,x_j) \geq 1$. Let $M_i$ be the event that the ball of radius $2$ abound $x_i$ intersects non-trivially $\mathcal{C}_{h_0}$, i.e such that for some $x$ with $d(x,x_i)\leq 2$, for all $h\in \Pi_0$, $h_0(x)\leq h(x)$.

    The probability of $M_i$ is less than the probability that $\Pi_0$ contains no horofunction $h$ such that $h(x_i)\leq Ad^N(o,x_i)-B-4$. The measure of the set of all such $h\in \mathcal{O}_{F_\rho}$ is equal to at least $e^{A'd(o,x_i)-B'}$ for some $A',B'>0$, using Proposition \ref{prop: vague limit of measures in flat}. Recall that if a random variable $X$ follows a Poisson law with intensity $\lambda$, $\mathbb{P}(X=0)=e^{-\lambda}$. In particular :
$$\mathbb{P}(M_i)\leq e^{-e^{A'd(o,x_i)-B'}}.$$

The number of $i\in I$ such that $d(o,x)\leq n$ is at most $e^{Dn}$ for some $D>0$ since the balls of radius $1$ around the $x_i$ are disjoint, of fixed volume, and the volume of balls in the symmetric space $\X$ is exponential. Therefore:

$$\sum_{i\in I}\mathbb{P}(M_i)\leq \sum_{n\geq 0} e^{Dn}e^{-e^{A'n-B'}}.$$
This sum converges, so by the Borel-Cantelli lemma, almost surely only finitely many of the events $M_i$ occur.

\medskip

We now move to $(ii)$. This part of the proof is very close to the argument in \cite[{Proposition 6.2}]{FMW23poissonvoronoi}. 
We take an infinite diverging sequence $(x_i)_{i\in \N}$ in $E$ such that for some $m>0$ for all $i\in \N$, $h(x_i)< m$, that exists since $h$ is not proper on $E$. 
Let $N_i$ be the event $\Pi_0\cap C^{x_i}_{-\infty,m}=\emptyset$, which implies in particular that $x_i\in \mathcal{C}_{h_0}$.
The probability of the event $N_i$ is equal as previously to $e^{-\mu_N(C^{x_i}_{-\infty,m})}\geq e^{-e^{2\rho_{max}M}}$.

Up to taking a subsequence of the diverging sequence $(x_i)$, and thanks to Lemma \ref{lem:Matrix Coefficient Vanishing}, one can assume that for all $i<j\in \N$:
$$\mu_N(C^{x_i}_{-\infty,m}\cap C^{x_j}_{-\infty,m})\leq \log(1+2^{-i-j}).$$

Hence $\mathbb{P}(N_i\cap N_j)=e^{-\mu_N(C^{x_i}_{-\infty,m}\cup C^{x_j}_{-\infty,m})}\leq \mathbb{P}(N_i)\mathbb{P}(N_j)(1+\frac{1}{2^{i+j}})$. This and the fact that $\mathbb{P}(N_i)\geq \exp(-e^{2\rho_{\max}m})$ allows us to get that the following limit is equal to $1$:

$$\lim_{n\to +\infty}\frac{(\sum_{i=1}^n\mathbb{P}(N_i))^2}{\sum_{i,j=1}^n\mathbb{P}(N_i\cap N_j)}= 1.$$

Using the Kochen-Stone generalisation of the Borel-Cantelli Lemma \cite{KS64}, we obtain that infinitely many of the events $N_i$ occur.

\medskip

 Finally let us prove $(iii)$. We consider as previously a set of points $(x_i)$ and we denote by $\overline{M}_i$ the complementary event to $M_i$, \ie the probability that the ball of radius $2$ abound $x_i$ contains no point $x$ such that $f(x)> h(x)$ for some $h\in \Pi_N$.

 \medskip

 By the definition of the Poisson distribution, as before we have for some $A',B'>0$:
 $$\mathbb{P}(\overline{M}_i)\leq 1-e^{-e^{-A'd(o,x_i)+B'}}\leq e^{-A'd(o,x_i)+B'} $$

 The number of $i\in I$ such that $d(o,x)\leq n$ for $n$ large enough is subexponential by hypothesis since the balls of radius $1$ around the $x_i$ are disjoint. Therefore for some $A'',B''>0$:
$$\sum_{i\in I}\mathbb{P}(\overline{M}_i)\leq \sum_{n\geq 0} e^{-A''n+B''}.$$
This sum converges, so by the Borel-Cantelli lemma, almost surely only finitely many of the events $\overline{M}_i$ occur.
\end{proof}

\begin{proof}[Proof of Theorem \ref{thm:ProperBissectro dichotomy}]
Using the Mecke formula (Lemma \ref{lem:MeckeApplied}) the theorem boils down to showing that given two $f,g\in \mathcal{O}_{F_\rho}$ in generic position, if $\Pi_0$ is a Poisson point process on $\mathcal{O}_{F_\rho}$ with measure intensity $\mu_N$ and $(\mathcal{C}_h)_{h\in \lbrace f,g\rbrace\cup \Pi_0}$ is the corresponding Voronoi diagram, the side $ \mathcal{C}_f\cap \mathcal{C}_g$ is bounded or unbounded depending on the behavior of $f$ on it. Note that  $\mathcal{C}_f\cap \mathcal{C}_g=\mathcal{C}_f\cap \Bis(f,g)$. 

\medskip

We use the characterisation of the proper-on-bisector property from Proposition \ref{prop:MainProp of proper horofunctions}. If $\mathcal{O}_{F_\rho}$ has the proper-on-bissector property, part (i) of Lemma \ref{lem:3 conditions} implies that almost surely $\Bis(f,g)\cap \mathcal{C}_f$ is bounded, which concludes that the side is bounded. 

If $\mathcal{O}_{F_\rho}$ does not have the proper-on-bissector property, part (ii) of Lemma \ref{lem:3 conditions} implies that almost surely $\Bis(f,g)\cap \mathcal{C}_f$ is unbounded. This concludes the proof.

\end{proof}

As a consequence we get the following corollary, giving examples, of ideal Poisson--Voronoi tessellation with \emph{bounded sides}, 
meaning either that almost surely all pairs of cells share at most a bounded side, respectively with \emph{unbounded sides}, meaning that almost surely all pairs of cells share an unbounded side.

Let $n$ be the rank of $\X$ and let $(\tau_i)_{1\leq i\leq n}$ be the renormalized fundamental coweights defined in Section \ref{subsection:Tau_i}, that satisfy $\rho(\tau_i)=1$ for all $i$.

\begin{corollary}
Let $\X$ be symmetric space of non-compact type with a polyhedral Finsler metric associated to a norm $N_\omega=\max_{w\in W}w\cdot \omega$ for a linear form $\omega\in \mathfrak{a}^*$.

\begin{enumerate}
\item 
Let $n$ be the rank of $\X$ and let $(\tau_i)_{1\leq i\leq n}$ be the renormalized fundamental co-weights defined in Section \ref{subsection:Tau_i}, that satisfy $\rho(\tau_i)=1$ for all $i$.
If for some $1\leq i\leq n$, $\forall j\neq i,\;\omega(\tau_i)<\omega(\tau_j)$, then the ideal Poisson--Voronoi tessellation for $d_{N_\omega}$ has bounded sides. 

In particular, if $\X$ is a rank-one symmetric space, the ideal Poisson--Voronoi tessellation has bounded sides. 
    \item 
Let the rank of $\X$ be bigger or equal to two, and let $\omega=\rho$, then the ideal Poisson--Voronoi tessellation for $d_{N_\rho}$ has unbounded sides. 

\end{enumerate}
\end{corollary}

The first part of this corollary applies to generic elements $\omega$. 
For example, when $\mathfrak{g} = \mathfrak{sl}(d,\R)$, one can take $\omega = \varpi_k$ for any $k \ne d/2$. 
However, when $d=4$ and $\omega = \varpi_2$, the norm $N_{\varpi_2}$ does not have the proper-on-bisector property, as follows from Proposition \ref{prop:ProperbissectorLinear forms}. In this case $G=\SL(4,\R)$ we illustrate for which linear form $\omega$ the corrollary applies in Figure \ref{fig:SL4Wslinar forms}.

\begin{figure}[ht]
    \centering
    \begin{tikzpicture}[xscale=2, yscale=2]
        \draw (-1,0) -- (1,0) -- (0,1) -- cycle;
        \node[above left] (A) at (0,1/3) {$\rho$};
        \draw[red, thick] (0,1/3) -- (0,1);
        \draw[blue, thick] (0,1/3) -- (-.5,0);
        \draw[blue, thick] (0,1/3) -- (.5,0);
        \node[below] (B) at (0,0) {$\mathbb{P}(\mathfrak{a}^*)$};
        \node[left] (B) at (-1,0) {$\varpi_1$};
        \node[right] (B) at (1,0) {$\varpi_3$};
        \node[above] (B) at (0,1) {$\varpi_2$};
        \fill[red] (0,1) circle (.025);
        \fill (1,0) circle (.025);
        \fill (-1,0) circle (.025);
        \fill[red] (0,1/3) circle (.025);
    \end{tikzpicture}
    \caption{Set of dominant linear forms $\omega$ in $\mathbb{P}(\mathfrak{a}^*)$ for the rank $3$ group $\SL(4,\R)$. 
    In red are the linear forms $\omega$ such that the orbit $\mathcal{O}_{F_\rho}$ for the metric $N_\omega$ does not have the proper-on-bisector property.
    The linear forms where $F_\rho$ is not a vertex is the union of red and blue points.}
    \label{fig:SL4Wslinar forms}
\end{figure}
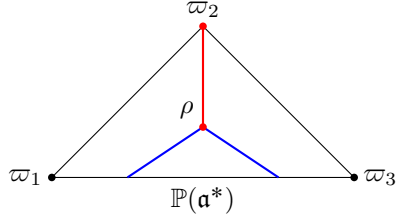

\begin{proof}
The proper-on-bisector property is immediatly determined in these cases by Proposition \ref{prop:ProperbissectorLinear forms}.
\end{proof}

\subsection{Illustration of the results in  \texorpdfstring{$\mathbb{H}^2\times \mathbb{H}^2$}{H2*H2}}
\label{subsec:ExamplesV2}

In this section we keep the notations from Section \ref{subsec:H2H2}. We discuss the proper-on-bissector property and the consequences for the ideal Poisson--Voronoi tessellation.

\medskip

Be begin with the metric on $\mathbb{H}^2\times \mathbb{H}^2$ defined from the norm $N=N_\infty$. The orbit of horofunctions $\mathcal{O}_{F_\rho}$  has the proper-on-bissector property, as illustrated in Figure \ref{fig:H2Example1}. Hence almost surely the side between any two cells of the ideal Poisson--Voronoi tessellation for the measure $\mu_{N_\infty}$ is bounded. 

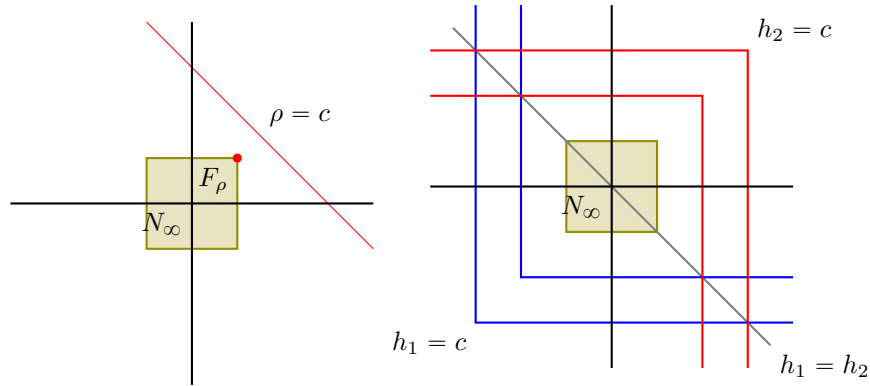
\begin{figure}[htbp!]
\begin{center}
\begin{tikzpicture}[scale=1.2]

    \draw[olive, fill=olive!20, thick] (.5,.5) -- (-.5,.5)--(-.5,-.5) --(.5,-.5)-- cycle;

    \draw[red] (2,-0.5) -- (-0.5,2);
    \fill[red] (.5,.5) circle (.05);

    \draw[thick] (2,0) -- (-2,0);
    \draw[thick] (0,2) -- (0,-2);
    \node[below left] (A) at (.5,.5) {$F_\rho$};
    \node[below left] (A) at (0,0) {$N_\infty$};
    \node[above right] (A) at (.75,.75) {$\rho=c$};
\end{tikzpicture}
\begin{tikzpicture}[scale=1.2]

    \draw[olive, fill=olive!20, thick] (.5,.5) -- (-.5,.5)--(-.5,-.5) --(.5,-.5)-- cycle;
    
    \draw[gray, thick] (-1.75,1.75) -- (1.75,-1.75);

    \draw[blue, thick] (-1,2) -- (-1,-1) -- (2,-1);
    \draw[blue, thick] (-1.5,2) -- (-1.5,-1.5) -- (2,-1.5);
    \draw[red, thick] (1,-2) -- (1,1) -- (-2,1);
    \draw[red, thick] (1.5,-2) -- (1.5,1.5) -- (-2,1.5);

    \draw[thick] (2,0) -- (-2,0);
    \draw[thick] (0,2) -- (0,-2);
    
    \node[below right] (A) at (1.75,-1.75) {$h_1=h_2$};
    \node[above right] (A) at (1.5,1.5) {$h_2=c$};
    \node[below left] (A) at (-1.5,-1.5) {$h_1=c$};
    \node[below left] (A) at (0,0) {$N_\infty$};

\end{tikzpicture}
\end{center}
\caption{Illustration of $F_{\rho}$ for the norm $N_\infty$ for $\X=\mathbb{H}^2\times \mathbb{H}^2$, as well as the level lines of two opposite horofunctions on the orbit $\mathcal{O}_{F_\rho}$.}
\label{fig:H2Example1}
\end{figure}
\medskip

For $N=N_1$, the orbit of horofunctions $\mathcal{O}_{F_\rho}$ does not have the proper-on-bisector property, as illustrated in Figure \ref{fig:H2Example2}. The sides of the cells of the ideal Poisson--Voronoi tessellation hence are almost surely all unbounded. This recovers a result of D'Achille \cite{DAchille2024ipvtbeyondhyperbolic}.

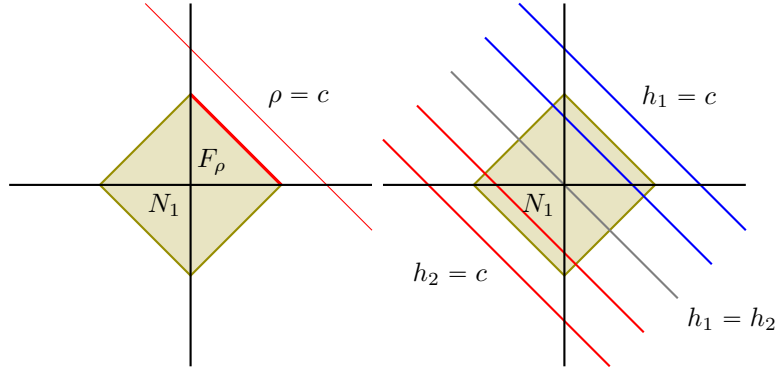
\begin{figure}[htbp!]
\begin{center}
\begin{tikzpicture}[scale=1.2]

    \draw[olive, fill=olive!20, thick] (1,0) -- (0,1)--(-1,0) --(0,-1)-- cycle;

    \draw[red] (2,-0.5) -- (-0.5,2);
    \draw[red, very thick] (1,0) -- (0,1);

    \draw[thick] (2,0) -- (-2,0);
    \draw[thick] (0,2) -- (0,-2);
    \node[below left] (A) at (.5,.5) {$F_\rho$};
    \node[below left] (A) at (0,0) {$N_1$};
    \node[above right] (A) at (.75,.75) {$\rho=c$};
\end{tikzpicture}
\begin{tikzpicture}[scale=1.2]

    \draw[olive, fill=olive!20, thick] (1,0) -- (0,1)--(-1,0) --(0,-1)-- cycle;
    
    \draw[gray, thick] (-1.25,1.25) -- (1.25,-1.25);

    \draw[blue, thick] (2,-0.5) -- (-0.5,2);
    \draw[blue, thick] (1.625,-.875) -- (-.875,1.625);
    \draw[red, thick] (-2,0.5) -- (0.5,-2);
    \draw[red, thick] (-1.625,.875) -- (.875,-1.625);

    \draw[thick] (2,0) -- (-2,0);
    \draw[thick] (0,2) -- (0,-2);
    
    \node[below right] (A) at (1.25,-1.25) {$h_1=h_2$};
    \node[above right] (A) at (.75,.75) {$h_1=c$};
    \node[below left] (A) at (-.75,-.75) {$h_2=c$};
    \node[below left] (A) at (0,0) {$N_1$};

\end{tikzpicture}
\end{center}
\caption{Illustration of $F_{\rho}$ for the norm $N_1$ for $\X=\mathbb{H}^2\times \mathbb{H}^2$, as well as the level lines of two opposite horofunctions on the orbit $\mathcal{O}_{F_\rho}$.}
\label{fig:H2Example2}
\end{figure}

\medskip

For $N=N_{a,b}$ and $a>b>0$, the orbit of horofunctions $\mathcal{O}_{F_\rho}$ changes with respect to the previous case, and it satisfies the proper-on-bisector property, as illustrated in Figure \ref{fig:H2Example3}. Hence the sides of the cells of the ideal Poisson--Voronoi tessellation for the measure $\mu_{N_{a,b}}$ are almost surely all bounded.

\begin{figure}[htbp!]
\begin{center}
\begin{tikzpicture}[scale=1.2]

    \draw[olive, fill=olive!20, thick] (.5,0) -- (0,1)--(-.5,0) --(0,-1)-- cycle;

    \draw[red] (2,-0.5) -- (-0.5,2);
        \fill[red] (0,1) circle (.05);

    \draw[thick] (2,0) -- (-2,0);
    \draw[thick] (0,2) -- (0,-2);
    \node[left] (A) at (0,1) {$F_\rho$};
    \node[below left] (A) at (-.5,0) {$N_{2,1}$};
    \node[above right] (A) at (.75,.75) {$\rho=c$};
\end{tikzpicture}
\begin{tikzpicture}[scale=1.2]

    \draw[olive, fill=olive!20, thick] (.5,0) -- (0,1)--(-.5,0) --(0,-1)-- cycle;
    
    \draw[gray, thick] (-2,0) -- (2,0);

    \draw[blue, thick] (-1.9,-2) -- (0,1.8) -- (1.9,-2);
    \draw[blue, thick] (-1.7,-2) -- (0,1.4) -- (1.7,-2);

    \draw[red, thick] (1.9,2) -- (0,-1.8) -- (-1.9,2);
    \draw[red, thick] (1.7,2) -- (0,-1.4) -- (-1.7,2);
    \draw[thick] (0,2) -- (0,-2);
    
    \node[above] (A) at (2,0) {$h_1=h_2$};
    \node[right] (A) at (1.45,1) {$h_2=c$};
    \node[left] (A) at (-1.35,-.75) {$h_1=c$};

\end{tikzpicture}
\end{center}
\caption{Illustration of $F_{\rho}$ for the norm $N_{2,1}$ for $\X=\mathbb{H}^2\times \mathbb{H}^2$, as well as the level lines of two opposite horofunctions on the orbit $\mathcal{O}_{F_\rho}$.}
\label{fig:H2Example3}
\end{figure}
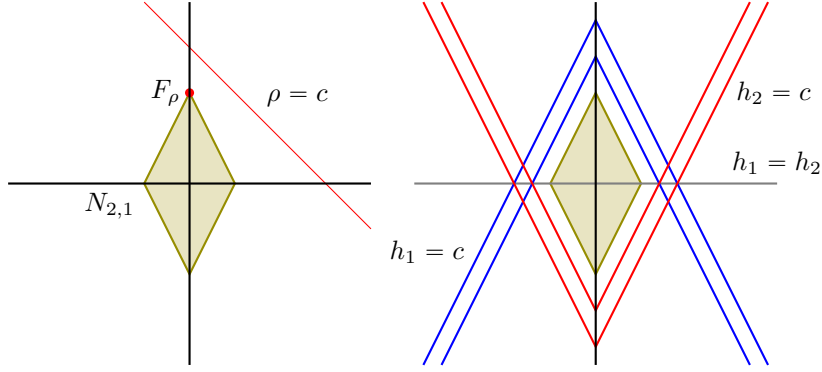

\subsection{Products of symmetric spaces}
\label{subsec:Corollaries}
We give more corollaries of Theorem \ref{thm:ProperBissectro dichotomy} for products of symmetric spaces, describing cases  for which the ideal Poisson--Voronoi tessellation has \emph{bounded sides}, respectively \emph{unbounded sides}. meaning either that almost surely all pairs of cells share at most a bounded side, respectively always shares a unbounded side.

\medskip

We first consider products of symmetric spaces (real) rank $1$ of non-compact type, for instance $\mathbb{H}^n$ or $\mathbb{H}^n_\C$. Every such group admits a unique invariant Finsler metric which is the Riemannian metric $d$ associated with the Euclidean norm on $\R\simeq \mathfrak{a}$. By convention we say that the normalized Riemannian distance is the one such that $\rho(v)=1$ for $v\in \mathfrak{a}$ of norm $1$. 

\begin{corollary}
    Let $\X=\mathcal{X}_1\times \cdots \times \mathcal{X}_d$ be a product of rank one symmetric spaces of non-compact type, with $\mathcal{X}_i$ equipped with the normalized Riemannian distance $d_i$. 
Let us consider for some positive $0<a_1\leq a_2\leq \cdots\leq a_n$ the polyhedral Finsler metric:
    $$d_a((x_i),(y_i))=\sum_{i=1}^n a_i d_i(x_i,y_i).$$
\begin{itemize}
\item[-] If $a_1<a_2$, then the ideal Poisson--Voronoi tessellation for $d_a$ has bounded sides,
\item[-] If $a_1=a_2$, then the ideal Poisson--Voronoi tessellation for $d_a$ has unbounded sides.
\end{itemize}
Let us consider for some positive $a_1\geq a_2\geq \cdots\geq a_n>0$ the polyhedral Finsler metric:
    $$d_\infty((x_i),(y_i))=\max(d_i(x_i,y_i)).$$
Then the ideal Poisson--Voronoi tessellation for $d$ has bounded sides.
\end{corollary}

\begin{proof}
The Cartan subalgebra for the associated Lie group $G$ is naturally $\R^n$, and note that the opposition involution is trivial, as the Weyl group is just the group of elements acting like $\Diag(\pm1,\cdots \pm1)$. Note also that $\rho=\rho_1+\cdots +\rho_n$ is the sum of the half sums of roots for the symmetric spaces $\mathcal{X}_i$.

For the metric $d_a$, the face $F_\rho$ has dimension $k$, which is the maximum integer such that $a_1=a_k$. By Corollary \ref{cor: proper bisectors when opposition involution is trivial}, $\mathcal{O}_{F_\rho}$ has the proper-on-bisector property if and only if $k=1$, which concludes the proof.

For the metric $d_\infty$, the face $F_\rho$ is the vertex $(1,\cdots, 1)$, which concludes the proof once again by Corollary \ref{cor: proper bisectors when opposition involution is trivial}.
    
\end{proof}

Now for a product of any two symmetric spaces, we obtain the following corollaries. These corollaries naturally extend to products with more than two factors, and hence are stronger than the previous one.

\begin{corollary}
Let $\mathcal{X}_1,\mathcal{X}_2$ be two symmetric spaces of non-compact type  associated to semi-simple groups $G_1,G_2$ with polyhedral Finsler metrics associated to norms $N_1,N_2$ on their Cartan subalgebras $\mathfrak{a}_1,\mathfrak{a}_2$. 

Given the norm $N_\infty=\max(N_1,N_2)$ the ideal Poisson--Voronoi tessellation for $d_{N_\infty}$ has bounded sides if and only if both the ideal Poisson--Voronoi tessellation for $d_{N_1}$ and $d_{N_2}$ have bounded sides.

\medskip

Given the polyhedral norm $N=N_1+N_2$:
\begin{itemize}
    \item[-] If $N_1(\rho_1)>N_2(\rho_2)$, then the ideal Poisson--Voronoi tessellation for $d_N$ has bounded sides if and only if the ideal Poisson--Voronoi tessellation for $d_{N_1}$ have bounded sides.
    \item[-] If $N_i(\rho_i)=N_i(\rho_i)$ and both opposition involutions for $G_1,G_2$ are trivial, then the ideal Poisson--Voronoi tessellation for $d_N$ has unbounded sides.
\end{itemize}
\end{corollary}
In this statement, we write $N_i(\rho_i)$ for the maximum value of the linear form $\rho_i\in \mathfrak{a}_i^*$ on the unit ball of $N_i$ in $\mathfrak{a}_i$, which is reached on $F_{\rho_i}$. In the second case if we remove the triviality of the opposition involution, the boundedness if the sides of the Poisson--Voronoi tessellation can be checked from Proposition \ref{prop:Caracterisation proper bisectors Frho}.

\begin{proof}

For $N_\infty$, note that $F_\rho$ is naturally $F_{\rho_1}\times F_{\rho_2}$. In particular $\Lin_0(F_\rho)=\Lin_0(F_{\rho_1})\oplus \Lin_0(F_{\rho_2})$. Let $(w_0^1,w_0^2)$ be the opposition involution on $\mathfrak{a}=\mathfrak{a}_1\oplus \mathfrak{a}_2$. The intersection from Proposition \ref{prop:Caracterisation proper bisectors Frho}, that characterizes the proper-on-bisector property for $F_\rho$, is empty if and only if it is empty for $\mathcal{X}_1$ and $\mathcal{X}_2$. This concludes the proof.

\medskip

For $N$, note that $F_\rho$ is equal to $F_{\rho_1}$ in the first case and is equal to $F_{\rho_1}\oplus F_{\rho_2}$ in the second. As before the opposition involution on $\mathfrak{a}=\mathfrak{a}_1\oplus \mathfrak{a}_2$ writes as $(w_0^1,w_0^2)$. Hence in the first case we obtain once again the result by Proposition \ref{prop:Caracterisation proper bisectors Frho}. In the second case $F_{\rho}$ is the affine span of $F_{\rho_1}$ and $F_{\rho_2}$, which is not a vertex, hence giving the desired result by Corollary \ref{cor: proper bisectors when opposition involution is trivial}.

\end{proof}

\begin{rem}
    We believe that this result also holds in the case $N_1(\rho_1)>N_2(\rho_2)$ when the metric $N_2$ is a general Finsler norm that is not polyhedral, for instance for Euclidean metrics. However we do not define the ideal Poisson--Voronoi tessellation for such non-polyhedral metrics in this paper.
\end{rem}

\section{Tessellation with convex tiles}
\label{sec:Selberg}

Aside from the exceptional case of real hyperbolic space, the cells of a Voronoi tessellation with respect to the Riemannian metric are typically not convex.

However when $G=\PSL(n,\R)$, the associated symmetric space admits a projective model: it can be identified with an open convex set in a large projective space, so that $\PSL(n,\R)$ acts by conjugation. 
More concretely it can be seen as the space of projective positive symmetric tensors in $\mathbb{P}(S^2 \mathbb{R}^n)$, or dually as the space of positive symmetric $n$ by $n$ matrices up to a scalar.

\medskip

In this setup, there is an invariant $\mathfrak{s}:\X \times \X \to \R$, called the Selberg invariant, which is at bounded distance from the polyhedral Finsler metric associated to the norm $N=N_{\omega_1}$. This invariant has the advantage that its bisectors are projective hyperplane in the projective model. Consequently the tiles of a  Voronoi diagram with respect to the Selberg invariant are convex polyhedra in the projective model. 
Examples of such tessellations are illustrated Figure \ref{fig:CellSymSpace} and Figure \ref{fig:CellSymSpace2}.

It turns out that there is a natural identification between  horofunctions for $\mathfrak{s}$ and horofunctions for $d_{N_{\iota\omega_1}}$.

\subsection{The Selberg invariant and the dual Satake compactification.}

\label{subsec:Selberg}
We begin with some linear algebra to define the two projective models for the symmetric space $\X$. Let $V$ be a $n$-dimensional real vector space. In what follows, making a distinction between $V$ and $V^*$ will be important. Let $S^2V$ be the second symmetric power of $V$, or in other words the space of symmetric bilinear forms on $V^*$. The space of symmetric bilinear forms on $V$ is $\left(S^2V\right)^*\simeq S^2V^*$. Let $S_{>0}^2V\subset S_{\geq 0}^2V\subset S^2V$ be the cones of positive and semi-positive bilinear forms on $V^*$. Simillarly let $S_{>0}^2V^*\subset S_{\geq 0}^2V^*\subset S^2V^*$ be the cones of non-zero positive and semi-positive bilinear forms on $V$.

\medskip

Let us fix a volume form on $V$. This also defines a volume form on $V^*$. Let $\mathcal{X}\subset S_{>0}^2V$ be the set of elements of determinant one, which makes sense if we view them as linear maps $V^*\to V$. Let similarly $\mathcal{X}^*\subset S_{>0}^2V^*$ be the set of elements of determinant one, if we view these elements as linear maps $V\to V^*$. These spaces $\mathcal{X}, \mathcal{X}^*$ can both naturally be identified with the symmetric space $\X=\SL(V)/SO(V)$, we call them respectively the \emph{projective model} and the \emph{dual projective model} of $\X$. Given an element $X\in \mathcal{X}$ viewed as a map $V^*\to V$, we obtain the corresponding element of $\mathcal{X}^*$ by considering $X^{-1}:V\to V^*$.

\medskip

\begin{defn}
Given two points $X\in \mathcal{X}$ and $Y\in \mathcal{X}^*$, the \emph{Selberg invariant} is equal to :
$$\mathfrak{s}(X,Y)= \frac{1}{2}\log\frac{1}{n}\Tr(XY).$$ 
Here $X$ is again viewed as a linear map $V\to V^*$ and $Y$ as a linear map $V^*\to V$.
\end{defn}

This invariant can be viewed coarsely as a distance, up to indetifyting $\mathcal{X}$ and $\mathcal{X}^*$ with $\X$.

\begin{prop}
    Let $X\in \mathcal{X}$ and $Y\in \mathcal{X}^*$, $\mathfrak{s}(X,Y)\geq 0$ with equality if and only if $Y= X^{-1}$, \ie if $X$ and $Y$ correspond to the same point in $\X$.
\end{prop}

\begin{proof}
    The endomorphsim $XY:V\to V$ has determinant one, it is symmetric with respect to the scalar product defined by $X$ and positive. Hence if we write $\lambda_1,\cdots \lambda_n>0$ for its eigenvalues, one has $\Tr(XY)=\lambda_1+\cdots+\lambda_n$ and $1=\lambda_1\cdots \lambda_n$, so $\Tr(XY)\geq n$ with equality if and only if $X^{-1}=Y$.
\end{proof}

The Selberg invariant defines a map $\X\times \X \to \R_{>0}$ which does not satisfy the triangular inequality, but which is close to the distance associated to a Finsler metric.

\begin{rem}

When $n=2$, one can identify $\X$ with the hyperbolic plane, and in this case $\mathfrak{s}(x,y) = \frac12 \log \cosh(d(x,y))$.

\end{rem}

Recall that $\omega_1$ is the linear form on the restricted Cartan sublagebra $\mathfrak{a}=\lbrace\Diag(\lambda_1, \cdots, \lambda_n)\mid \lambda_1+\cdots+\lambda_n=0 \rbrace $ of $\mathfrak{sl}_n(\R)$, that associates $\lambda_1$ to the diagonal matrix $\Diag(\lambda_1, \cdots, \lambda_n)$. This linear form defines a norm $N_{\omega_1}$ on $\mathfrak{a}$ such that :$$N_{\omega_1}\left(\lambda_1, \cdots, \lambda_n\right)=\max_{1\leq i\leq n}\lambda_i.$$
This norm defines a Finsler distance $d_{\omega_1}$ on $\X$ so that $d_{\omega_1}(x,y)=N_{\omega_1}(\vec{d}(x,y))$. The action opposition involution $\iota$ on $\mathfrak{a}$ exchanges any asymmetric norm $N$ with $\iota N(v)=N(-v)$. In particular :
$$N_{\iota\omega_1}\left(\lambda_1, \cdots, \lambda_n\right)=\max_{1\leq i\leq n}-\lambda_i.$$ 

\begin{prop}
\label{prop:Compariason SelbergDistance}
Given $X\in \mathcal{X}\simeq \X$ and $Y\in \mathcal{X}^*\simeq \X$; $d_{\omega_1}(Y,X)$ is half the logarithm of the largest eigenvalue of $XY:V^*\to V^*$. In particular:
$$d_{\omega_1}(Y,X)-\frac12 \log(n)\leq \mathfrak{s}(X,Y)\leq d_{\omega_1}(Y,X).$$
$$d_{\iota\omega_1}(X,Y)-\frac12 \log(n)\leq \mathfrak{s}(X,Y)\leq d_{\iota\omega_1}(X,Y).$$
\end{prop}

\begin{remark}
In what follows, we want to consider a tessellation of $\mathcal{X}$ with convex tiles, so we want $X\in \mathcal{X}$ to be the first argument. Hence the Finsler metric that will appear will be $d_{\iota\omega_1}$. 
\end{remark}
\begin{proof}
Recalll that $g\in \SL(V)$ acts on $S^2V$ by $g\cdot X= g Xg^{*}$, where $g^*:V^*\to V^*$ is the dual of $g$, when $X$ is viewed as an element $V^*\to V$. The vector valued distance between $X\in \mathcal{X}$ and $Y\in \mathcal{X}^*$ is the element $(\lambda_1, \cdots, \lambda_n)\in \mathfrak{a}$ such that $\lambda_1\geq \lambda_2 \geq \cdots \geq \lambda_n$ and the eigenvalues of $XY:V^*\to V^*$ are $(e^{2\lambda_1}, \cdots, e^{2\lambda_n})$. We can rewrite $\mathfrak{s}(X,Y)$ as 
$$\frac{1}{2}\log\left(\frac{1}{n}\sum_{i=1}^ne^{2\lambda_i}\right).$$
From that and the expression of $d_{\omega_1}$ we deduce the result.
    
\end{proof}

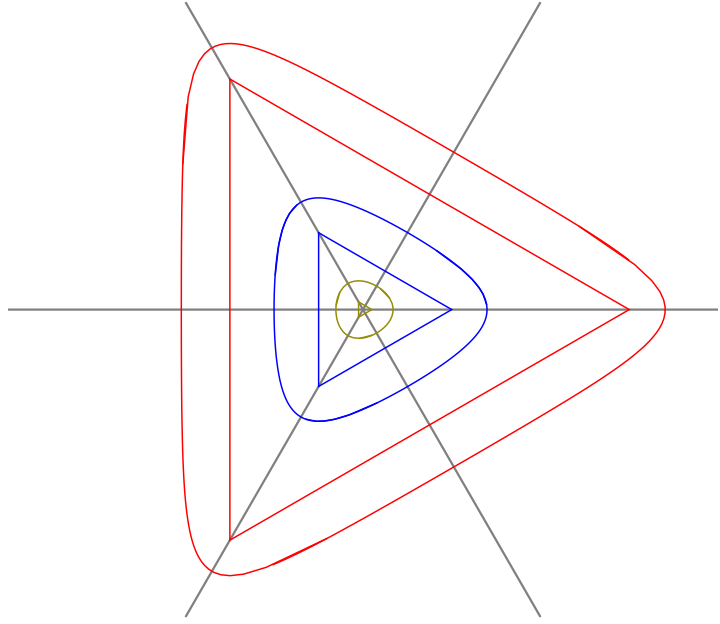
\begin{figure}[h!]
\begin{center}
\begin{tikzpicture}[scale=.7]
\begin{axis}[
      hide axis,
    axis equal,
    samples=100,
    xmin=-4, xmax=4,
    ymin=-4, ymax=4,
    width=15cm,    
    height=15cm,
]
    \draw[very thick, gray] (2*2,0) -- (0,0) --(-2*2*.5,2*2*0.866);
    \draw[very thick, gray] (-2*2,0) -- (0,0) --(2*2*.5,-2*2*0.866);
    \draw[very thick, gray] (0,0) -- (-2*2*.5,-2*2*0.866);
    \draw[very thick, gray] (0,0) -- (2*2*.5,2*2*0.866);

\draw[blue, thick,] (1,0) -- (-.5,0.866) --(-.5,-0.866)  -- cycle;
\draw[olive, thick] (0.1*1,0.1*0) -- (-0.1*.5,0.1*0.866) -- (-0.1*.5,-0.1*0.866)  -- cycle;

 \draw[red, thick] (3*1,3*0) -- (-3*.5,3*0.866) -- (-3*.5,-3*0.866)  -- cycle;

\pgfmathdeclarefunction{A}{1}{\pgfmathparse{3*exp(1)*exp(-#1) - exp(-3*#1)}}

\pgfmathdeclarefunction{B}{1}{\pgfmathparse{3*exp(0.25)*exp(-#1) - exp(-3*#1)}}

\pgfmathdeclarefunction{C}{1}{\pgfmathparse{3*exp(0.1)*exp(-#1) - exp(-3*#1)}}

\pgfmathdeclarefunction{D}{1}{\pgfmathparse{3*exp(3)*exp(-#1) - exp(-3*#1)}}
\addplot[blue, thick, domain=-2:2]
  ({x},
   { ln( (A(x) + sqrt(A(x)^2 - 4))/2)/sqrt(3) }
  );
\addplot[blue, thick,transform canvas={rotate around={120:(0,0)}}, domain=-2:2]
   ({x},
   { ln( (A(x) + sqrt(A(x)^2 - 4))/2)/sqrt(3) }
   );
\addplot[blue, thick,transform canvas={rotate around={240:(0,0)}}, domain=-2:2]
   ({x},
   { ln( (A(x) + sqrt(A(x)^2 - 4))/2)/sqrt(3) }
   );
      
\addplot[olive, thick, domain=-2:2]
  ({x},
   { ln( (C(x) + sqrt(C(x)^2 - 4))/2)/sqrt(3) }
  );
\addplot[olive, thick,transform canvas={rotate around={120:(0,0)}},, domain=-2:2]
   ({x},
   { ln( (C(x) + sqrt(C(x)^2 - 4))/2)/sqrt(3) }
   );
\addplot[olive, thick,transform canvas={rotate around={240:(0,0)}}, domain=-2:2]
   ({x},
   { ln( (C(x) + sqrt(C(x)^2 - 4))/2)/sqrt(3) }
   );

   \addplot[red, thick, domain=-3:3]
  ({x},
   { ln( (D(x) + sqrt(D(x)^2 - 4))/2)/sqrt(3) }
  );
\addplot[red, thick,transform canvas={rotate around={120:(0,0)}},, domain=-3:3]
   ({x},
   { ln( (D(x) + sqrt(D(x)^2 - 4))/2)/sqrt(3) }
   );
\addplot[red,thick,transform canvas={rotate around={240:(0,0)}}, domain=-3:3]
   ({x},
   { ln( (D(x) + sqrt(D(x)^2 - 4))/2)/sqrt(3) }
   );

\end{axis}
    
\end{tikzpicture}
\end{center}
\caption{Illustration of the Selberg invariant level sets and the balls of the norm $N_{\omega_1}$ on a model flat for $\dim(V)=3$. The 3 different colors illustrate the balls of radius $3$, $1$ and $0.1$. For $r$ small the $r$-balls for the Selberg invariant become close to the $r^{\frac{1}{2\dim(V)}}$-balls for the Riemmanian metric.}
\label{fig:Selberginvariant}
\end{figure}

On a model flat, the Finsler metric $\omega_1$ can be viewed as a maximum of linear forms, and the invariant $\mathfrak{s}$ is a smoothening of this maximum.

\medskip

Let $\mathcal{H}(\mathcal{X})$ be the horofunction compactification of $\mathcal{X}^*$ defined using the Selberg invariant. More precisely we can view points in $\mathcal{X}^*$ as functions  $\mathfrak{s}(\cdot,Y)$ on $\mathcal{X}$, given $Y\in \mathcal{X}^*$. If we consider the space of functions $\mathfrak{s}(\cdot,Y)+c$ for $x\in \mathcal{X}^*$, $c\in \R$ we obtain a set of the form $\mathcal{H}(\mathcal{X})\cup \mathcal{X}^*$. 
However in this case the space $\mathcal{H}(\mathcal{X})$ admits a simple description in terms of linear algebra.

\medskip

It is easily checked that, given a sequence $Y_n \in \mathcal{X}^\ast$ and $\lambda_n \to 0$, the forms $\lambda_n Y_n$ have a non-zero limit in $S^2_{\geq 0} V^*$ if and only if the functions on $\mathcal{X}$ given by $\mathfrak{s}(\cdot,Y_n)+\log(\lambda_n)$ converge.
So we may also view $\mathcal{H}(\mathcal{X})$ as the set $S^2_{\geq 0} V^*\setminus S^2_{> 0} V^*$ of non-zero, semi-positive symmetric bilinear forms on $V$ that are not positive. 
From now on we consider these spaces naturally identified.

We denote by $\partial\mathcal{X}^*$ the projectivisation of $\mathcal{H}(\mathcal{X})$ which defines a compactification of $\mathcal{X}^*\simeq\mathbb{P}(S^2_{>0}V^*)$ we call this compactification  the \emph{dual Satake compactification}. This is naturally a compactification of $\mathcal{X}^*$, and we will see that it is related to the horofunction compactification of $d_{\iota\omega_1}$. Each element $Y_0\in \mathcal{H}(\mathcal{X})$ has naturally associated a function $X\in \mathcal{X}\mapsto \log\frac{1}{n}\Tr(XY_0)$. These functions are horofunctions for $\mathfrak{s}$ in the sense that there are limits up to constant functions of functions of the form $X\mapsto \mathfrak{s}(X,Y)$. 

\begin{remark}
We will also consider the \emph{Satake compactification} $\partial \mathcal{X}$, which is natually a compactification of $\mathcal{X}$, and is the projectivisation of $\mathcal{H}(\mathcal{X}^*)=S^2_{\geq 0} V\setminus S^2_{> 0} V$.
\end{remark}

The action of $\SL(V)$ on $\mathcal{H}(\mathcal{X})$ decomposes into $n-1$ orbits $\mathcal{O}_1, \mathcal{O}_2, \cdots, \mathcal{O}_{n-1}$ where $\mathcal{O}_k$ is the set of semi-positive elements $Y_0\in S^2 V^*$ of corank $k$ (of rank $n-k$). 
The orbit $\mathcal{O}_k$ has a natural map to $\Gr(k,V)$ given by taking radicals: $Y \mapsto \rad(Y)$, where
    \[ \rad(Y) \coloneqq \{ v \in V : Y(v,w)=0, \, \forall w \in V \} .\]
The fibers are the inner products on $V/\rad(V)$. 
In particular, $\mathcal{O}_1$ fibers over $\mathbb{P}(V)$ with fiber over $\ell$ the space of inner products on $V/\ell$, and $\mathcal{O}_{n-1}$ is naturally an $\R$-bundle over $\mathbb{P}(V^\ast) = \Gr(d-1,V)$. 
On the other side, the unit ball for $N_{\iota\omega_1}$ admits (up to the action of the Weyl group) exactly one face for every $1\leq i\leq n-1$.

The space of horofunctions $\mathcal{H}^{N_{\iota \omega_1}}(\X)$ decomposes into orbits $\mathcal{O}_{F_k}$ for $1\leq k\leq n-1$ where $F_k$ is the unique face of the unit ball of $N_{\iota \omega_1}$ of dimension $k-1$ intersecting $\mathfrak{a}^+$ . 
Concretely if $\mathfrak{a}=\lbrace \Diag(\lambda_1, \cdots ,\lambda_n)\mid \lambda_1+\cdots +\lambda_n=0 \rbrace $ then the face is :
$$F_k=\left\lbrace \Diag(\lambda_1,\lambda_2,\cdots, \lambda_{k},-1,\cdots ,-1)\mid \lambda_{1}+\cdots +\lambda_{k}=n-k,\, \lambda_i\geq -1\right\rbrace.$$ 

\begin{prop}
\label{prop:Identification Selberg and Finler}
    The natural identification $\X\simeq \mathcal{X}$ extends to a $G$-equivariant homeomorphism $\phi: \overline{\mathcal{D}^{N_{\iota\omega_1}}(\mathbb{X})}\to S^2{V^*}^{\geq 0}$ such that:
\begin{itemize}
\item[(i)] For all $X\in \X\simeq \mathcal{X}$, for all $h\in \overline{\mathcal{D}^{N_{\iota\omega_1}}(\X)}$ :
$$h(X)-\frac{1}{2}\log(n)\leq \mathfrak{s}(X,\phi(h))\leq h(X),$$
\item[(ii)] $\phi$ is equivariant with respect to the $\R$-action defined by $\lambda\cdot h=(X \mapsto h(X)+\lambda)$ and $\lambda \cdot Y=e^{2\lambda} Y$ for $\lambda\in \R$, $h\in \overline{\mathcal{D}^N(\X)}$ and $Y\in S^2V^{\geq 0}$.
\item[(iii)]  This homeomorphism identifies $\mathcal{O}_{F_k}\subset \overline{\mathcal{D}^{N_{\iota\omega_1}}(\mathbb{X})}$ with $\mathcal{O}_k\subset S^2{V^*}^{\geq 0}$.
\end{itemize}
\end{prop}

This homeomorphism on the horofunction compactification is given in \cite[{Theorem 5.5}]{haettel2018horofunctioncompactificationssymmetricspaces}. However we verify that it satisfies the additional properties.

\begin{proof}
If we fix a model pointed flat, the stabilizer of the model horofunction $h_{F_k}\in \mathcal{O}_{F_k}$, as described in Section \ref{subsec:The stabilizer of a horofunction} is the set of matrices that can be written as:

$$M=\begin{pmatrix}
B & 0 \\
A & U
\end{pmatrix}.$$
Here $(B,U)\in S(\GL(k,\R)\times\SO(n-k,\R))$ and $A$ is any $(n-k)\times k$ matrix. This subgroup is equal to the stabilizer of the semi-positive bilinear form $Y_0\in S^2 {V^*}^{\geq 0}\setminus S^2{V^*}^{>0 }$, which is associated to the following symmetric matrix written in block form with blocks of size $(k,n-k)$:
$$Y_0=\begin{pmatrix}
0 & 0 \\
0 & \Id
\end{pmatrix}.$$
Here we view this matrix as a bilinear from, using the standard scalar product associated with the base-point $o\in \mathcal{X}$.
Hence as homogeneous spaces one can identify $\mathcal{O}_{F_k}=G/R_{F_k}$ with $\mathcal{O}_k$ so that $\phi(h_{F_k})=Y_0$.

\medskip

Hence one can define a bijection $\phi: \overline{\mathcal{D}^{N_{\iota\omega_1}}(\mathbb{X})}\to S^2(V^*)^{\geq 0}$ 
that identifies for $Y\in \X\simeq \mathcal{X}^{-1}$ and $c\in \R$ the function $d(\cdot,Y)+c$ with $e^{2c}Y$, so that $\mathfrak{s}(\cdot,Y)+c=\mathfrak{s}(\cdot , e^{2c}Y)$. This identifies each $\mathcal{O}_{F_k}$ with $\mathcal{O}_k$ as previously described. This map clearly satisfies $(ii)$ and $(iii)$.

\medskip

In order to check $(i)$, one can use Proposition \ref{prop:Compariason SelbergDistance}, so it only remains to consider the case when $Y\in \mathcal{O}_{F_k}$, $1\leq k\leq n-1$.

The standard horofunction $h_{F_k}\in \mathcal{O}_{F_k}$ has the following formula, given the choice of a basepoint  $X_0\in \mathcal{X}$ and a model flat, by Proposition \ref{prop:Explicit formula for Horofunctions}:
 $$h_{F_k}(X)= \max_{\alpha\in \Ker(Y_0)^*} \frac{1}{2}\log(X(\alpha, \alpha)/X_0(\alpha,\alpha)).$$
 Here $\Ker(Y_0)^*\subset V^*$ is the space of linear forms vanishing on $\Ker(Y_0)\subset V$.

 \medskip

 We write $X$ as a symmetric matrix with the following block form:
 $$X=\begin{pmatrix}
 * & * \\
 * & X'
 \end{pmatrix}.$$

 We obtain that $\mathfrak{s}(X,\phi(h_{F_k}))=\frac{1}{2}\log(\frac{1}{n}\Tr(X'{X_0}_{|\Ker(Y_0)^*}))$ and $h_{F_k}(X)$ is half the operator norm of $X'$, for the metric $X_{|\Ker(Y_0)^*}$. The inequality $(i)$ follows directly from that.

 \medskip

 Now let us check that $\phi$ is continuous. This is an adaptation of the arguments of \cite[Theorem 5.5]{haettel2018horofunctioncompactificationssymmetricspaces}. Let $(h_m)_{m\in \N}$ be a sequence in $\overline{\mathcal{D}^{N_{\iota\omega_1}}(\X)}$ that converges to $h_\infty \in \mathcal{H}^{N_{\iota\omega_1}}(\X)$. We will prove that for any subsequence of $(h_m)$, there is a sub-subsequence that converges to $\phi(h_\infty)$. 
 
 First one can assume that $h_m=d(\cdot,y_m)+t_m$ for some $y_m\in \X$ and $t_m\in \R$, up to replacing $(h_m)$ by a distance functions that is close enough, so that the limit remains unchanged. 

 Now fix $X_0\in \mathcal{X}$ a basepoint in the model flat as previously and write $y_m=k_m\cdot \exp(\iota v_m)\cdot o$ with $v_m\in \mathfrak{a}^+$ in the model Cartan subalgebra. 
 Up to taking a subsequence one can assume that $k_m$ converges to $k_\infty\in K$.
 Let $\mathcal{O}_{F_k}$ be the orbit of the space of horofunctions containing $h_\infty$. 
 This means that $v_m=\Diag(\lambda_1^{(m)},\cdots, \lambda_n^{(m)})\in \mathfrak{a}^+$ is such that 
 $(\lambda_1^{(m)}-t_m,\cdots, \lambda_k^{(m)}-t_m)$ converge to some real numbers 
 $(\lambda_{1}^\infty,\cdots, \lambda_{n-k}^\infty)$ but $\lambda_{k+1}^{(m)}-t_m\to -\infty$. 
 In this case the matrices corresponding to the corresponding points in $S^2{V^*}^{\geq 0}$ are the following, using the base bilinear form $X_0$:
 $$k_m\cdot \Diag(e^{2\lambda_n^{(m)}-t_m}, \cdots, e^{2\lambda_1^{(m)}-t_m}).$$

This converges to $k_\infty \Diag(0,\cdots, ,0,e^{\lambda_{n-k}^\infty},\cdots, e^{\lambda_{1}^\infty})$, which is equal to $\phi(h_\infty)$ since $\phi$ is equivariant and $h_\infty=k_\infty\cdot\exp(\Diag(0,\cdots,0,\lambda_{n-k}^\infty,\cdots, \lambda_1^\infty))\cdot h_{F_k}$. This concludes the proof of the continuity.

 \medskip

 Finally note that $\phi$ is proper. Indeed any compact set in $S^2{V^*}^{\geq 0}$ is contained for some $C>0$ in the set of elements $Y$ such that $-C\leq \mathfrak{s}(X_0,Y)\leq C$. The pre-image of this set is compact by $(ii)$, since the set of horofunctions $h$ such that $-C-\log(n)/2 \leq h(X_0)\leq C+\log(n)/2$ is compact. 
 Hence $\phi$ is a homeomorphism.

\end{proof}

The orbit $\mathcal{O}_{F_\rho}$ for the Finsler distance $d_{\iota\omega_1}$ has the proper-on-bisector property, and the orbit $\mathcal{O}_1$ satisfies a similar property for the Selberg invariant.

\begin{lem}
    \label{lem:Proper Horofunctions Selberg}
    Let $Y_1,Y_2\in \mathcal{O}_1$ be in generic position, meaning that the lines $\ker Y_1$ and $\ker Y_2$ are distinct. Then for any $O\in \mathcal{X}^*$, there exist $A,B>0$ such that for all $X\in \mathcal{X}$ such that $\mathfrak{s}(X,Y_1)=\mathfrak{s}(X,Y_2)$:
    $$\mathfrak{s}(X,Y_1)\geq A\mathfrak{s}(X,O)-B. $$
\end{lem}

\begin{proof}
Since the kernels of $Y_1,Y_2$ are disjoint, then $Y_1+Y_2$ is positive, and for some $\epsilon>0$, for all $X\in \mathcal{X}$:
$$\Tr(XY_1)+\Tr(XY_2)\geq \epsilon \Tr(XO).$$

Hence if we assume $\Tr(XY_1)=\Tr(XY_2)$, then :
$$\mathfrak{s}(X,Y_1)\geq 1/2\log(\epsilon\Tr(XO)/2n)\geq \mathfrak{s}(X,O)+\frac{1}{2}\log(\epsilon/2).$$

This concludes the proof.
\end{proof}

\subsection{Properties of ideal Selberg--Poisson--Voronoi tessellations.}

All the orbits $\mathcal{O}_{F_k}\simeq \mathcal{O}_k$ for $1\leq k\leq n-1$ carry an invariant measures.

\begin{prop}
Each stratum $\mathcal{O}_{F_k}$ carries a unique $\SL(V)$-invariant measure $\mu_k$, up to scaling. 
In particular $\mu_{1}$ is equal to the pushforward by the homeomorphism $\phi$ from Proposition \ref{prop:Identification Selberg and Finler} of the measure $\mu_{N_{\iota \omega_1}}$ on $\mathcal{O}_{F_{1}}=\mathcal{O}_{F_\rho}$ obtained previously from a limiting process, up to a scalar.
\end{prop}

\begin{proof}

For all $1\leq k<n$, the face $F_k$ satisfies the hypothesis of Proposition \ref{prop: Existence measure}. 
The measure $\mu_k$ is then the unique invariant measure. In the case $k=1$ it must coincide with the $G$-invariant measure $\mu_{N_{\iota \omega_1}}$.

\end{proof}

We first note the following result, that is an analog of the second part in Theorem \ref{thm: Convergence measure}. Let $\mathfrak{C}^x_{-\infty,R}$ for $x\in \mathcal{X}$ and $R\in \R$ be the cylinder of elements $Y_0\in\mathcal{O}_{F_1}$ such that $h_{Y_0}(X)\leq R$. 

\begin{lem}
\label{lem: Selberg Control volume cones}
There exist a constant $\frac{1}{n^{2\rho_{max}}}<C_{1}<n^{{2\rho_{max}}}$ where $2\rho_{\max}=n^2-n$ such that for all $X\in \mathcal{X}$ and $R\in \R$:
$$\mu_{1}(\mathfrak{C}^X_{-\infty,R})=C_{1}e^{2\rho_{\max} R}.$$
The measures $\mu_k$ also satisfy $\mu_{k}(\mathfrak{C}^X_{-\infty,R})=C_ke^{\delta_k R}$ for some $C_k,\delta_k>0$.
\end{lem}

The second part of this Lemma is not used in what follows.

\begin{proof}
 Let us consider the cylinders $C^x_{-\infty,R}$ defined previously as the sets of elements $Y_0$ such that $h_{\phi^{-1}(Y_0)}(X)\leq R$. By the second part of Theorem \ref{thm: Convergence measure} we have that $\mu_{1}(C^x_{-\infty,R})=e^{2\rho_{\max}R}$, where $2\rho_{\max}$ is the evaluation of the supremum of the sum positive roots $2\rho$ on the unit ball of $N_{\iota\omega_1}$. Here $2\rho_{\max}=n^2-n$, as the sum of positive roots on $\mathfrak{a}$ is the linear form:
 $$2\rho:\Diag(\lambda_1,\cdots, \lambda_n)\mapsto \sum_{i=1}^n\lambda_i(n+1-2i).$$
 The unit ball of $N_{\iota \omega_1}$ is the set of elements $\Diag(\lambda_1,\cdots, \lambda_n)$ with $\lambda_i\geq -1$ for all $i$ and $\lambda_1+\cdots +\lambda_n=0$. The maximum is reached for $\Diag(n-1,-1,,\cdots,-1)\in F_{1}$.
\medskip

Finally we use the inequality of Proposition \ref{prop:Identification Selberg and Finler} to see that :
$$C^x_{-\infty,
R-\log(n)} \subset \mathfrak{C}^x_{-\infty,R}\subset C^x_{-\infty,
R+\log(n)} .$$

Hence we obtain a constant $C_{1}$ for $R=0$ that satisfies the desired inequality. The inequality for all $R$ comes from the homogeneity of the measure from Theorem \ref{thm: Convergence measure}, and the fact that for $\lambda\in \R$, $\lambda_*h_Y(X)=h_Y(X)+\lambda$.

\medskip

For the other measures $\mu_k$, we note that $\mu_k(\mathfrak{C}^x_{0,1})$ is finite as $\mathfrak{C}^X_{0,1}$ is compact. Moreover by the uniqueness of the invariant measure, there exist $c\in \R$ such that $\lambda_*\cdot \mu_k= e^{\delta_k^{(\lambda)}}\mu_k$ for all $\lambda\in \R$ acting on $\mathcal{O}_k$ via $\lambda\cdot h=h(\cdot)+\lambda$, for some $\delta_k^{(\lambda)}$. Moreover $\lambda\mapsto \delta_k^{(\lambda)}$ is continuous and $\delta_k^{(\lambda+\lambda')}=\delta_k^{(\lambda)}+\delta_k^{(\lambda')}$ for all $\lambda,\lambda'\in \R$, $m\in \N_{>0}$, hence $\delta^{(\lambda)}_k=\lambda \delta_k$ for all $\lambda$. In particular $\mu_k(\mathfrak{C}^x_{-\infty,0})$ is finite and $\mu_k(\mathfrak{C}^x_{-\infty,R})=e^{\delta_kR}\mu_k(\mathfrak{C}^x_{-\infty,0})$.

\end{proof}

The proof of the following lemma is identical to the proof of Lemma \ref{lem:3 conditions}, using Lemma \ref{lem: Selberg Control volume cones} :

\begin{lem}
\label{lem:3 conditionsbis}
    Let $Y_0\in \mathcal{O}_{1}$ be a horofunction and let $E\subset \mathcal{X}$ be a subset. Let $\Pi_0\cup \lbrace Y_0 \rbrace$, where $\Pi_0$ is a Poisson point process on $\mathcal{O}_1$ with measure intensity $\mu_1$. We consider the Voronoi tessellation $(\mathcal{C}_Y)_{Y\in \Pi\cup \lbrace Y_0\rbrace}$ of $\mathcal{X}$ associated to $\Pi_0\cup \lbrace Y_0\rbrace$ for the Selberg invariant \footnote{The functions we compare to build the diagram are $\mathfrak{s}(\cdot, Y)$ for $Y\in \Pi_0\cup \lbrace Y_0\rbrace$}.
    \begin{itemize} 
    \item[(i)] Suppose that $\mathfrak{s}(\cdot,Y_0)$ has linear growth on $E$ , meaning that for some $A,B>0$, for all $X\in E$:
$$\mathfrak{s}(X,Y_0)\geq A \mathfrak{s}(X,O)-B.$$

Then almost surely $C_{Y_0}\cap E$ is bounded.

\item[(ii)] Suppose that $E$ is unbounded. If $h$ is not proper and bounded from below on $E$, then almost surely $C_{Y_0}\cap E$ is unbounded.

\item[(iii)]  Suppose that $-h$ has linear growth on $E$,  meaning that for some $A,B>0$, for all $X\in E$:
$$\mathfrak{s}(X,Y_0)\leq -A\mathfrak{s}(X,O)+B.$$

Suppose also that the $1$-neighborhood $E_1$ of $E$ has subexponential volume growth, meaning that the volume of $E_1 \cap B(o,t)$ is $o(e^{\epsilon t})$ for all $\epsilon>0$. Then almost surely $E\setminus C_{Y_0}$ is bounded.
\end{itemize}

\end{lem}

We prove the following results that makes the ideal Selberg--Poisson--Voronoi tessellations similar to ideal Poisson--Voronoi tessalations in the hyperbolic spaces $\HH^n$: every cell is convex, admit a unique point at infnity that we call its center and every geodesic from the center of a cell intersects the cell in a non-trivial interval.

\begin{thm}
\label{thm:Selberg Main}
    Let $\Pi_0$ be a Poisson point process on $\mathcal{O}_{F_{1}}$ with intensity $\mu_{1}$, and let $(\mathcal{C}_Y)_{Y\in \Pi_0}$ be the corresponding Voronoi tessellation of $\mathcal{X}$. Almost surely the following holds:
    
    \begin{itemize}
        \item[(i)] For every $Y_0\in \Pi_0$ the cell $\mathcal{C}_{Y_0}$ is projectively convex, has locally finitely many sides and its closure contains a single point $P_{Y_0}\in\partial \mathcal{X}$, which is the unique point on which the linear form on $S^2V$ associated to $Y_0$ vanishes.
        \item[(ii)] For every distinct $Y_1,Y_2\in \Pi_0$ the side $\mathcal{C}_{Y_1}\cap\mathcal{C}_{Y_2}\subset \mathcal{X}$ is bounded.
        \item[(iii)] For every projective line $L$ in $\mathbb{P}(S^2V)$ passing through $P_{Y_0}$ and some $[X]$ for $X\in \mathcal{X}$, the intersection of $L$ with $\mathcal{C}_{Y_0}$ contains a non-trivial segment $[X',P_{Y_0})$.
    \end{itemize}

\end{thm}
We emphasize here that the compactification of the symmetric space $\partial \mathcal{X}$ for which we look at the cells is not the same as the horofunction compactification $\partial \mathcal{X}^*$, which is the quotient of the space of functions $\mathcal{H}(\mathcal{X})$ used to construct the process. This subtlety appears only in higher rank, and because the opposition involution is not trivial.
\medskip

\begin{proof}[{Proof of Theorem~\ref{thm:Selberg Main}}]

We use again Proposition \ref{prop:Mecke} to say that one can consider a fixed horofunction $Y_0\in \mathcal{O}_{1}$ in order to prove $(i)$ and $(ii)$ or two fixed opposite horofunctions $Y_1,Y_2\in \mathcal{O}_{1}$ in order to prove $(ii)$ and check that for a Poisson point process $\Pi_0'$ on $\mathcal{O}_{1}$ with intensity $\mu_1$ the Voronoi tessellation $(\mathcal{C}_h)_{Y\in \Pi'_0\cup \lbrace Y_0\rbrace }$ and respectively $(\mathcal{C}_Y)_{Y\in \Pi'_0\cup \lbrace Y_1,Y_2\rbrace }$ satisfy these properties almost surely.

Let us prove $(i)$. The cell $\mathcal{C}_{Y_0}$ is the set of $X\in \mathcal{X}$ such that for all $Y\in \Pi_0'$ one has $\mathfrak{s}(X,Y_0)\leq \mathfrak{s}(X,Y)$. This inequality is equivalent to $\Tr(X(Y_0'-Y_0))\geq 0$, hence it is the intersection of $\mathcal{X}$ with linear half spaces in $S^2V$, meaning it is corresponds to a convex subset in $\mathbb{P}(S^2V)$. The fact that it has locally finitely many sides comes from the fact that for all $X\in \mathcal{X}$ and $R>0$, almost surely $\lbrace Y\in \Pi_0'\mid \mathfrak{s}(X,Y)\leq R\rbrace$ is finite, as in Proposition \ref{prop: Tessellation}.

Given $[Y_0]\in \mathcal{O}_{1}$, if we fix some basis so that its associated bilinear form is given by the following symmetric matrix $Y_0=\Diag(0,1,\cdots, 1)$. Now let us define $[X_\infty]=P_{Y_0}$. This symmetric matrix span the unique line of semi-positive elements $X$ so that $\Tr(XY_0)=0$.

Now we show $P_{Y_0}$ belongs to $\overline{\mathcal{C}_{y_0}}$. Let us assume that $y_0$ corresponds to the symmetric matrix $Y_0$. Consider the sequence $(X_k)$ in $\mathcal{X}$ corresponding to the symmetric matrices $X_k=e^{-k/n}\Diag(e^{k},1,\cdots, 1)$, so that $\det(X_k)=1$. This sequence clearly converges to $P_{Y_0}$ in $\overline{\mathcal{X}}$. We have $\mathfrak{s}(X_k,Y_0)=-\frac{k}{2n}+\frac{1}{2}\log (\frac{n-1}{n})$. By applying part $(iii)$ of Lemma \ref{lem:3 conditionsbis} to the sequence $(X_k)$, we see that for all but finitely many $k\in \N$, $X_k\in \mathcal{C}_{y_0}$ and hence the limit, which is $P_{Y_0}$, belongs to $\overline{\mathcal{C}_{Y_0}}$.

Now let us show that $\overline{\mathcal{C}_{Y_0}}\cap \partial \mathcal{X}\subset\lbrace P_{Y_0}\rbrace$. We define for $\epsilon>0$ the set $E_\epsilon\subset \overline{\mathcal{X}}\subset \mathbb{P}(S^2V)$ of elements $X\in \mathcal{X}$, such that $\Tr(XY_0)/\Tr(X)\leq \epsilon$. 
First note that if $X\notin E_\epsilon$, then $\mathfrak{s}(X,Y_0)> \mathfrak{s}(X,\Id)+\frac{1}{2}\log(\epsilon)$. We can therefore apply part $(i)$ of Lemma \ref{lem:3 conditionsbis} to the complement of $E_\epsilon$.
This proves that for all $\epsilon>0$, almost surely there is a compact subset $K\subset\mathcal{X}$ such that $\mathcal{C}_{Y_0}\cap K\subset E_\epsilon$. Hence by considering rational $\epsilon$, we obtain that almost surely $\overline{\mathcal{C}_{Y_0}}\cap \partial \mathcal{X}\subset \bigcap_{\epsilon>0} \overline{E_\epsilon}$. But this intersection contains the classes of non-zero semi-positive matrices $X$ such that $\Tr(XY_0)=0$, which consist only of the line corresponding to $P_{Y_0}\in \partial \mathcal{X}$ as $Y_0$ has rank $n-1$. This concludes the proof of part $(i)$. 

Now let us show part $(ii)$. For almost every pair $Y_1,Y_2$ for the measure $\mu_1$ the corresponding flags are opposite. Using Lemma \ref{lem:3 conditionsbis} part $(i)$, and the linear properness on bisectors from Lemma
\ref{lem:Proper Horofunctions Selberg} we obtain that almost surely $\mathcal{C}_{Y_1}\cap \mathcal{C}_{Y_2}$ is bounded.

\medskip

Now for part $(iii)$ Let us assume that $X\in \mathcal{X}$ corresponds the identity symmetric basis and that $Y_0=\Diag(0,1,\cdots,1)$ in some basis. Pick a compact neighborhood $K$ of $\Id$ in the space of positive symmetric matrices and let consider one parameter subgroup :$$(g_t)=(e^{-t/n}\Diag(e^t,1,\cdots,1))_{t\in \R}.$$

We define $E$ to be the union $\bigcup_{t\in \R }g_t\cdot K$, which is also the projective cone with base $K$ and tip $P_{Y_0}$ in $\overline{\mathcal{X}}$. This cone is a union of translates of a fixed compact set along a geodesic, and along this geodesic $\mathfrak{s}(\cdot,Y_0)$is linearly descreasing. Hence this cone satisfies the hypothesis of part $(iii)$ in Lemma \ref{lem:3 conditionsbis}. Hence almost surely $E\setminus (\mathcal{C}_{y_0}\cap E)$ is relatively compact. Hence for all $X'\in K$, the intersection of the segment $(P_{Y_0},X']$ with the convex $\mathcal{C}_{y_0}$ is a convex of the form $(P_{Y_0},X'']$. We now apply this result to a countable familly of compact sets $K$ covering $\mathcal{X}$ in order to obtain this result amost surely for any $X'\in \mathcal{X}$. 

\end{proof}

\begin{figure}[htbp!]
\centering

\begin{tikzpicture}[scale=.85,line cap=round,line join=round]
\draw[olive, fill=olive!20, thick]
    (-1.4,-1.4) rectangle (1.4,1.4);
\draw (-2.2,0)--(3.1,0);
\draw (0,-2.1)--(0,2.4);
\draw (-2,-2)--(2.45,2.45);
\draw (-2,2)--(2,-2);

\draw[red, very thick] (.45,.35)--(.45,3);
\draw[red, very thick] (.45,.35)--(3.35,.35);

\draw[figure arrow] (.45,1.95)--(.45,2.8);
\draw[figure arrow] (2.3,.35)--(3.15,.35);
\draw[figure arrow] (1.9,1.9)--(2.5,2.5);

\fill[blue] (.45,1.95) circle[radius=2pt];
\fill[blue] (2.3,.35) circle[radius=2pt];
\fill[blue] (1.9,1.9) circle[radius=2pt];

\node at (2.85,1.8) {$\mathfrak a^+$};
\node at (.35,-2.55) {$\mathrm{Sp}(4,\mathbb R)$};

\end{tikzpicture}

\begin{tikzpicture}[scale=.85,line cap=round,line join=round]
\draw[olive, fill=olive!20, thick]
    (180:1.8)--(60:1.8)--(300:1.8)--cycle;
\draw (-2.1,0)--(2.65,0);
\draw (240:2.45)--(60:2.45);
\draw (120:2.45)--(300:2.45);

\begin{scope}[shift={(.45,0)}]

\draw[red, very thick] (0:0)--(30:2.9);
\draw[red, very thick] (0,0)--(-30:2.9);

\draw[figure arrow] (30:1.95)--(30:2.7);
\draw[figure arrow] (0:1.5)--(0:2.7);
\draw[figure arrow] (-30:1.95)--(-30:2.7);

\fill[blue] (30:1.95) circle[radius=2pt];
\fill[blue] (0:1.5) circle[radius=2pt];
\fill[blue] (-30:1.95) circle[radius=2pt];

\end{scope}

\node at (2.45,1.8) {$\mathfrak a^+$};
\node at (.25,-2.55) {$\mathrm{SL}(3,\mathbb R)$};

\end{tikzpicture}

\caption{Two metrics on the Cartan subalgebras for $G=\SL(3,\R)$ and $\Sp(4,\R)$with the level line of one horofunction in $F_\rho$. For $\Sp(4,\R)$ the blue arrows point to different points in the dual horofunction compactification, whereas for $\SL(3,\R)$ these three arrows point to the same point in the dual horofunction compactification.}
\label{fig:LastFig}
\end{figure}
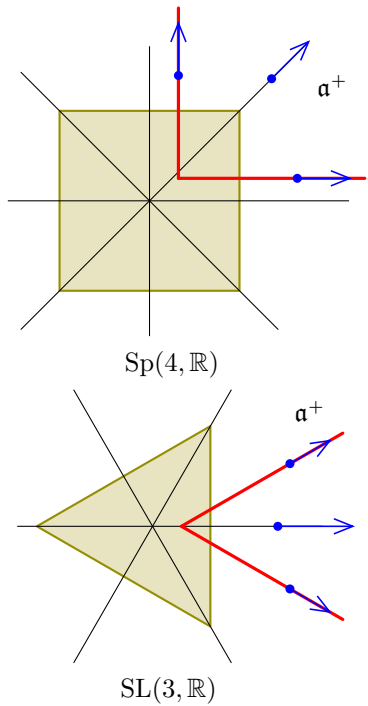

\begin{rem}
\label{rem:ClosureHoroball}
The fact that the closure of a cell for the Selberg ideal Poisson--Voronoi tessellation for $\Sp(4,\R)$ contains a unique point in the "dual"\footnote{Recall again here that the horofunction compactification for the Selberg invariant is naturally the dual Satake compactification, whose dual is the standard Satake compactification} horofunction compactification is coming from the fact that the closure of horoballs also contain a unique point in the dual horofunction compactification.

\medskip

The analog is not true for general polyhedral Finsler metrics. For instance if we consider the Finsler metric $d_{\omega_1}$ restricted to the symmetric space $\X_{\Sp} \subset \mathcal{X}$ of $\Sp(4,\R)$, the orbit $\mathcal{O}_{F_\rho}$ is different from the case $G=\SL(4,\R)$ with the same metric: a standard horofunction in $\mathcal{O}_{F_\rho}$ is given by : 
    \[ h_\omega(x) = \max_{\ell \subset \omega} b_\ell(x). \]
Here $\omega  \in \Lag(\R^4)$ and $b_\ell(x) = \Abs{\hat{\ell}}_x$.
The horoballs for $d_{\omega_1}$ contain more than one point in the compactification $\overline{\X_{\Sp}}$ induced by the standard Satake compactification: $\{ h_\omega \le 0\}$ contains all of $\mathbb{P}(\omega)$. We illustrate this phenomenon in Figure \ref{fig:LastFig}, where we show for two metrics for $G=\SL(3,\R)$ and $\Sp(4,\R)$ the level lines on a flat of a horofunction. For $\Sp(4,\R)$ the blue arrows point to different points in the dual horofunction compactification, whereas for $\SL(3,\R)$ these three arrows point to the same point in the dual horofunction compactification.

In particular the cells of the ideal Poisson--Voronoi tessellations in the symmetric space $\X_{\Sp}$ will also have several boundary points. For every $x\in \partial_N\X$ in the closure of a horoball,  we can find a sequence $(x_n)$ in the horoball converging to $x$. The ideal Poisson Voronoi tessellation associated to $\lbrace h_0\rbrace \cup \Pi_N$ will be such that infinitely many terms of our sequence belong to the cell of $h_0$ (Lemma \ref{lem:3 conditions} part (ii)), meaning that this point $x$ almost surely belongs to the boundary of this cell.
\end{rem}

\bibliographystyle{amsalpha}
\bibliography{biblio}

@misc{GuichardKasselWienhard,
 author = {Olivier Guichard and Fanny Kassel and Anna Wienhard},
 title = {Tameness of {Riemannian} locally symmetric spaces arising from {Anosov} representations},
 year = {2015},
 howpublished = {Preprint, {arXiv}:1508.04759 [math.{GT}] (2015)},
 url = {https://arxiv.org/abs/1508.04759},
 arXiv = {arXiv:1508.04759}
}

@article{KS64,
  author  = {Kochen, Simon and Stone, Charles},
  title   = {A note on the {Borel-Cantelli} lemma},
  journal = {Illinois J. Math.},
  volume  = {8},
  year    = {1964},
  pages   = {248--251},
  mrnumber = {0161355}
}

@article{Gaboriau_cost,
 author = {Gaboriau, Damien},
 title = {Cost of equivalence relations and groups},
 fjournal = {Inventiones Mathematicae},
 journal = {Invent. Math.},
 issn = {0020-9910},
 volume = {139},
 number = {1},
 pages = {41--98},
 year = {2000},
 language = {French},
 doi = {10.1007/s002229900019},
 zbMATH = {1388907},
 Zbl = {0939.28012}
}

@article{FriedlandFreitas2004pMetrics,
  author  = {Friedland, Shmuel and Freitas, Pedro J.},
  title   = {{$p$}-Metrics on {$GL(n,\mathbb{C})/U_n$} and Their
             {Busemann} Compactifications},
  journal = {Linear Algebra and its Applications},
  volume  = {376},
  pages   = {1--18},
  year    = {2004},
  doi     = {10.1016/S0024-3795(03)00661-X},
  url     = {https://doi.org/10.1016/S0024-3795(03)00661-X}
}

@article{FriedlandFreitas2004Siegel,
  author  = {Friedland, Shmuel and Freitas, Pedro J.},
  title   = {Revisiting the {Siegel} Upper Half Plane. {I}},
  journal = {Linear Algebra and its Applications},
  volume  = {376},
  pages   = {19--44},
  year    = {2004},
  doi     = {10.1016/S0024-3795(03)00662-1},
  url     = {https://doi.org/10.1016/S0024-3795(03)00662-1}
}

@mastersthesis{Schilling2013Horofunction,
  author  = {Schilling, Anna-Sofie},
  title   = {Horofunction Compactification of Finite-Dimensional
             Normed Spaces and of Symmetric Spaces},
  school  = {Ruprecht-Karls-Universit{\"a}t Heidelberg},
  type    = {Diploma thesis},
  address = {Heidelberg, Germany},
  year    = {2013},
  month   = oct,
  url     = {https://www.mathi.uni-heidelberg.de/~geodyn/teaching/Theses/DA_Anna_Schilling.pdf}
}

@article{FMW23poissonvoronoi,
  author      = {Frączyk, Mikołaj and Mellick, Sam and Wilkens, Amanda},
  title       = {{P}oisson--{V}oronoi tessellations and fixed price in higher rank},
  journal     = {Ann. of Math. (2)},
  pubstate    = {forthcoming},
  eprint      = {2307.01194},
  eprinttype  = {arxiv},
  eprintclass = {math.GT},
  note = {to appear, arXiv:2307.01194}
}

@book{LP17,
  author    = {Last, G{\"u}nter and Penrose, Mathew},
  title     = {Lectures on the {P}oisson {P}rocess},
  publisher = {Cambridge University Press},
  year      = {2017},
  series    = {Institute of Mathematical Statistics Textbooks},
  doi       = {10.1017/9781316104477}
}

@phdthesis{Bhupatiraju2019LowIntensity,
  author  = {Bhupatiraju, Sandeep},
  title   = {The {L}ow-{I}ntensity {L}imit of {B}ernoulli--{V}oronoi and
             {P}oisson--{V}oronoi {M}easures},
  school  = {Indiana University},
  address = {Bloomington, Indiana},
  year    = {2019},
  month   = may,
  url     = {https://hdl.handle.net/2022/29005}
}

@article {BudzinskiCurienPetri2025cheeger,
    AUTHOR = {Budzinski, Thomas and Curien, Nicolas and Petri, Bram},
     TITLE = {On {C}heeger constants of hyperbolic surfaces},
   JOURNAL = {Invent. Math.},
  FJOURNAL = {Inventiones Mathematicae},
    VOLUME = {242},
      YEAR = {2025},
    NUMBER = {2},
     PAGES = {511--530},
      ISSN = {0020-9910,1432-1297},
   MRCLASS = {53A05 (05C50 30F99)},
  MRNUMBER = {4965870},
       DOI = {10.1007/s00222-025-01361-w},
       URL = {https://doi.org/10.1007/s00222-025-01361-w},
}

@article {JS2024polyhedral,
    AUTHOR = {Ji, Lizhen and Schilling, Anna-Sofie},
     TITLE = {Polyhedral horofunction compactification as a polyhedral ball},
   JOURNAL = {Asian J. Math.},
  FJOURNAL = {Asian Journal of Mathematics},
    VOLUME = {29},
      YEAR = {2025},
    NUMBER = {1},
     PAGES = {1--26},
      ISSN = {1093-6106,1945-0036},
   MRCLASS = {52A21 (14M25 14M27 46B20)},
  MRNUMBER = {4899346},
       DOI = {10.4310/ajm.250429053604},
       URL = {https://doi.org/10.4310/ajm.250429053604},
}

@article {DAchilleCurienEnriquezLyonsUnel2026ipvthyperbolic,
    AUTHOR = {D'Achille, Matteo and Curien, Nicolas and Enriquez,
              Nathana\"el and Lyons, Russell and \"Unel, Meltem},
     TITLE = {Ideal {P}oisson-{V}oronoi tessellations on hyperbolic spaces},
   JOURNAL = {Ann. Probab.},
  FJOURNAL = {The Annals of Probability},
    VOLUME = {54},
      YEAR = {2026},
    NUMBER = {2},
     PAGES = {846--891},
      ISSN = {0091-1798,2168-894X},
   MRCLASS = {60G55 (60D05)},
  MRNUMBER = {5036557},
       DOI = {10.1214/25-aop1781},
       URL = {https://doi.org/10.1214/25-aop1781},
}

@article{DAchille2024ipvtbeyondhyperbolic,
    author = {Matteo D'Achille},
    title = {Ideal {P}oisson–{V}oronoi tessellations beyond hyperbolic spaces},
    volume = {31},
    journal = {Electronic Communications in Probability},
    publisher = {Institute of Mathematical Statistics and Bernoulli Society},
    pages = {1--12},
    year = {2026},
    doi = {10.1214/26-ECP815},
    URL = {https://doi.org/10.1214/26-ECP815}
}

@article {Planche1995finsler,
    AUTHOR = {Planche, Pierre},
     TITLE = {Structures de {F}insler invariantes sur les espaces
              sym\'etriques},
   JOURNAL = {C. R. Acad. Sci. Paris S\'er. I Math.},
  FJOURNAL = {Comptes Rendus de l'Acad\'emie des Sciences. S\'erie I.
              Math\'ematique},
    VOLUME = {321},
      YEAR = {1995},
    NUMBER = {11},
     PAGES = {1455--1458},
      ISSN = {0764-4442},
   MRCLASS = {53C60 (53C35)},
  MRNUMBER = {1366100},
MRREVIEWER = {O.\ Kowalski},
}

@misc{haettel2018horofunctioncompactificationssymmetricspaces,
      title={{H}orofunction {C}ompactifications of {S}ymmetric {S}paces}, 
      author={Thomas Haettel and Anna-Sofie Schilling and Cormac Walsh and Anna Wienhard},
      year={2018},
      eprint={1705.05026},
      archivePrefix={arXiv},
      primaryClass={math.DG},
      url={https://arxiv.org/abs/1705.05026}, 
}

@article{davaloriestenberg2026dirichletsanosov,
    url = {https://doi.org/10.1515/crelle-2026-0077},
    title = {Dirichlet domains for {A}nosov subgroups},
    author = {Colin Davalo and J. Maxwell Riestenberg},
    journal = {Journal für die reine und angewandte Mathematik (Crelles Journal)},
    doi = {doi:10.1515/crelle-2026-0077},
    year = {2026},
    lastchecked = {2026-10-12},
    note = {to appear, arXiv:2402.06408}
}

@article {KapovichLeebPorti2018dynamicsonflagmanifolds,
    AUTHOR = {Kapovich, Michael and Leeb, Bernhard and Porti, Joan},
     TITLE = {Dynamics on flag manifolds: domains of proper discontinuity
              and cocompactness},
   JOURNAL = {Geom. Topol.},
  FJOURNAL = {Geometry \& Topology},
    VOLUME = {22},
      YEAR = {2018},
    NUMBER = {1},
     PAGES = {157--234},
      ISSN = {1465-3060,1364-0380},
   MRCLASS = {53C35 (22E40 37B05 51E24)},
  MRNUMBER = {3720343},
MRREVIEWER = {Gabriele\ Link},
       DOI = {10.2140/gt.2018.22.157},
       URL = {https://doi.org/10.2140/gt.2018.22.157},
}

@article {KapovichLeeb2018finsler,
    AUTHOR = {Kapovich, Michael and Leeb, Bernhard},
     TITLE = {Finsler bordifications of symmetric and certain locally
              symmetric spaces},
   JOURNAL = {Geom. Topol.},
  FJOURNAL = {Geometry \& Topology},
    VOLUME = {22},
      YEAR = {2018},
    NUMBER = {5},
     PAGES = {2533--2646},
      ISSN = {1465-3060,1364-0380},
   MRCLASS = {20F65 (22E40 51E24 53B40 53C35)},
  MRNUMBER = {3811766},
MRREVIEWER = {Loreno\ Heer},
       DOI = {10.2140/gt.2018.22.2533},
       URL = {https://doi.org/10.2140/gt.2018.22.2533},
}

@article {BenjaminiSchramm2001percolationsinH2,
    AUTHOR = {Benjamini, Itai and Schramm, Oded},
     TITLE = {Percolation in the hyperbolic plane},
   JOURNAL = {J. Amer. Math. Soc.},
  FJOURNAL = {Journal of the American Mathematical Society},
    VOLUME = {14},
      YEAR = {2001},
    NUMBER = {2},
     PAGES = {487--507},
      ISSN = {0894-0347,1088-6834},
   MRCLASS = {82B43 (60D05 60K35)},
  MRNUMBER = {1815220},
MRREVIEWER = {Olle\ H\"aggstr\"om},
       DOI = {10.1090/S0894-0347-00-00362-3},
       URL = {https://doi.org/10.1090/S0894-0347-00-00362-3},
}

@book {Eberlein1996book,
    AUTHOR = {Eberlein, Patrick B.},
     TITLE = {Geometry of nonpositively curved manifolds},
    SERIES = {Chicago Lectures in Mathematics},
 PUBLISHER = {University of Chicago Press, Chicago, IL},
      YEAR = {1996},
     PAGES = {vii+449},
      ISBN = {0-226-18197-9; 0-226-18198-7},
   MRCLASS = {53-02 (53C20 53C21 53C35)},
  MRNUMBER = {1441541},
MRREVIEWER = {Mar\'ia\ J.\ Druetta},
}

@article {Walsh2007normed,
    AUTHOR = {Walsh, Cormac},
     TITLE = {The horofunction boundary of finite-dimensional normed spaces},
   JOURNAL = {Math. Proc. Cambridge Philos. Soc.},
  FJOURNAL = {Mathematical Proceedings of the Cambridge Philosophical
              Society},
    VOLUME = {142},
      YEAR = {2007},
    NUMBER = {3},
     PAGES = {497--507},
      ISSN = {0305-0041,1469-8064},
   MRCLASS = {53C70 (46B20 52A41)},
  MRNUMBER = {2329698},
MRREVIEWER = {Claire\ Anantharaman-Delaroche},
       DOI = {10.1017/S0305004107000096},
       URL = {https://doi.org/10.1017/S0305004107000096},
}

@article{EberleinONeill1973visibility,
  title={Visibility manifolds},
  author={Eberlein, Patrick and O’Neill, Barrett},
  journal={Pacific Journal of Mathematics},
  volume={46},
  number={1},
  pages={45--109},
  year={1973},
  publisher={Mathematical Sciences Publishers}
}

@book{onishchikvinberg2012liegroups,
  title={Lie groups and algebraic groups},
  author={Onishchik, Arkadij L and Vinberg, Ernest B},
  year={2012},
  publisher={Springer Science \& Business Media}
}

@book{helgason1979differential,
  title={Differential geometry, {L}ie groups, and symmetric spaces},
  author={Helgason, Sigurdur},
  volume={80},
  year={1979},
  publisher={Academic press}
}

@incollection {Kapovich2023geometricalgorithms,
    AUTHOR = {Kapovich, Michael},
     TITLE = {Geometric algorithms for discreteness and faithfulness},
 BOOKTITLE = {Computational {A}spects of {D}iscrete {S}ubgroups of {L}ie
              {G}roups},
    SERIES = {Contemp. Math.},
    VOLUME = {783},
     PAGES = {87--112},
 PUBLISHER = {Amer. Math. Soc., [Providence], RI},
      YEAR = {[2023] \copyright 2023},
      ISBN = {978-1-4704-6804-0},
   MRCLASS = {20-08 (20F67 22E40 53C35)},
  MRNUMBER = {4556436},
MRREVIEWER = {Sami\ Douba},
       DOI = {10.1090/conm/783/15736},
       URL = {https://doi.org/10.1090/conm/783/15736},
}

@article {Du2024selbergbisectors,
    AUTHOR = {Du, Yukun},
     TITLE = {Geometry of {S}elberg's bisectors in the symmetric space
              {$SL(n,\mathbb{R})/SO(n,\mathbb{R})$}},
   JOURNAL = {J. Lond. Math. Soc. (2)},
  FJOURNAL = {Journal of the London Mathematical Society. Second Series},
    VOLUME = {110},
      YEAR = {2024},
    NUMBER = {4},
     PAGES = {Paper No. e12992, 36},
      ISSN = {0024-6107,1469-7750},
   MRCLASS = {20F65 (15A22 53C35)},
  MRNUMBER = {4801894},
MRREVIEWER = {Vladimir\ V.\ Gorbatsevich},
       DOI = {10.1112/jlms.12992},
       URL = {https://doi.org/10.1112/jlms.12992},
}

@misc{Du2025selbergcompleteness,
      title={{B}usemann-{S}elberg Functions and Completeness for {D}irichlet-{S}elberg domains in $SL(n,\mathbb{R})/SO(n,\mathbb{R})$}, 
      author={Yukun Du},
      year={2025},
      eprint={2412.06987},
      archivePrefix={arXiv},
      primaryClass={math.GR},
      url={https://arxiv.org/abs/2412.06987}, 
}

@incollection {Selberg1960discontinuous,
    AUTHOR = {Selberg, Atle},
     TITLE = {On discontinuous groups in higher-dimensional symmetric spaces},
 BOOKTITLE = {Contributions to function theory ({I}nternat. {C}olloq.
              {F}unction {T}heory, {B}ombay, 1960)},
     PAGES = {147--164},
 PUBLISHER = {Tata Inst. Fund. Res., Bombay},
      YEAR = {1960},
   MRCLASS = {20.65 (22.70)},
  MRNUMBER = {130324},
MRREVIEWER = {E.\ Grosswald},
}

@article{BFGM07,
  author  = {Bader, Uri and Furman, Alex and Gelander, Tsachik and Monod, Nicolas},
  title   = {Property ({T}) and rigidity for actions on {B}anach spaces},
  journal = {Acta Mathematica},
  volume  = {198},
  number  = {1},
  pages   = {57--105},
  year    = {2007},
  doi     = {10.1007/s11511-007-0013-0}
}

\end{document}